\documentclass{amsart}

\usepackage[T1]{fontenc}
\usepackage[utf8]{inputenc}
\usepackage{eucal}
\usepackage[final,factor=1100,stretch=10,shrink=10]{microtype}

\SetProtrusion{encoding={*},family={bch},series={*},size={6,7}}
{1={ ,750},2={ ,500},3={ ,500},4={ ,500},5={ ,500},
6={ ,500},7={ ,600},8={ ,500},9={ ,500},0={ ,500}}

\SetExtraKerning[unit=space]
{encoding={*}, family={bch}, series={*}, size={footnotesize,small,normalsize}}
{\textendash={400,400}, 
"28={ ,150}, 
"29={150, }, 
\textquotedblleft={ ,150}, 
\textquotedblright={150, }} 

\usepackage{tikz}
\usepgflibrary{arrows}
\usetikzlibrary{matrix}
\usetikzlibrary{decorations.pathreplacing}

\usepackage{color}
\usepackage{boxedminipage}

\usepackage{amsmath}
\usepackage{amssymb}
\usepackage{amsthm}

\usepackage[colorlinks, allcolors=.]{hyperref} 

\theoremstyle{plain} 
\newtheorem{theorem}{Theorem}[section]
\newtheorem*{theorem*}{Theorem}

\newtheorem{lemma}[theorem]{Lemma}
\newtheorem{corollary}[theorem]{Corollary}          
\newtheorem{proposition}[theorem]{Proposition}              

\theoremstyle{definition}  
\newtheorem{definition}[theorem]{Definition}
\newtheorem{notation}[theorem]{Notation}
\newtheorem{construction}[theorem]{Construction}
\newtheorem*{axiom*}{Axiom}
\newtheorem*{basic-data*}{Basic Data}
\newtheorem*{notation*}{Notation}

\theoremstyle{remark}
\newtheorem{remark}[theorem]{Remark}
\newtheorem*{remark*}{Remark}
\newtheorem{example}[theorem]{Example}
\newtheorem{conjecture}[theorem]{Conjecture}

\numberwithin{equation}{theorem}

\usepackage{enumerate}

\newcommand{\A}{\mathcal{A}}
\newcommand{\B}{\mathcal{B}}
\newcommand{\C}{\mathcal{C}}
\newcommand{\D}{\mathcal{D}}
\newcommand{\E}{\mathcal{E}}
\newcommand{\G}{\mathbb{G}}

\newcommand{\LH}{\mathrm{L}^{\!\!\mathrm{H}}}
\newcommand{\LL}{\mathrm{L}}
\newcommand{\M}{\mathcal{M}}

\newcommand{\NH}{\mathrm{N}^{\mathrm{H}}}

\newcommand{\Q}{\mathcal{Q}}
\newcommand{\R}{\mathcal{R}}

\newcommand{\Z}{\mathbb{Z}}

\newcommand{\coloneq}{\mathrel{\mathop:}=}

\DeclareMathOperator{\Aut}{Aut}
\DeclareMathOperator{\cat}{Cat}
\DeclareMathOperator{\CSS}{CSS}

\DeclareMathOperator{\Fun}{Fun}

\DeclareMathOperator{\uHom}{\underline{Hom}}
\DeclareMathOperator{\id}{id}

\DeclareMathOperator{\Seg}{Seg}
\DeclareMathOperator{\set}{Set}

\DeclareMathOperator{\thy}{Thy}

\DeclareMathOperator{\map}{Map}
\DeclareMathOperator{\pre}{\mathcal{P}}

\DeclareMathOperator{\gaunt}{Gaunt}

\DeclareMathOperator{\corr}{Corr}
\DeclareMathOperator{\precat}{PreCat}

\DeclareMathOperator*{\colim}{colim}

\newcommand{\op}{\mathrm{op}}
\newcommand{\ob}{\mathrm{ob}}

\definecolor{CSPcolor}{rgb}{0.3,0.3,0.75}	
\definecolor{CBcolor}{rgb}{0.5,0.0,0.5}		
\definecolor{Removecolor}{rgb}{0.7,0.3,0.0}		

\title{On the unicity of the theory of higher categories}
\author{Clark Barwick and Christopher Schommer-Pries}

\thanks{The second author was supported by NSF fellowship DMS-0902808.}

\begin{document}

\begin{abstract} We axiomatise the theory of $(\infty,n)$-categories. We prove that the space of theories of $(\infty,n)$-categories is a $B(\Z/2)^n$. We prove that Rezk's complete Segal $\Theta_n$ spaces, Simpson and Tamsamani's Segal $n$-categories, the first author's $n$-fold complete Segal spaces, Kan and the first author's $n$-relative categories, and complete Segal space objects in any model of $(\infty, n-1)$-categories all satisfy our axioms. Consequently, these theories are all equivalent in a manner that is unique up to the action of $(\Z/2)^n$.
\end{abstract}


\maketitle

\tableofcontents

\section{Introduction}

Any model for the theory of $(\infty,n)$-categories must, at a minimum, form an \emph{$\infty$-category} $\C$.\footnote{By this we mean a \emph{quasicategory} in the sense of Michael Boardman and Rainer Vogt, André Joyal, and Jacob Lurie; we freely use the language and technology of Lurie's books \cite{HTT,Lurie-HA}.} Such an $\infty$-category must contain the \emph{gaunt $n$-categories} (Definition \ref{def:gaunt}) as a full subcategory -- these are strict $n$-categories with no nontrivial isomorphisms at any level. In particular, $\C$ contains the \emph{$k$-cells} $C_k$ for $0\leq k\leq n$. These are the strict $n$-categories with the universal property that the set of $k$-morphisms of a strict $n$-category $D$ is the set of functors $C_k\to D$.

In order for the objects of $\C$ to be considered as $(\infty,n)$-categories, they must be built from cells together with composition operations. These composition operations are governed by pasting diagrams, possibly quite general ones. The largest conceivable collection of these pasting diagrams is the class of gaunt $n$-categories itself, so we encode these properties via a pair of axioms:
\begin{enumerate}[{C.}1]
\item Strong generation. Every object of $\C$ is the canonical colimit (in the $\infty$-categorical sense) of the diagram of all the gaunt $n$-categories that map to it.
\item Weak generation. Every object of $\C$ admits a `cell decomposition' -- i.e., it is \emph{some} colimit of cells (again in the $\infty$-categorical sense).
\end{enumerate}

\emph{Correspondences} are configurations of $(\infty,n)$-categories parametrized by cells. These must be well-behaved:
\begin{enumerate}[{C.}3]
\item Internal Homs for correspondences. For any $0\leq k\leq n$, the $\infty$-category of objects of $\C$ over the $k$-cell $C_k$ has internal Homs.
\end{enumerate}

Our pasting diagrams are constructed by repeatedly applying certain gluing operations. These operations must remain colimits when viewed in $\C$. To ensure this, we identify (\ref{Note:FundPushouts}) a finite list $S_{00}$ of pushouts of gaunt $n$-categories, and we introduce another axiom:
\begin{enumerate}[{C.}4]
\item Fundamental pushouts. The image of the diagrams $S_{00}$ are pushouts in $\C$.
\end{enumerate}

Finally, the $\infty$-category $\C$ must be minimal (in a rather weak sense) with these features:
\begin{enumerate}[{C.}5]
\item Versality. If $\D$ is an $\infty$-category that contains the gaunt $n$-categories as a full subcategory and satisfies the axioms above, then there is a left adjoint $K\colon\C\to\D$ and a natural transformation $\eta$ between the restriction of $K$ to the gaunt $n$-categories and the inclusion of the gaunt $n$-categories into $\D$ such that $\eta$ is an equivalence on cells.
\end{enumerate}
An $\infty$-category $\C$ that contains a copy of $\gaunt_n$ and satisfies these axioms is called a \emph{theory of $(\infty,n)$-categories}. In this paper
\begin{itemize}
\item we prove that there is a \emph{unique} theory of $(\infty,n)$-categories, up to equivalence;
\item we prove that, up to the formation of opposites, there is a contractible space of equivalences of the theory of $(\infty,n)$-categories; and
\item we prove that all the best-known purported models of $(\infty,n)$-categories satisfy these axioms and are therefore equivalent in an essentially unique manner.
\end{itemize}

In more detail, the main theorem is:
\begin{theorem}[Unicity] The moduli space $\thy_{(\infty,n)}$ of theories of $(\infty,n)$-categories is a $B(\Z/2)^n$. 
\end{theorem}

Any theory of $(\infty,n)$-categories has internal Homs, and is thus canonically enriched in itself. Results of David Gepner and Rune Haugseng \cite{MR3345192,MR3402334} show that categories enriched in $(\infty,n)$-categories are a model of $(\infty,n+1)$-categories. Thus the unicity theorem for the $\infty$-category of $(\infty,n)$-categories implies the unicity of the $(\infty,n+1)$-category of $(\infty,n)$-categories.

\subsection*{Other axiomatizations of higher categories}
Carlos Simpson \cite[Conjectures 2 and 3]{Simpson01} conjectured a similar unicity result for the theory of $n$-categories.
Simpson suggests ten axioms (Properties 1--10), which are extremely different from those here.
Nevertheless, the kind of unicity that Simpson proposed (and even the idea that one could axiomatize the homotopy theory of higher categories itself) was of course a direct inspiration for our work here.

Bertrand To\"en \cite{Toen} later proved a Unicity Theorem of the kind above for the theory of $(\infty,1)$-categories.
His framework provides seven axioms.
The basic data is that of a homotopy theory $\M$ containing a cosimiplicial \emph{interval object}.
A subset of his axioms (A2, A6, A7) imply that this interval object can be used to define a right adjoint $N$ from $\M$ to the homotopy theory $\CSS$ of complete Segal spaces.
The axioms (A6) is that $N$ is conservative, and the rest of the axioms are used to show that the left adjoint of $N$ is fully faithful, which shows this is an equivalence.
Since $\CSS$ satisfies our axioms for $n=1$ (see \ref{thm:CSSnisathy}), one knows a posteriori that To\"en's axioms and ours specify the same homotopy theory.

One may ask whether this is clear a priori.
It seems not: there doesn't seem to be any simple mechanism by which one could translate To\"en's axioms into ours or vice versa.
While our axioms do have the one point in common that we each require the existence of a well-behaved (presentable) homotopy theory and internal Homs, the similarities end here.
Even the basic data we are axiomatizing is not the same: while To\"en's axioms are precisely adapted to the comparison with $\CSS$, our axioms remain agnostic about the ``shapes'' of the basic objects that are used to generate models of higher categories.

\subsection*{Plan}
This paper is divided into three parts. The first part concerns various aspects of strict $n$-category theory, most particularly including the theory of gaunt $n$-categories. It does not make use of any $\infty$-category theory.

The second part concerns the axiomatization. We first introduce our axioms. Then we show that $\thy_{(\infty,n)}$ is nonempty by explicitly constructing a theory of $(\infty,n)$-categories that satisfies our axioms. We show that any other theory of $(\infty,n)$-categories is equivalent to this given theory -- $\thy_{(\infty,n)}$ is connected. We then compute the based loopspace at the point we constructed -- that is, the space of autoequivalences of the model of $(\infty,n)$-categories. There are $n$ obvious involutions, which are given by forming the opposite at each categorical level; it turns out that up to a contractible space of identifications, these are \emph{all} of the autoequivalences. 


In the third and final part of this paper we prove that most of the purported models of $(\infty,n)$-categories in the literature satisfy our axioms. These include:
\begin{enumerate}
\item Charles Rezk's complete Segal $\Theta_n$-spaces,
\item the $n$-fold complete Segal spaces of the first-named author,
\item André Hirschowitz and Simpson's Segal $n$-categories,
\item the $n$-relative categories of the first-named author and Dan Kan,
\item categories enriched in any internal model category whose underlying homotopy theory is a homotopy theory of $(\infty,n)$-categories,
\item when $n=1$, Boardman and Vogt's quasicategories,
\item when $n=1$, Lurie's marked simplicial sets, and
\item when $n=2$, Lurie's scaled simplicial sets,
\end{enumerate}
Consequently they are all equivalent to our model, in a manner that is unique up to the formation of the opposites at the various levels. This also confirms that any \emph{model categories} that these $\infty$-categories underlie are Quillen equivalent. In fact, Quillen equivalences between model categories of $(\infty,n)$-categories are easily recognized (Proposition \ref{prop:cellsdetectquillenequiv}): a Quillen adjunction between two model categories of $(\infty,n)$-categories is a Quillen equivalence if and only if it preserves the cells up to weak equivalence. This implies that many of the known Quillen functors relating various models are in fact Quillen equivalences. 

\subsection*{Acknowledgements}
We are grateful to Charles Rezk, who noticed an error in a very early version of this paper.
We also thank the referee, who made a number of helpful suggestions to make our arguments more legible.


\part{Preliminaries on gaunt $n$-categories}


\section{Strict $n$-categories} \label{Sect:strict}

\begin{definition} A small \emph{strict $0$-category} is a set.
Proceeding recursively, for any positive integer $n$, a small \emph{strict $n$-category} is a small category enriched in small $(n-1)$-categories.
A {\em functor} between strict $n$-categories will mean an enriched functor.
We denote by $\cat_{n}$ the category of small strict $n$-categories and functors.  
\end{definition}

\begin{remark} For the rest of this paper we will hold the convention that, unless otherwise stated, all strict $n$-categories are small.
\end{remark}

\begin{definition} A set can be regarded as a $1$-category with only identity morphisms, and this defines a fully faithful functor
\[\cat_0\hookrightarrow \cat_{1}\]
that respects products.
Passing to enriched categories then yields a sequence of fully faithful functors
\[\cat_0 \hookrightarrow \cat_{1}\hookrightarrow \cdots \hookrightarrow\cat_{(n-1)} \hookrightarrow \cat_{n}\hookrightarrow \cdots.\]

We will tacitly treat these functor as \emph{inclusions} in order to treat strict $k$-categories as examples of strict $n$-categories when $0\leq k\leq n$.
In particular, a strict $n$-categories in the image of $\cat_0$ under this inclusions will be called {\em discrete}.
\end{definition}

\begin{remark} It is well known that the fully faithful inclusion $i:\cat_k\hookrightarrow\cat_n$ admits a right adjoint $j_k$.
The right adjoint $j_k: \cat_n \to \cat_k$ carries a strict $n$-category $C$ the maximal $k$-category $j_kC$ contained therein.
\end{remark}

\begin{example} \label{Ex:Strictncats} The following are some important examples of strict $n$-categories:
\begin{enumerate}[(\ref{Ex:Strictncats}.1)]
\item The empty $n$-category $\emptyset$ is the empty set, regarded as an $n$-category.
Later it will be convenient to write $\partial C_{0}\coloneq\emptyset$.
\item The \emph{$0$-cell} $C_0$ is the singleton set, viewed as a strict $n$-category.
This is also the terminal strict $n$-category.
\item The $1$-category $E$ is the ``walking isomorphism,'' that is, the unique contractible groupoid that contains exactly two objects.
\item The \emph{$k$-cell} $C_k$ is the strict $k$-category defined inductively as follows: the set of object of $C_k$ is the set $\{\bot,\top\}$, and one has
\begin{equation*}
\hom_{C_k}(x,y)\coloneq\begin{cases}
C_0&\textrm{if }x=y;\\
C_{k-1}&\textrm{if }x=\bot\textrm{ and }y=\top;\\
\emptyset&\textrm{otherwise.}
\end{cases}
\end{equation*}
There is a unique composition law making this a strict $k$-category, and therefore a strict $n$-category for $n\geq k$. 
\item The $(k-1)$-category $\partial C_k\coloneq j_{k-1}C_k$ can be described as the $(k-1)$-category of  ``walking parallel $(k-1)$-morphisms''.
\item A finite ordinal $S$ gives rise to a $1$-category $\Delta^S$, whose objects are elements of $S$ in which there is a unique morphism $s \to s'$ if and only if $s \leq s'$.
The simplex category of nonempty finite ordinals will be denoted $\Delta$, as usual.
\end{enumerate} 
\end{example}

\begin{notation} We may generalize the fourth example in the following manner.
Suppose $X$ a strict $n$-category.
We obtain a strict $(n+1)$-category $\sigma X$, the \emph{suspension} of $X$, as follows.
The set of objects of $\sigma X$ is the set $\{\top,\bot\}$, and one defines
\begin{equation*}
\hom_{\sigma X}(x,y)\coloneq\begin{cases}
C_0&\textrm{if }x=y;\\
X&\textrm{if }x=\bot\textrm{ and }y=\top;\\
\emptyset&\textrm{otherwise.}
\end{cases}
\end{equation*}
There is a unique composition law that makes this into a strict $n$-category.

Observe that the $k$-fold suspension of the zero cell $C_0$ is now nothing more than the $k$-cell $\sigma^k (C_0) = C_k$.
Furthermore, the suspension functor preserves both pullback and pushout squares. 
Consequently, we have an isomorphism
\[ \sigma(\emptyset)\cong C_0 \sqcup C_0 \cong \partial C_1, \]
and therefore by induction we have
\[ \sigma^k(\emptyset) \cong \sigma^{k-1}(C_0 \cup^{\emptyset} C_0) \cong C_{k-1} \cup^{\partial C_{k-1}} C_{k-1} \cong \partial C_{k}. \]
The canonical inclusion $\partial C_k \hookrightarrow C_{k-1}$ arises as the $k$-fold suspension  of the unique functor $C_0 \sqcup C_0 \to C_0$.
\end{notation}

The following proposition is well known.
\begin{proposition} \label{prop:cellsgenerateStrictNCat} The cells ($C_i$, $0 \leq i \leq n$) generate $\cat_n$ under colimits; that is, the smallest full subcategory of $\cat_n$ containing the cells and closed under colimits is all of $\cat_n$. \qed
\end{proposition}


\section{Gaunt $n$-categories}\label{sect:gaunt}

\begin{definition}\label{def:gaunt} A strict $n$-category $X$ is \emph{gaunt} if for any $1\leq k \leq n$, the $n$-category $X$ is \emph{local} with respect to the natural functor
\[\sigma^{k-1}E \to \sigma^{k-1}(C_0) = C_{k-1};\]
that is, the induced map
\begin{equation*}
\cat_n(C_{k-1},X) \to \cat_n(\sigma^{k-1}E, X)
\end{equation*}
is a bijection.
Equivalently, a strict $n$-category $X$ in gaunt just in case, for any $1\leq k\leq n$, any invertible $k$-morphism is an identity.

We write $\gaunt_n\subset\cat_n$ for the full subcategory spanned by the gaunt $n$-categories.
\end{definition}

\begin{remark} \label{rmk:maxsubisgaunt} Observe that the suspension of a gaunt $n$-category is again gaunt. Note also that if a strict $n$-category $X$ is gaunt, then for any $0\leq k\leq n$, so is the strict $k$-category $j_k X$.
\end{remark}

\begin{remark} Rezk observed \cite[\S~10]{Rezk} that the $1$-category $E$ may be exhibited in $\cat_1$ as a pushout of more elementary $n$-categories:
\begin{equation*}
E\cong\Delta^3\cup^{(\Delta^{\{0,2\}}\sqcup\Delta^{\{1,3\}})}(\Delta^0\sqcup\Delta^0).
\end{equation*}
Consequently, a strict $n$-category $X$ is gaunt if and only if for each $k \geq 0$ the following natural map is a bijection:
\begin{equation*}
\Fun(C_k, X) \to \Fun(\sigma^k(\Delta^3), X ) \times_{\Fun(\sigma^k(\Delta^{\{0,2\}}\sqcup\Delta^{\{1,3\}}), X)} \Fun( \sigma^k(\Delta^0\sqcup\Delta^0) , X).
\end{equation*} 
\end{remark}

The following is an easy consequence of the fact that $\cat_n$ is a presentable category.
\begin{proposition}\label{Cor:GauntLocalization} The inclusion
\[\gaunt_n\hookrightarrow\cat_n\]
admits a left adjoint $L^G$ that exhibits $\gaunt_n$ as a localization of $\cat_{n}$. \qed
\end{proposition}
\noindent In particular, $\gaunt_n$ is a presentable category. In fact, we can be more precise.
\begin{lemma} \label{lma:gauntisfp} The category $\gaunt_n$ is locally finitely presentable.
\begin{proof} Since $\cat_n$ is locally finitely presentable, it suffices to show that the inclusion $\gaunt_n \hookrightarrow \cat_{n}$ commutes with filtered colimits.
To this end, suppose $\Lambda$ a filtered category, and suppose $D\colon\Lambda\to\cat_n$ a diagram such that for any object $\alpha\in\Lambda$, the $n$-category $D_{\alpha}$ is gaunt.
We claim that the colimit $D=\colim_{\alpha\in\Lambda}D_{\alpha}$ (formed in $\cat_n$) is gaunt as well.
This claim now follows readily from the fact that both $C_k$ and $\sigma^k(E)$ are compact objects in $\cat_n$.
\end{proof}
\end{lemma}

\begin{remark} \label{rmk:gauntisIndoffpgaunt} The previous lemma now implies that the category $\gaunt_n$ can be identified with the category of Ind-objects of the full subcategory $\gaunt_n^{\omega}\subset\gaunt_n$ spanned by the compact objects of $\gaunt_n$.
That is \cite[Corollary 2.1.9$'$]{Makkai-Pare}, for any category $\mathcal{D}$ that admits all filtered colimits, if $\Fun^{\omega}(\gaunt_n,\mathcal{D})$ denotes the full subcategory of $\Fun(\gaunt_n,\mathcal{D})$ spanned by those functors that preserve filtered colimits, then the restriction functor
\begin{equation*}
\Fun^{\omega}(\gaunt_n,\mathcal{D})\to\Fun(\gaunt_n^{\omega},\mathcal{D})
\end{equation*}
is an equivalence.
\end{remark}

\begin{corollary} \label{cor:GauntGenByCells} Suppose $0\leq k\leq n$. Then the smallest full subcategory of $\gaunt_n$ that is closed under colimits and contains the cells $C_r$ for $r\leq k$ is $\gaunt_k$. 
\begin{proof} The inclusion of gaunt $k$-categories commutes with colimits (as it admits a right adjoint, see Rk.~\ref{rmk:maxsubisgaunt}), whence it is enough to consider the case $k=n$.
This now follows readily from Corollary \ref{Cor:GauntLocalization} and Proposition \ref{prop:cellsgenerateStrictNCat}.
\end{proof}
\end{corollary}


\section[Endomorphisms and automorphisms of the category of Gaunt n-categories]{Endomorphisms and Automorphisms of $\gaunt_n$} We now demonstrate that the full subcategory of $\Fun(\gaunt_n,\gaunt_n)$ spanned by the autoequivalences is discrete, and in particular it is the set $(\Z/2)^n$. We will also give conditions ensuring an endomorphism is an equivalence. 

\begin{definition} The \emph{globular category} $\G_n$ consists of the full subcategory of $\gaunt_n$ consisting of the $i$-cells $C_i$ for $i \leq n$.

An {\em $n$-globular set} is a presheaf of sets on $\G_n$.
The {\em $k$-cells} $X_k$ of a globular set is the set obtained by evaluating the presheaf $X$ on $C_k$.
\end{definition}

\begin{example} Any strict $n$-category $X$ gives rise to an $n$-globular set, which we also denote by $X$, such that $X_k\coloneq\cat_n(C_k,X)$. In other words, $X_k$ is the set of \emph{$k$-morphisms} of $X$.
\end{example}

\begin{remark} The globular sets considered here are sometimes called {\em reflexive globular sets}, in order to emphasize the fact that our globular category $\G_n$ includes {\em degeneracies} $C_k \to C_{k-1}$. The non-reflexive globular category $\G_n^\textrm{nr}$ consists of the subcategory of $\G_n$ with the same objects but only with the morphisms which are injective on $n$-morphisms. As a category it is generated by object $C_k$ $0 \leq k \leq n$ with morphisms
	\begin{equation*}
		s_k,t_k: C_{k-1} \to C_k
	\end{equation*}
satisfying $s_kt_{k-1} = t_kt_{k-1}$ and $s_ks_{k-1} = t_ks_{k-1}$. We will have only cursory use for non-reflexive globular sets in this paper. 
\end{remark}

\begin{remark} In this language, an alternative, noninductive definition of strict $n$-category is possible: a strict $n$-category $X$ is an $n$-globular set together with a family of operations $n \geq k\geq j$:
\[\ast_j:X_k\times_{X_{j-1}}X_k\to X_k,\]
which are associative, unital, and suitably compatible.
\end{remark}

\begin{lemma} \label{lma:uniqueauto} There is a unique natural transformation from the identity functor on $\gaunt_n$ to itself.
\begin{proof} Such a natural transformation consists of component maps (i.e., functors) $ \eta_X\colon X \to X$ for each gaunt $n$-category $X$.
We will show that $\eta_X = \id_X$ for all $X$.
The functor $\eta_X$ induces, for each $0 \leq k \leq n$, a map on sets of $k$-cells,
\[(\eta_X)_k\colon X_k \to X_k.\]
Since a functor is completely determined by the map on $k$-cells for each $k$, it is enough to show that each $(\eta_X)_k$ is the identity.
By naturality of $\eta$ it is enough to show that the single functor $\eta_{C_n} = \id_{C_n}$.

One my now show that $\eta_{C_k} = \id_{C_k}$ by inducting on $k$.
When $k=0$ the claim is obvious.
Now the inductive hypothesis asserts $\eta_{C_k}$ is a functor which restricts to the identity functor on $\partial C_k$.
There is only one functor with this property, namely $\eta_{C_k} = \id_{C_k}$.
\end{proof}  
\end{lemma}

\begin{construction} For any category $C$ enriched in a symmetric monoidal category $V$, one may of course form the opposite enriched category $X^{\op}$.
This is an involution on the category of $V$-enriched categories.

Inducting this yields a free action $\rho$ of $(\Z/2)^n$ on $\cat_n$.
We will show in a moment that in fact this action produces an equivalence between $(\Z/2)^n$ and the monoidal full subacategory of $\Fun(\gaunt_n,\gaunt_n)$ spanned by the autoequivalences.

For now, let us restrict this action: we note that $\rho$ restricts to an action on the globular category $\G_n$.
\end{construction}

\begin{proposition}\label{prop:aut_glob_cat} Every autoequivalence of $\G_n$ is isomorphic to $\rho(g)$ for some element $g\in (\Z/2)^n$.
\begin{proof} First we observe that the $k$-cell $C_k \in \G_n$ is the unique object such that, up to isomorphism, there exists precisely $k$ other objects of $\G_n$ which occur as proper retracts (namely all the cells $C_i$ with $0 \leq i < k$).
Consequently, every autoequivalence $F$ of $\G_n$ must fix the objects. 
  
Next we observe that for each $i<j$, there exists precisely one epimorphism $C_j \twoheadrightarrow C_i$.
Since epimorphisms are preserved by any equivalence of categories, this unique epimorphism is also preserved by $F$.
Similarly, for each $0\leq i < n$, there exist precisely two monomorphisms $C_i \hookrightarrow C_n$; these are either preserved by $F$ or else they are permuted.
Thus every autoequivalence determines an element $\gamma(F) \in (\Z/2)^n$ such that $\gamma(F)_i=0$ just in case the pair of monomorphisms $C_i \hookrightarrow C_n$ is preserved by $F$, and $\gamma(F)_i=1$ just in case the pair of monomorphisms $C_i \hookrightarrow C_n$ is permuted by $F$.
Note that of course $\gamma(\rho(g))=g$.

To conclude the proof, we observe that the symbol $\gamma(F)$ determines $F$. Indeed, every morphisms $C_i \to C_j$ in $\G_n$ admits a factorization
\begin{equation*}
C_i \twoheadrightarrow C_k \hookrightarrow C_n \twoheadrightarrow C_j
\end{equation*}
for some $C_k \in \G_n$. 
\end{proof}
\end{proposition}

\begin{lemma} \label{lma:CellsPreservedByEquiv} Let $F$ be an autoequivalence of the category $\gaunt_n$. Then $F$ restricts to an equivalence between $\G_n$ and its essential image in $\gaunt_n$; that is, $F(C_k) \cong C_k$ for all $0 \leq k \leq n$.
\begin{proof} The proper retracts of $C_n$ are precisely the cells $C_k$ for $0 \leq k < n$, and, as before, each $C_k$ is distinguished as the unique such retract such that, up to isomorphism, there exists precisely $k$ other objects which occur as further proper retracts (the cells $C_s$ for $s<k$).
Thus it is enough to show that $F(C_n) \cong C_n$ for the $n$-cell alone, as this implies the analogous statement $F(C_k) \cong C_k$ for all $0 \leq k \leq n$.
  
Recall that a \emph{generator} of a category $\C$ is an object $X$ such that the corepresentable functor $\C(X, -): \C \to \set$ is faithful \cite[pg.~127]{MR1712872}.
The collection of generators is preserved under any autoequivalence.
The $n$-cell $C_n$ is a generator for $\cat_n$ and hence also $\gaunt_n$, however the $k$-cells $C_k$ for $k < n$ are not generators.
Thus no proper retract of $C_n$ is a generator.
We claim that in fact $C_n$ is the unique generator such that every proper retract is \emph{not} a generator.
If this characterization holds, then any autoequivalence necessarily preserves the $n$-cell up to automorphism and the lemma follows.  
  
In fact we will prove a stronger statement: we claim that the $n$-cell $C_n$ is a retract of every generator of $\gaunt_n$.
Now consider the gaunt $n$-category $\partial C_{n+1}$. This may be written as
\begin{equation*}
\partial C_{n+1} = C_n \cup^{\partial C_n} C_n.
\end{equation*}
There are exactly two non-identity $n$-morphisms in $\partial C_{n+1}$; call them $a,b:C_n\to\partial C_{n+1}$.
Observe that these two functors differ only on the unique nontrivial $n$-morphism of $C_n$. 
 The unique non-trivial $n$-morphism of $C_n$, viewed as a map $C_n \to C_n$, corresponds to the element $\id\in\gaunt_n(C_n,C_n)$.   
  
Suppose now that $X$ is a generator of $\gaunt_n$.
Then there must exist a functor $X \to C_n$ such that the induced map $X_n\to\gaunt_n(C_n,C_n)$ contains the element $\id\in\gaunt_n(C_n,C_n)$ in its image, 
 for otherwise $\C(X,-)$ would not be able to distinguish $a$ and $b$, contradicting the fact that $X$ is a generator.
Thus there exists an $n$-morphism $f$ of $X$ which maps via this functor to $\id\in\gaunt_n(C_n,C_n)$.
Corresponding to $f$ is a section $C_n \to X$ that carries the unique nontrivial $n$-morphism of $C_n$ to $f$. This exhibits $C_n$ as a retract of $X$, as desired.
\end{proof}
\end{lemma}

\begin{remark} We thank Dimitri Ara who pointed out an error in an earlier version of the above lemma.
An alternative proof of this lemma has appeared in work of Dimitri Ara, Moritz Groth, and Javier Guti\'errez \cite{1312.4994}.
\end{remark}

The composition law $\ast_{j}$ (for $k$-morphisms, $n\geq k \geq j$) is corepresented by a map 
\begin{equation*}
w_j^k\colon C_{k}\to C_{k}\cup^{C_{j-1}}C_{k}.
\end{equation*}
Let $E\colon \gaunt_n \to \gaunt_n$ be an endofunctor. We will say that \emph{$E$ commutes with compositional pushouts} if for all $k \geq j \geq 1$ the natural map is a isomorphism: \begin{equation*}
	E(C_{k}\cup^{C_{j-1}}C_{k}) \cong E(C_{k})\cup^{E(C_{j-1})}E(C_{k}).
\end{equation*}
Note that it is sufficient to consider only the case $k=n$. The other cases, being retracts of the $n=k$ case, follow automatically. 

\begin{lemma}\label{lem:Endoofgaunt}
	Any endofunctor $E$ of the category $\gaunt_n$ that commutes with compositional pushouts and restricts to an automorphism of the category $\G_n$ is an autoequivalence and isomorphic to a functor of the form $\rho(g)$ for some $g\in(\Z/2)^n$.
\begin{proof}
	By Proposition \ref{prop:aut_glob_cat} the restriction of $E$ to $\G_n$ is necessarily of the form $\rho(g)$ for some $g\in (\Z/2)^n$. It suffices to prove that $E\circ\rho(g)\simeq\rho(g)\circ\rho(g)$, which is in turn simply the identity, so without loss of generality we may assume $\rho(g) = \id$, that is, that $E$ restricts to the identity functor on $\G_n$.
	
Now in this case, the isomorphisms
\begin{equation*}
\gaunt_n(C_k,X)\cong\gaunt_n(F(C_k),F(X))=\gaunt_n(C_k,F(X))
\end{equation*}
are natural in both $C_k$ and $X$, whence one obtains a natural isomorphism $\gamma\colon U\cong U\circ F$, where $U\colon\gaunt_{n}\to\Fun(\G_{n}^{\op},\set)$ is the forgetful functor from gaunt $n$-categories to globular sets.
It thus remains to show that this natural isomorphism is compatible with the compositions $\ast_j$.

For this, consider
\begin{equation*}
w_j^k\colon C_{k}\to C_{k}\cup^{C_{j-1}}C_{k},
\end{equation*}
 the morphism that corepresents the composition law $\ast_{j}$, as above. Since $E$ preserves compositional pushouts, one obtains a commutative diagram
\begin{equation*}
\begin{tikzpicture} 
\matrix(m)[matrix of math nodes, 
row sep=4ex, column sep=4ex, 
text height=1.5ex, text depth=0.25ex] 
{C_k&C_{k}\cup^{C_{j-1}}C_{k}\\
&E(C_{k})\cup^{E(C_{j-1})}E(C_{k})\\
E(C_k)&E(C_{k}\cup^{C_{j-1}}C_{k}).\\}; 
\path[>=stealth,->,font=\scriptsize] 
(m-1-1) edge node[above]{$w_j^k$} (m-1-2) 
edge[-,double distance=1.5pt] (m-3-1) 
(m-1-2) edge[-,double distance=1.5pt] (m-2-2) 
(m-2-2) edge node[right]{$\cong$} (m-3-2)
(m-3-1) edge node[below]{$E(w_j^k)$} (m-3-2); 
\end{tikzpicture}
\end{equation*}
Hence for any gaunt $n$-category $X$, we obtain a commutative diagram
\begin{equation*}
\begin{tikzpicture} 
\matrix(m)[matrix of math nodes, 
row sep=4ex, column sep=4ex, 
text height=1.5ex, text depth=0.25ex] 
{\gaunt_n(C_{k}\cup^{C_{j-1}}C_{k},X)&\gaunt_n(C_k,X)\\
\gaunt_n(E(C_{k}\cup^{C_{j-1}}C_{k}),E(X))&\\
\gaunt_n(E(C_{k})\cup^{E(C_{j-1})}E(C_{k}),E(X))&\gaunt_n(F(C_{k}),F(X))\\
\gaunt_n(C_{k}\cup^{C_{j-1}}C_{k},E(X))&\gaunt_n(C_k,E(X))\\}; 
\path[>=stealth,->,font=\scriptsize] 
(m-1-1) edge node[above]{$\ast_j$} (m-1-2) 
edge node[left]{$\cong$} (m-2-1) 
(m-1-2) edge node[right]{$\cong$} (m-3-2)
(m-2-1) edge node[left]{$\cong$} (m-3-1) 
(m-3-1) edge[-,double distance=1.5pt] (m-4-1)
(m-3-2) edge[-,double distance=1.5pt] (m-4-2)
(m-4-1) edge node[below]{$\ast_j$} (m-4-2); 
\end{tikzpicture}
\end{equation*}
in which the top and bottom morphisms are exactly the composition functors. Hence the natural isomorphism $\gamma$ is compatible with compositions, whence it lifts to a natural isomorphism $\id\cong E$, as desired.
\end{proof}	
\end{lemma}

\begin{remark}\label{rmk:gauntomegaendo}
	The above lemma is also valid for $\gaunt_n^\omega$ in place of $\gaunt_n$. 
\end{remark}

\begin{corollary} \label{Cor:autofgaunt} Any autoequivalence of the category $\gaunt_n$ is isomorphic to a functor of the form $\rho(g)$ for some $g\in(\Z/2)^n$.
\begin{proof} Any autoequivalence preserves all colimits, in particular the compositional pushouts. Moreover by Lemma~\ref{lma:CellsPreservedByEquiv}, every autoequivalence $F$ restricts to an autoequivalence of $\G_n$, and hence the assumptions of the previous lemma are met. 
	
\end{proof}
\end{corollary}

\begin{theorem}\label{rk:autofgaunt} The full subcategory
\[\Aut(\gaunt_n)\subset\Fun(\gaunt_n,\gaunt_n)\]
of the category spanned by the autoequivalences is equivalent to the discrete set $(\mathbb{Z}/2)^{n}$.
\begin{proof} Indeed, it follows from Corollary~\ref{Cor:autofgaunt} that it is essentially surjective, and it follows from Lemma \ref{lma:uniqueauto} that the action functor $\rho:(\Z/2)^n\to\Aut(\gaunt_n)$ is fully faithful.
\end{proof}
\end{theorem}
\noindent Combining these results with Rm.~\ref{rmk:gauntisIndoffpgaunt} we also obtain:
\begin{corollary}\label{cor:Aut_kappa_Gaunt_n} The full subcategory
\[\Aut(\gaunt_n^\omega)\subset\Fun(\gaunt_n^\omega,\gaunt_n^\omega)\]
spanned by the autoequivalences is equivalent to the discrete set $(\mathbb{Z}/2)^{n}$.
\end{corollary}


\section{Correspondences and the category $\gaunt_n$}\label{sect:Corr}

The category $\cat_n$ is cartesian closed, which means that for each strict $n$-category $X$, the endo-functor $(-) \times X$ admits a right adjoint, $\uHom(X,-)$. The strict $n$-categories $\uHom(X,Y)$ make $\cat_n$ enriched in itself, and thus it can be regarded as a \emph{large} strict $(n+1)$-category. 

The category $\gaunt_n$, as a full subobject of $\cat_n$, is likewise a large strict $(n+1)$-category. Moreover the $n$-category $\uHom(X,Y)$ is gaunt whenever $Y$ is gaunt, making $\gaunt_n$ itself cartesian closed. In fact, if we replace $\gaunt_n$ by a skeleton, then it is a large \emph{gaunt} $(n+1)$-category, though we will not need this observation. 

If $X$ is a strict $n$-category (possibly large) then we may consider the slice category $(\gaunt_n/X)$ of gaunt $n$-categories equipped with a functor to $X$. This is naturally enriched in $\cat_n$ as well and hence also a (large) $(n+1)$-category. If $p_0:A_0 \to X$ and $p_1: A_1 \to X$ are in $(\gaunt_n/X)$, then the $\cat_n$-enriched hom from $p_0$ to $p_1$ is given by the following pull-back:
\begin{center}
\begin{tikzpicture}
		\node (LT) at (0, 1.5) {$\uHom(p_0, p_1)$};
		\node (LB) at (0, 0) {$\uHom(A_0, A_1)$};
		\node (RT) at (3, 1.5) {$\mathrm{pt}$};
		\node (RB) at (3, 0) {$\uHom(A_0, X)$};
		\draw [->] (LT) -- node [left] {$ $} (LB);
		\draw [->] (LT) -- node [above] {$ $} (RT);
		\draw [->] (RT) -- node [right] {$p_0$} (RB);
		\draw [->] (LB) -- node [below] {$(p_1)_*$} (RB);
		\node at (0.5, 1) {$\ulcorner$};
\end{tikzpicture}
\end{center}

\begin{definition} A \emph{correspondence} (or \emph{$k$-correspondence}, for clarity) of gaunt $n$-categories is an object of the overcategory $(\gaunt_n/C_k)$.
In other words, a $k$-correspondence of gaunt $n$-categories is a gaunt $n$-category $M$ along with a functor $M \to C_k$. 
\end{definition}
 
\begin{remark}
By Lemma \ref{lma:gauntisfp} $\gaunt_n$ is locally finitely presentable. It follows (see \cite[Corollary 2.44,~2.47]{MR1294136}) that each of the categories of $k$-correspondence is also locally finitely presentable. 	

	 Given two $k$-correspondences $M\to C_k$ and $N\to C_k$, we may form the \emph{product correspondence} $M\times_{C_k}N$. This is the product in the category $(\gaunt_n/C_k)$ of $k$-correspondences.
\end{remark}

The terminology and connection to the classical theory of correspondences is made more clear by introducing the following category.
\begin{definition} Define $\corr_n^0\coloneq\gaunt_n$, and for $k>0$, define $\corr_n^k$ recursively as follows. The objects are triples $(X_0,X_1,F)$ consisting of two gaunt $n$-categories $X_0$ and $X_1$ and a functor
\[F\colon X_0^\op\times X_1\to\corr_{n-1}^{k-1}.\]
(Here $\op = \rho(1,0, \dots, 0)$ is the opposite obtained by reversing just the 1-morphisms of $X_0$.) A morphism $(X_0,X_1,F)\to (Y_0,Y_1,G)$ is a triple $(f_0,f_1,\alpha)$ consisting of functors $f_0\colon X_0\to Y_0$ and $f_1\colon X_1\to Y_1$ and a natural transformation
\[\alpha\colon F\to G\circ(f_0^{\op}\times f_1).\]
\end{definition}

\begin{lemma} There is a natural equivalence of categories
\[\phi_n^k\colon(\gaunt_n/C_k) \simeq \corr_n^k.\] 
\end{lemma}

\begin{proof} The functor $\phi_n^0$ is simply the identity. We now define $\phi_n^k$ recursively. Suppose that $k>0$, and assume that the equivalence $\phi_{n-1}^{k-1}$ has been defined.

Let us define the functor $\phi_n^k\colon(\gaunt_n/C_k)\to \corr_n^k$. For any object $p\colon X \to C_k$ of $(\gaunt_n/C_k)$, let $\phi_n^k(p)$ be the triple $(X_0,X_1,F)$, where $X_0$ and $X_1$ are the fibers of $p$ over $0$ and $1$, respectively, and the functor $F\colon X_0^\op\times X_1\to\corr_{n-1}^{k-1}$ is the composite $\phi_{n-1}^{k-1}\circ F'$, where the functor
\[F'\colon X_0^\op\times X_1\to(\gaunt_{n-1}/C_{k-1})\]
is defined as follows:
\begin{itemize}
\item For any objects $(x_0,x_1)\in X_0^\op\times X_1$, let $F(x_0,x_1)$ be the $(n-1)$-category $X(x_0,x_1)$, equipped with the functor $X(x_0,x_1)\to C_k(0,1)=C_{k-1}$ induced by $p$.
\item For any objects $(x_0,x_1)$ and $(y_0,y_1)$ of $X_0^\op\times X_1$, the functor
\[X_0(y_0,x_0)\times X_1(x_1,y_1)\to\Fun_{/C_{k-1}}(X(x_0,x_1),X(y_0,y_1))\]
is simply composition.
\end{itemize}
For any commutative triangle
\begin{equation*}
\begin{tikzpicture}[baseline]
\matrix(m)[matrix of math nodes,
row sep=4ex, column sep=4ex,
text height=1.5ex, text depth=0.25ex]
{X && Y\\
&C_k& \\ };
\path[>=stealth,->,font=\scriptsize]
(m-1-1) edge node[above]{$f$} (m-1-3)
edge[inner sep=0.5pt] node[below left]{$p$} (m-2-2)
(m-1-3) edge[inner sep=0.5pt] node[below right]{$q$} (m-2-2);
\end{tikzpicture}
\end{equation*}
of gaunt $n$-categories, we define
\[\phi_n^k(f)\coloneq(f_0,f_1,\alpha)\colon\phi_n^k(p)=(X_0,X_1,F'\circ\phi_n^k)\to(Y_0,Y_1,G'\circ\phi_n^k)=\phi_n^k(q),\]
where $f_0$ and $f_1$ are the restrictions of $f$ to the fibers, and $\alpha$ is the composite $\phi_{n-1}^{k-1}\ast\alpha'$, in which the natural transformation $\alpha'$ is the one whose components are given by the functor $X(x_0,x_1)\to Y(f(x_0),f(x_1))$ induced by $f$.
	
We now construct a quasi-inverse $\psi_n^k\colon\corr_n^k\to(\gaunt_n/C_k)$ to $\phi_n^k$. Again, when $k=0$, we let $\psi_n^0$ be the identity, and we proceed recursively. We assume $k>0$ and that the quasi-inverse $\psi_{n-1}^{k-1}$ to $\phi_{n-1}^{k-1}$ has been defined.

For any object $(X_0, X_1, F) \in \corr_n^k$, define a gaunt $n$-category $U(X_0, X_1, F)$ with object set $\ob X_0 \sqcup \ob X_1$ and
\begin{equation*}
U(X_0, X_1, F)(a,b)\coloneq\begin{cases}
X_0(a,b) & \textrm{ if } a,b \in X_0 \\
X_1(a,b) & \textrm{ if } a,b \in X_1 \\
\psi_{n-1}^{k-1}(F(a,b))& \textrm{ if } a \in X_0,b \in X_1 \\
\varnothing & \textrm{else}
\end{cases}
\end{equation*}
The composition in $U(X_0, X_1, F)$ is the obvious one, and it is clear that this defines a functor $U\colon\corr_n^k\to\gaunt_n$. We now apply this functor to the terminal object of $\corr_n^k$, namely the triple $(C_0,C_0,\phi_{n-1}^{k-1}(C_{k-1}))$. Since $U(C_0,C_0,\phi_{n-1}^{k-1}(C_{k-1}))=C_k$, it follows that $U$ factors through a functor $\psi_n^k\colon\corr_n^k\to(\gaunt_n/C_k)$.

It is now a simple matter to observe that $\psi_n^k$ is indeed quasi-inverse to $\phi_n^k$.
\end{proof}

\begin{remark} Unwinding the functor $\phi_n^k$ in the argument above, one finds that the product in the category $\corr_n^k$ may be written recursively in the following manner:
\begin{equation*}
	(X_0, X_1, F) \times (Y_0, Y_1, G) \cong (X_0 \times Y_0, X_1 \times Y_1, F \otimes G)
\end{equation*}
where $F \otimes G$ is defined as the composite:
\[(X_0 \times Y_0)^\op \times (X_1 \times Y_1) \cong (X_0^\op \times X_1) \times (Y_0^\op \times Y_1)\stackrel{(F,G)}{\longrightarrow} \corr_{n-1}^{k-1}\times\corr_{n-1}^{k-1}\stackrel{\times}{\longrightarrow} \corr_{n-1}^{k-1}.\]
\end{remark}

\begin{lemma}\label{lm:corrinthom} The category $(\gaunt_n/C_k)$ of $k$-correspondences is cartesian closed.
\begin{proof} The claim is that for any $k$-correspondence $N\to C_k$, the functor
\[-\times_{C_k}N\colon(\gaunt_n/C_k)\to(\gaunt_n/C_k)\]
admits a right adjoint $\uHom_{C_k}(N,-)$. 

Since $\gaunt_n/C_k$ is locally finitely presentable, it is cocomplete and admits a strong generator \cite[Th.~1.20]{MR1294136} and is co-wellpowered \cite[Th.~1.58]{MR1294136}. Thus by the special adjoint functor theorem \cite[Sect.~V.8]{MR1712872} the functor
\[-\times_{C_k}N\colon(\gaunt_n/C_k)\to(\gaunt_n/C_k)\]
admits a right adjoint precisely if it commutes with colimits.

When $k=0$, this follows from the fact that $\gaunt_n$ itself is cartesian closed.

For $k>0$, suppose $(\gaunt_{n-1}/C_{k-1})$ is cartesian closed. To prove that the category $(\gaunt_n/C_k)$ is cartesian closed, we require a description of colimits in terms of the equivalent category $\corr_n^k$. For any small category $\Lambda$ and any diagram $X\colon\Lambda\to\corr_n^k$ with
\[X_{\lambda}=(X_{\lambda,0},X_{\lambda,1},F_{\lambda}),\]
it is easy to see that the colimit is given by the triple $(X_0,X_1,F)$ where
\[X_0=\colim_{\lambda\in\Lambda}X_{\lambda,0}\text{\quad and\quad}X_1=\colim_{\lambda\in\Lambda}X_{\lambda,1},\]
and $F\colon X_0^{\op}\times X_1\to\corr_{n-1}^{k-1}$ is the enriched left Kan extension of
\[\colim_{\lambda\in\Lambda}F_{\lambda}\colon\colim_{\lambda\in\Lambda}(X_{\lambda,0}^{\op}\times X_{\lambda,1})\to\corr_{n-1}^{k-1}\]
along the diagonal
\[\colim_{\lambda\in\Lambda}(X_{\lambda,0}^{\op}\times X_{\lambda,1})\to X_0^{\op}\times X_1.\]

Now for any object $Y=(Y_0,Y_1,G)\in\corr_n^k$, we wish to compare $(\colim_{\lambda\in\Lambda}X_{\lambda})\times Y$ and $\colim_{\lambda\in\Lambda}(X_{\lambda}\times Y)$. In light of our descriptions of products in $\corr_n^k$, we see that the former is $(X_0 \times Y_0, X_1 \times Y_1, F \otimes G)$, and the latter is the colimit of the diagram $Z\colon\Lambda\to\corr_n^k$ that carries $\lambda$ to
\[(X_{\lambda,0} \times Y_1, X_{\lambda,1} \times Y_1, F_{\lambda} \otimes G).\]
Note that, since $\gaunt_n$ is cartesian closed, one has
\[\colim_{\lambda\in\Lambda}(X_{\lambda,0}\times Y_0)\cong X_0\times Y_0\text{\quad and\quad}\colim_{\lambda\in\Lambda}(X_{\lambda,1}\times Y_1)\cong X_1\times Y_1;\]
hence our description of colimits in $\corr_n^k$ exhibits the colimit of $Z$ as $(X_0 \times Y_0, X_1 \times Y_1, (F \otimes G)')$, where $(F\otimes G)'$ is the enriched left Kan extension of
\[\colim_{\lambda\in\Lambda}(F_{\lambda}\otimes G)\colon\colim_{\lambda\in\Lambda}((X_{\lambda,0}\times Y_0)^{\op}\times (X_{\lambda,1}\times Y_1))\to\corr_{n-1}^{k-1}\]
along the diagonal
\[\colim_{\lambda\in\Lambda}((X_{\lambda,0}\times Y_0)^{\op}\times (X_{\lambda,1}\times Y_1))\to (X_0\times Y_0)^{\op}\times(X_1\times Y_1).\]
Our induction hypothesis is that $\corr_{n-1}^{k-1}$ is cartesian closed; so this enriched left Kan extension can be identified with the composition of the enriched left Kan extension of
\[\colim_{\lambda\in\Lambda}(F_{\lambda},G)\colon\colim_{\lambda\in\Lambda}(X_{\lambda,0}^{\op}\times X_{\lambda,1})\times(Y_0^{\op}\times Y_1)\to\corr_{n-1}^{k-1}\times\corr_{n-1}^{k-1}\]
along the diagonal
\[\colim_{\lambda\in\Lambda}(F_{\lambda},G)\colon\colim_{\lambda\in\Lambda}(X_{\lambda,0}^{\op}\times X_{\lambda,1})\times(Y_0^{\op}\times Y_1)\to (X_0^{\op}\times X_1)\times(Y_0^{\op}\times Y_1).\]
But now this left Kan extension is simply the product of $G$ with the left Kan extension that defines $F$. In other words, we have an isomorphism $(F\otimes G)'\cong F\otimes G$, whence the proof is complete.
\end{proof}
\end{lemma}


\section[The categories Theta and Upsilon]{The categories $\Theta_n$ and $\Upsilon_n$}\label{sect:Upsilon}

For many purposes, it is unwieldy to contemplate all gaunt $n$-categories (or even all compact gaunt $n$-categories). The critical structural features of $n$-categories are already captured by far smaller categories. One such smaller category is Joyal's category $\Theta_n$ of `$n$-disks' (Definition \ref{dfn:wreathproduct}):
\begin{definition}[{\cite[Definition 3.1]{MR2331244}}]\label{dfn:wreathproduct}
	Let $C$ be a small category. The {\em wreath product} $\Delta \wr C$ is the category 
	\begin{itemize}
		\item whose objects consist of tuples $([n]; c_1, \dots, c_n)$ where $[n] \in \Delta$ and $c_i \in C$, and
		\item whose morphisms from $([m]; a_1, \dots, a_m)$ to $([n]; b_1, \dots, b_n)$ consist of tuples $(\phi; \phi_{ij})$, where $\phi: [m] \to [n]$, and $\phi_{ij}: a_i \to b_j$ where $0 < i \leq m$, and $\phi(i-1) < j \leq \phi(i)$.
	\end{itemize}

The category $\Theta_n$ is now defined inductively as a wreath product: $\Theta_1 = \Delta$, and $\Theta_n = \Delta \wr \Theta_{n-1}$. In particular this gives rise to embeddings $\sigma: \Theta_{n-1} \to \Theta_n$, given by $\sigma(o) = ([1]; o)$, and $\iota: \Delta \to \Theta_n$ given by $\iota([n]) = ([n]; ([0]), \dots, ([0]))$.
\end{definition}

There is a fully-faithful embedding $i\colon\Theta_n \hookrightarrow \gaunt_n$ as a dense subcategory \cite[Th.~3.7]{MR2331244}. The image under $i$ of $([m]; a_1, \dots, a_m)$ may be described inductively as the following colimit:
\begin{equation*}
	i([m]; o_1, \dots, o_m) = \sigma( i(o_1)) \cup^{C_0} \sigma( i(o_2)) \cup^{C_0} \cdots  \cup^{C_0} \sigma( i(o_m)). 
\end{equation*}
This colimit, taken in $\gaunt_n$, is a series of pushouts in which $C_0$ is embedded into $\sigma( i(o_k))$ via $\top$ and into $\sigma(i(o_{k+1}))$ via $\bot$ as described after Ex.~\ref{Ex:Strictncats}. There is no possible confusion by the meaning of $\sigma$, as $i(\sigma(o)) = \sigma( i(o))$ for all $o \in \Theta_{n-1}$.

Since we will be concerned with the study of correspondences, it is convenient to enlarge $\Theta_n$ to contain products of correspondences:
\begin{definition} \label{Def:Upsilon} The category $\Upsilon_n$ is the smallest full subcategory of $\gaunt_n$ that contains $\Theta_n$ and is closed under products of correspondences, $(M,N)\mapsto M\times_{C_k}N$.
\end{definition}

\begin{remark}\label{rmk:fiberprodofcells} We now examine the fiber products of cells in detail.
We aim to express these fiber products as simple iterated colimits of cells.
Let $\varphi: C_i \to C_j$ and $\psi: C_k \to C_j$ be a pair of functors ($i,j,k \geq 0$).
A map of cells $\varphi: C_i \to C_j$ either factors as a composite $C_i \to C_0 \to C_j$ or is a suspension $\varphi = \sigma(\xi)$ of some map $\xi:C_{i-1} \to C_{j-1}$.

We thus begin by contemplating the case in which $\varphi$ is \emph{not} the suspension of a map of lower dimensional cells.
In this case we have a diagram of pullback squares
\begin{center}
\begin{tikzpicture}
\node (LT) at (0, 1.5) {$C_i \times F$};
\node (LB) at (0, 0) {$C_i$};
\node (MT) at (2, 1.5) {$F$};
\node (MB) at (2, 0) {$C_0$};
\node (RT) at (4, 1.5) {$C_k$};
\node (RB) at (4, 0) {$C_j$};
\draw [->] (LT) -- node [left] {$$} (LB);
\draw [->] (LT) -- node [above] {$$} (MT);
\draw [->] (MT) -- node [above] {$$} (RT);
\draw [->] (MT) -- node [right] {$$} (MB);
\draw [->] (RT) -- node [right] {$\psi$} (RB);
\draw [->] (MB) -- node [below] {$$} (RB);
\draw [->] (LB) -- node [below] {$$} (MB);
\node at (0.5, 1) {$\ulcorner$};
\node at (2.5, 1) {$\ulcorner$};
\end{tikzpicture}
\end{center}
Here $F$ is the fiber of $\psi: C_k \to C_j$ over the unique object in the image of $\varphi$.
There are four possibilities:
\begin{enumerate}[(A)]
\item The image of $\psi$ may be disjoint from the image of $\varphi$, in which case $F = \partial C_0 = \emptyset$.
Hence $F$ and also $C_i \times F$ are the empty colimit of cells.
\item The fiber may be a zero cell, $F=C_0$, in which case $C_i \times F \cong C_i$ is trivially a colimit of cells.
\item The fiber may be the $k$-cell $F \cong C_k$, but we have $i=0$.
In this case $C_i \times F \cong F \cong C_k$ is again trivially a colimit of cells.
\item The fiber may be an $k$-cell $F \cong C_k$, and we have $i \geq 1$.
In this case we have (cf. \cite[Proposition 4.9]{Rezk})
\begin{equation*} 
C_i\times C_k\cong(C_i\cup^{C_0}C_k)\cup^{\sigma(C_{i-1}\times C_{k-1})}(C_k\cup^{C_0}C_i)
\end{equation*}
where for each pushout $C_x \cup^{C_0} C_y$, the object $C_0$ is included into the final object of $C_x$ and the initial object of $C_y$.
\end{enumerate}

As the suspension functor $\sigma$ commutes with pullback squares, a general pullback of cells is the suspension of one of the types just considered. Moreover, as the suspension functor also commutes with pushout squares, the above considerations give a recipe for writing any fiber product of cells as an iterated pushout of cells. This will be made precise in Lemma \ref{lma:fiberproductsareColimits}. 
\end{remark}

\begin{notation} The inclusion $\Upsilon_n\hookrightarrow\gaunt_n$ induces a fully faithful nerve functor
\begin{equation*}
\nu:\gaunt_n\hookrightarrow\Fun(\Upsilon_n^{\op},\set).
\end{equation*}
In particular, we may regard gaunt $n$-categories as particular presheaves of sets on the category $\Upsilon_n$ (precisely which presheaves will be determined in Corollary \ref{cor:0-TruncatedAreGaunt}). Note that the nerve functor commutes with all limits, hence in particular fiber products. 
\end{notation}

\begin{notation} \label{Note:FundPushouts} Let $S_{00}$ consist of the union $A\cup B\cup C\cup D$ of the following four finite sets of maps of presheaves on $\Upsilon_n$:
\begin{equation*}
A\coloneq\left\{ \nu C_{i-1} \cup^{\nu (\partial C_{i-1})}   \nu C_{i-1} \to \nu (\partial C_{i}) \; | \; 0 \leq i \leq n-1 \right\}
\end{equation*}
(when $i=0$, we interpret this as the empty presheaf mapping to the nerve of the empty $n$-category),
\begin{equation*}
B\coloneq\left\{ \nu C_j \cup^{\nu C_i} \nu C_j \to \nu(C_j \cup^{C_i} C_j) \; | \; 0 \leq i < j \leq n \right\},
\end{equation*}
\begin{equation*}
\begin{split}
C\coloneq\Big\{\nu( C_{i + j} \cup^{C_i} C_{i+k}) \cup^{\nu \sigma^{i+1}(C_{j-1} \times C_{k-1})}\nu(C_{i + k} &\cup^{C_i} C_{i +j}) \to \nu(C_{i + j} \times_{C_i} C_{i + k})\\
&\Big|\;0\leq i\leq n\textrm{, }0< j,k\leq n-i \Big\},
\end{split}
\end{equation*}
and, lastly,
\begin{equation*}
D\coloneq\left\{ \nu \sigma^k(\Delta^3) \cup^{\nu \sigma^k(\Delta^{\{0,2\}}\sqcup\Delta^{\{1,3\}})} \nu \sigma^k(\Delta^0\sqcup\Delta^0) \to \nu C_k \; | \; 0 \leq k \leq n\right\}.
\end{equation*}

Now let $S_0$ be the smallest class of morphisms $U \to V$ in $\Fun( \Upsilon_n^{\op}, \set)$ that (a) is closed under isomorphism, (b) contains $S_{00}$, and (c) is closed under the operation $-\times_{C_k}N$ for any functor $V\to C_k$ and any $k$-correspondence $N\to C_k$ with $N\in\Upsilon_n$.
\end{notation}

\begin{lemma} \label{lma:gauntIsLocal} Suppose $X$ a gaunt $n$-category. Then the presheaf $\nu X:\Upsilon_n^{\op}\to\set$ is local with respect to the morphisms of $S_0$. 
\begin{proof} Forming each of the pushouts of $S_{00}$ in $\gaunt_n$ yields an equivalence, so $X$ is local with respect to $S_{00}$.

Now let $S_0'\subseteq S_0$ denote the class of morphisms $f:U\to V$ in $S_0$ such that $\nu X$ is local with respect to $f$ for any gaunt $X$. We have observed that $S_0'$ contains $S_{00}$. It is also visibly closed under isomorphism.

We complete the proof by showing that $S'_0$ is closed under the operation $-\times_{C_i}N$ for any $N\in\Upsilon_n$.
Indeed, suppose $U \to V$ a morphism of $S_0'$.
We claim that for any morphism $V\to C_k$ and any functor $N\to C_k$, the map
\[\Upsilon_n(V\times_{C_k}N,X)\to\Upsilon_n(U\times_{C_k}N,X)\]
is a bijection.
For each $W\in\Upsilon_n$, we have 
\begin{align*}
\gaunt_n(W\times_{C_i}N, X)&\cong(\gaunt_{n}/C_k)(W\times_{C_k}N,X\times C_k)\\
&\cong(\gaunt_{n}/C_k)(W,\uHom_{C_k}(N,X\times C_k)),
\end{align*}
where $\uHom_{C_k}(N,-)$ denotes the right adjoint of Lemma \ref{lm:corrinthom}.
The claim now follows from the observation that as $\uHom_{C_i}(N, X \times C_i)$ is gaunt, it is local with respect to $U \to V$.
\end{proof}
\end{lemma}

\begin{lemma} \label{lma:fiberproductsareColimits} Let
\[\C\coloneq S_{00}^{-1} \Fun( \Upsilon_n^{\op}, \set)\]
denote the full subcategory of presheaves of sets which are local with respect to the the morphisms of $S_{00}$.
Let $\varphi: C_i \to C_j$ and $\psi: C_k \to C_j$ be an arbitrary pair of maps ($i,j,k \geq 0$).
Then $\nu(C_i \times_{C_j} C_k)$ is contained in the smallest full subcategory of $\C$ that contains the nerves of cells and is closed under the formation of colimits.  
\begin{proof} Recall that $\nu$ commutes with limits.
Let $m$ ($\leq i,j,k$) be the largest integer such that $\varphi = \sigma^m(g)$ and $\psi = \sigma^m(f)$ are both $m$-fold suspensions of maps, $g: C_{i-m} \to C_{j-m}$ and $f: C_{k-m} \to C_{j-m}$.

Suppose, without loss of generality, that $\varphi$ is not an $(m+1)$-fold suspension of a map.
We thus have an $m$-suspension of the situation considered in Rk. \ref{rmk:fiberprodofcells}; that is, we have a diagram of pullback squares
\begin{center}
\begin{tikzpicture}
\node (LT) at (0, 1.5) {$\sigma^m(C_{i-m} \times_{C_0} F)$};
\node (LB) at (0, 0) {$C_i = \sigma^m(C_{i-m})$};
\node (MT) at (4, 1.5) {$\sigma^m(F)$};
\node (MB) at (4, 0) {$C_m = \sigma^m(C_0)$};
\node (RT) at (7, 1.5) {$C_k$};
\node (RB) at (7, 0) {$C_j,$};
\draw [->] (LT) -- (LB);
\draw [->] (LT) --  (MT);
\draw [->] (MT) --  (RT);
\draw [->] (MT) --  (MB);
\draw [->] (RT) -- node [right] {$\psi = \sigma^m(f)$} (RB);
\draw [right hook->] (MB) -- node [below] {$\sigma^m(g)$} (RB);
\draw [->] (LB) -- node [below] {$\sigma^m(!)$} (MB);
\node at (0.5, 1) {$\ulcorner$};
\node at (4.5, 1) {$\ulcorner$};
\end{tikzpicture}
\end{center}
where as above $F$ denotes the fiber of $f: C_{k-m} \to C_{j-m}$ over the image of $g$.
So let us consider each of the cases A-D of Rk. \ref{rmk:fiberprodofcells} in turn.
\begin{enumerate}[(A)]
\item If $F = \emptyset$, then 
\begin{equation*}
C_i\times_{C_j}C_k\cong\sigma^m(C_{i-m}\times_{C_0}F)\cong\sigma^m(\emptyset)\cong\partial C_m.
\end{equation*}
In this case, the morphisms of $A\subset S_{00}$ provide an iterative construction of $\nu\partial C_m$ as a colimit in $S_{00}^{-1}\Fun(\Upsilon_n^{\op},\set)$ of cells.
\item Next, if $F\cong C_0$, then  
\begin{equation*}
C_i\times_{C_j}C_k\cong\sigma^m(C_{i-m}\times_{C_0}F)\cong\sigma^m(C_{i-m})\cong C_i
\end{equation*}
is already a cell.
\item Similarly, if $F\cong C_{k-m}$, but $i=m$, then
\begin{equation*}
C_i\times_{C_j}C_k\cong\sigma^m(C_{0}\times_{C_0}F)\cong\sigma^m(C_{k-m})\cong C_k
\end{equation*}
is again already a cell.
\item Finally, let us suppose that $F\cong C_{\ell}$ with $i=m+p$ and $k=m+\ell$ for $p>0$. In this case we have,
\begin{equation*}
C_i\times_{C_j}C_k\cong C_{m+p}\times_{C_m} C_{m+\ell}
\end{equation*}
is precisely the fiber product considered in the set $C\subset S_{00}$.
One readily observes that morphisms of $B$ and $C$ provide an inductive construction of this fiber product as an iterated colimit of cells in $\C$.\qedhere
\end{enumerate}
\end{proof}
\end{lemma}


\part{The moduli space of theories of $(\infty,n)$-categories}

\section{The Unicity Theorem}\label{sect:axioms}

We now introduce our axioms for the theory of $(\infty,n)$-categories.

\begin{basic-data*}
	We assume that $\C$ is a presentable $\infty$-category equipped with a fully faithful functor $f\colon\gaunt_n^{\omega}\hookrightarrow\C$.
\end{basic-data*}

The first axiom states that every object of $\C$ can be written as a colimit of gaunt $n$-categories in a canonical manner; this is called \emph{strong generation.}

\begin{axiom*}[\hypertarget{axiom:C.1}{C.1: Strong generation}] The functor $f$ is \emph{dense}, or, equivalently in the language of \cite[4.4.2]{G}, it \emph{strongly generates} $ \C $.
	That is, the left Kan extension of $f$ along itself is the identity functor on $\C$.
\end{axiom*}

Equivalently, any object $X\in\C$ is the canonical colimit of the gaunt $n$-categories mapping to it:
\[
\colim_{H\in\gaunt_n/X}H\simeq X.
\]
This is equivalent to the condition that the functor $ f $ induces an localization $ \pre(\gaunt_n) \to \C $;
that is, $\C $ can be written as $W^{-1}\pre(\gaunt_n)$ for some class of maps $W$ of small generation (\cite[Remark 20.4.1.5]{SAG}, \cite[Proposition 5.5.4.16]{HTT}).

The next axiom states that every object of $\C$ can be written as a colimit of objects of $\G_n$, but not necessarily in this canonical manner.

\begin{axiom*}[\hypertarget{axiom:C.2}{C.2: Weak generation}] If $\E\subseteq\C$ is a full subcategory that contains the image $f(\G_n)$ and is closed under colimits, then $\E=\C$.
\end{axiom*}

\begin{remark}\label{rmk:CellsDetect}
	In any $\infty$-category $\C$ which satisfies Axiom C.2 the cells detect equivalences. That is $f: X \to Y$ is an equivalence in $\C$ if and only if it induces  equivalences 
	$\map(C_k, X) \to \map(C_k, Y)$
for all $0 \leq k \leq n$. This is clear, since for such a map the full subcategory of those  $H \in \C$ such that $\map(H, X) \to \map(H,Y)$ is an equivalence is stable under colimits and contains the cells, and is thus all of $\C$. 
\end{remark}
%
%
%
%

We also demand that $\C$ admit internal Homs for correspondences.
\begin{axiom*}[\hypertarget{axiom:C.3}{C.3: Internal Homs for correspondences}] For any morphism $\eta\colon X \to f(C_i)$ of $\C$, the fiber product functor
\begin{equation*}
\eta^*\colon \C_{/f(C_i)} \to \C_{/X}
\end{equation*}
preserves colimits.
\end{axiom*}
\noindent The Adjoint Functor Theorem \cite[Corollary 5.5.2.9]{HTT} implies that this is equivalent to the existence of internal homs for the categories of correspondences $\C_{/f(C_i)}$.

Equivalently, Axiom C.3 states that every morphism to $ C_k $ is, in the language of Ayala--Francis \cite{AyalaFrancis}, an exponentiable fibration.
(We are grateful to the referee for suggesting this observation.)

We introduce a special collection of maps of $\C$. Denote $S_{00}$ consist of the union $A\cup B\cup C\cup D$ of the following four finite sets of maps of $\C$:
\begin{equation*}
A\coloneq\left\{f(C_{i-1})\cup^{f(\partial C_{i-1})}f(C_{i-1})\to f(\partial C_{i}) \; | \; 0 \leq i \leq n-1 \right\}
\end{equation*}
(when $i=0$, we interpret this as the initial object of $\C$ mapping to the image under $f$ of the empty $n$-category),
\begin{equation*}
B\coloneq\left\{f(C_j)\cup^{f(C_i)}f(C_j)\to f(C_j \cup^{C_i} C_j) \; | \; 0 \leq i < j \leq n \right\},
\end{equation*}
\begin{equation*}
\begin{split}
C\coloneq\Big\{f(C_{i+j}\cup^{C_i}C_{i+k})\cup^{f(\sigma^{i+1}(C_{j-1} \times C_{k-1}))}f(C_{i+k}&\cup^{C_i}C_{i+j})\to f(C_{i + j} \times_{C_i} C_{i + k})\\
&\Big|\;0\leq i\leq n\textrm{, }0< j,k\leq n-i \Big\},
\end{split}
\end{equation*}
and, lastly,
\begin{equation*}
D\coloneq\left\{f(\sigma^k([3]))\cup^{f(\sigma^k({\{0,2\}}\sqcup{\{1,3\}}))}f(\sigma^k([0]\sqcup[0]))\to f(C_k) \; | \; 0 \leq k \leq n\right\}.
\end{equation*}

\begin{axiom*}[\hypertarget{axiom:C.4}{C.4: Fundamental pushouts}] Each of the finite number of maps comprising $S_{00}$ is an equivalence.	
\end{axiom*}

Finally, we will require that $\C$ be \emph{versal} with Axioms C.1--4, so that any other presentable $\infty$-category satisfying Axioms C.1--4 admits some left adjoint from $\C$:
\begin{axiom*}[\hypertarget{axiom:C.5}{C.5: Versality}] For any $\infty$-category $\D$ and any fully faithful functor $g\colon\gaunt_n^{\omega}\hookrightarrow\D$ satisfying Axioms C.1--4, there exist a left adjoint $K\colon \C \to \D$ and a natural transformation $\xi\colon K\circ f \to g$ such that $\xi|_{\G_n}\colon K\circ f|_{\G_n} \to g|_{\G_n}$ is an equivalence.
\end{axiom*}

\begin{definition} A pair $(\C,f)$ satisfying these axioms (C.1--5) will be said to be a \emph{theory of $(\infty,n)$-categories.} We define a subcategory $\thy_{(\infty,n)}$ of the $\infty$-category of $\infty$-categories: the objects are $\infty$-categories $\C$ that underlie a theory $(\C,f)$ of $(\infty,n)$-categories, and the morphisms are equivalences of these $\infty$-categories. 
\end{definition}

The $\infty$-category $\thy_{(\infty,n)}$ is an $\infty$-groupoid and thus a homotopy type; it is the \emph{moduli space of theories of $(\infty,n)$-categories}. Our main theorem is a computation of this homotopy type.
\begin{theorem}[Unicity]\label{theorem:unicity} One has
\[
\thy_{(\infty,n)}\simeq B(\mathbb{Z}/2)^n.
\]
\end{theorem}

The proof will occupy the next few sections, and is organized as follows. 
\begin{enumerate}
	\item In Section~\ref{sec:colossal_model} we introduce a particular $\infty$-category and verify that it satisfies Axioms C.1-5. This shows that $\thy_{(\infty,n)}$ is non-empty, and we denote our chosen model by $\cat_{(\infty,n)}$. For Axioms C.1-4 this means simply varifying the corresponding properties of our model. For Axiom C.5, Versality, we suppose another $\infty$-category $\D$ satisfying Axioms C.1-4 and must then show how those axioms guerentee the existence of the desired functor and transformation out of $\cat_{(\infty,n)}$. In fact a stronger versality property holds: $\D$ need only satisfy C.1, C.3, and C.4.
	\item In Section~\ref{sec:connected} we show that the space $\thy_{(\infty,n)}$ is connected, i.e., any two $\infty$-categories satisfying C.1-5 are equivalent. The proof relies essentially on our $\infty$-categories satisfying Axioms C.1, C.2, and C.5. We note, however, that Axiom C.5 is quantified over the other four axioms, and so this step in fact relies on all five axioms.
	\item Having shown $\thy_{(\infty,n)}$ is non-empty and connected, it remains to compute its loopspace. This is done in Section~\ref{sec:loopspace} by directly computing the automorphism space of our prefered model. 
\end{enumerate}


\section{A colossal model of $(\infty,n)$-categories} \label{sec:colossal_model}

We will now construct a theory of $(\infty,n)$-categories -- i.e., an $\infty$-category that satisfies \hyperlink{axiom:C.1}{Axioms (C.1--5)}.
The axioms of Strong Generation, Internal Homs for correspondences, and Fundamental pushouts together suggest the following definition.

\begin{definition}
	Let $T_0$ be the smallest set of morphisms of $\pre(\gaunt_n^{\omega})$ that
	\begin{itemize}
		\item contains the morphisms of Notation \ref{Note:FundPushouts} (i.e., the morphisms that represent the fundamental pushouts of \hyperlink{axiom:C.3}{Axiom (C.3)}), and
		\item $T_0$ is stable under the operation $ H \times_{C_i} (-) $ for $ H \in \gaunt_n^{\omega}$.
	\end{itemize}
	Let $T$ be the saturated class of morphisms generated by $T_0$.

	Define the $\infty$-category of \emph{$(\infty,n)$-precategories} as the localization
	\[
		\precat_{(\infty,n)} \coloneq T^{-1}\pre(\gaunt_n^{\omega}) .
	\]
\end{definition}

One might begin to feel optimistic that perhaps $ \precat_{(\infty,n)} $ already satisfies our axioms;
indeed, it is easy to see that this $\infty$-category satisfies \hyperlink{axiom:C.1}{Axiom (C.1)} and \hyperlink{axiom:C.4}{Axiom (C.4)}.
As it happens, the closure of $T_0$ under fiber products over cells will ensure that it satisfies \hyperlink{axiom:C.3}{Axiom (C.3)} as well.
Nevertheless, it does not satisfy \hyperlink{axiom:C.2}{Axiom (C.2)}.
We can address this problem directly:
\begin{definition}
	The $\infty$-category $ \cat_{(\infty,n)}$ is the smallest subcategory that contains the cells $C_i$ (for $0\leq i\leq n$) and is closed under colimits.
\end{definition}

Unfortunately, by enforcing \hyperlink{axiom:C.2}{Axiom (C.2)}, we have lost our trivial proof of \hyperlink{axiom:C.1}{Axiom (C.1)}: the inclusion
\[
	\cat_{(\infty,n)} \hookrightarrow \precat_{(\infty,n)}
\]
preserves all colimits, so it admits a \emph{right} adjoint, but
it does not follow directly from this that our category is a localization of presheaves on $\gaunt_n^{\omega}$.
To guarantee this, we need a \emph{further} right adjoint.

To construct these adjoints, we employ our category $\Upsilon_n$ of Definition \ref{Def:Upsilon}.
\begin{notation} \label{Not:S_o-quasicat}
	We consider the $\infty$-category $\pre(\Upsilon_n)$ of presheaves on the category $\Upsilon_n$ of Definition \ref{Def:Upsilon} and the Yoneda embedding
	\begin{equation*}
		f\colon\Upsilon_n\hookrightarrow\tau_{\leq 0}\pre(\Upsilon_n)\hookrightarrow\pre(\Upsilon_n).
	\end{equation*}
	Let $S_{00}$ denote the image of the finite set of morphisms of the same name as defined in Notation \ref{Note:FundPushouts}, which also represent the morphisms that appeared in (C.3). 
	Let $S_0$ be the smallest class of morphisms of $\pre(\Upsilon_n)$ that is stable under equivalence, contains $S_{00}$, and is stable under the operation $X \times_{C_i}(-)$ for $X \in \Upsilon_n$. One may check that $S_0$ has countably many isomorphism classes of maps and agrees with the essential image of the class $S_0$ introduced in Notation \ref{Note:FundPushouts}. Let $S$ be the strongly saturated class of morphisms of $\pre(\Upsilon_n)$ generated by the class $S_0$.
	Let us study the localization $ S^{-1}\pre(\Upsilon_n)$.
	
	The inclusion $ j \colon \Upsilon_n \hookrightarrow \gaunt_n $ induces a functor $ j^{\ast} \colon \pre(\gaunt_n^{\omega}) \to \pre(\Upsilon_n) $, which admits a left adjoint $j_!$ (given by left Kan extension) and a right adjoint $j_{\ast} $ (given by right Kan extension).
	Since $j_!$ and $j^{\ast}$ each preserve those presheaves represented by objects of $\Upsilon_n$ as well as all colimits, it follows that
	\[
		j_!(S) \subseteq T \quad \text{and} \quad j^{\ast}(T) \subseteq S .
	\]
	Consequently,
	\[
		j^{\ast}(\precat_{(\infty,n)}) \subseteq S^{-1}\pre(\Upsilon_n) \quad \text{and} \quad j_{\ast}(S^{-1}\pre(\Upsilon_n)) \subseteq \precat_{(\infty,n)} .
	\]
	And so $j^{\ast} \colon \precat_{(\infty,n)} \to S^{-1}\pre(\Upsilon_n) $ admits a left adjoint $L_Tj_!$ (where $L_T$ is the localization $\pre(\gaunt_n^{\omega}) \to T^{-1}\pre(\gaunt_n^{\omega}) = \precat_{(\infty,n)} $) and a right adjoint $j_{\ast}$.

\end{notation}

\begin{lemma} \label{lem:fully-faithfulj}
	The restriction of the functor $ j^{\ast} $ to $ \cat_{(\infty,n)}$ is fully-faithful.
\end{lemma}

\begin{proof}
	Since the cells are contained in $\Upsilon_n$, it follows that $j_!j^*C_i \simeq C_i$. 
	Suppose that $Y \in \precat_{(\infty,n)}$; then the unit $Y \to j_{\ast}j^{\ast}Y$ induces an equivalence
		\[
			 \map(C_i, Y) \simeq \map(j_! j^* C_i, Y) \simeq  \map(C_i, j_{\ast}j^{\ast}Y) 
		\]
		for any cell $C_i$.
		
	Now consider the smallest subcategory of $ \cat_{(\infty,n)}$ consisting of objects $X$ such that the unit map induces an equivalence $\map(X,Y) \simeq \map(X,  j_{\ast}j^{\ast}Y) $	for all $Y \in \precat_{(\infty,n)}$. As we have seen this subcategory contains the cells. It is also closed under colimits since if we write $X \simeq \colim_{\alpha} X_{\alpha}$, where all the $X_{\alpha}$ are in this subcategory, then
	\[
		\map(X,Y) \simeq \lim_{\alpha} \map(X_{\alpha}, Y) \simeq \lim_{\alpha} \map(X_{\alpha}, j_{\ast}j^{\ast}Y) \simeq \map(X,j_{\ast}j^{\ast}Y).
	\]
	It follows that this subcategory is all of $ \cat_{(\infty,n)}$, and thus $j^{\ast}$ induces an equivalence $\map(X,Y) \simeq \map(j^{\ast}X,j^{\ast}Y)$.
\end{proof}

Before we continue to show that the restriction of $j^*$ to $ \cat_{(\infty,n)}$ is essentially surjective, we will first show that $S^{-1}\pre(\Upsilon_n)$ satisfies \hyperlink{axiom:C.3}{Axiom (C.3)}. We will formulate this as a proposition.

\begin{proposition} \label{prop:Axiom4reduction}
	Let $\R$ be a small category, and let $i\colon \R \to \gaunt_n^\omega $ be a functor.
	Let $K$ be a strongly saturated class of morphisms in $\pre(\R)$ of small generation.
	Denote by $i^*\colon\pre(\gaunt_n^\omega)\to\pre(\R)$ the precomposition with the functor $i$.
	Consider $\C \coloneq U^{-1}\pre(\R)$ along with the restriction $f$ of $i^{\ast}$ to the representable objects;
	then the pair $(\C,f)$ satisfies \hyperlink{axiom:C.3}{Axiom (C.3)} if and only if $ K $ enjoys the following condition.
\end{proposition}

\begin{enumerate}[{\hspace{\parindent}(C.}1{-bis)}]\addtocounter{enumi}{2}
	\item \hypertarget{axiom:C.3-bis}{}
		There is a subset $K_0\subset K$ that generates $K$ as a strongly saturated class for which the following condition holds.
		For any object $W$ of $\R$, any functor $i(W)\to C_k$ of $\gaunt_n$, any morphism $U\to V$ of $K_0$, and any morphism $V\to C_k$ of $\pre(\R)$, the induced morphism
		\begin{equation*}
			U\times_{C_k}i(W)\to V\times_{C_k}i(W)
		\end{equation*}
		lies in $K$.
\end{enumerate} 

\begin{proof}
	For any $K$-local object $X$ of $\pre(\R)$, a morphism $Y\to X$ represents an object of $(K^{-1}\pre(\R))_{/X}$ if and only if, for any morphism $U\to V$ of $K_0$, the square
\begin{center}
 \begin{tikzpicture}
 	\node (LT) at (0, 1) {$\map(V,Y)$};
 	\node (LB) at (0, 0) {$\map(V,X)$};
 	\node (RT) at (3, 1) {$\map(U,Y)$};
 	\node (RB) at (3, 0) {$\map(U,X)$.};
 	\draw [->] (LT) -- node [left] {$ $} (LB);
 	\draw [->] (LT) -- node [above] {$ $} (RT);
 	\draw [->] (RT) -- node [right] {$ $} (RB);
 	\draw [->] (LB) -- node [below] {$ $} (RB);
 \end{tikzpicture}
\end{center}
is homotopy cartesian, since the horizontal map at the bottom is an equivalence. For this, it suffices to show that the induced map on homotopy fibers over any vertex of $\map(V,X)$ is an equivalence. Unpacking this, we obtain the condition that for any morphism $V\to X$, the map
\begin{equation*}
\map_{/X}(V,Y)\to \map_{/X}(U,Y)
\end{equation*}
is a weak equivalence. We therefore deduce that $(K^{-1}\pre(\R))_{/X}$ may be exhibited as a localization $K_X^{-1}(\pre(\R)_{/X})$, where $K_X$ is the strongly saturated class generated by the set of diagrams of the form
\begin{center}
 \begin{tikzpicture}
 	\node (LT) at (0, 1) {$U$};
 	\node (B) at (1, 0) {$X$};
 	\node (RT) at (2, 1) {$V$};
 	\draw [->] (LT) -- node [left] {$ $} (B);
 	\draw [->] (LT) -- node [above] {$\phi$} (RT);
 	\draw [->] (RT) -- node [right] {$ $} (B);
 \end{tikzpicture}
\end{center}
in which $\phi\in K_0$.

Now suppose $\eta\colon Z\to C_k$ a morphism of $K^{-1}\pre(\R)$. Since colimits are universal in $\pre(\R)$ \cite[\S~6.1.1]{HTT}, the functor
\begin{equation*}
\pre(\R)_{/C_k}\to\pre(\R)_{/Z}
\end{equation*}
given by pullback along $\eta$ preserves all colimits, and the universal property of localizations guarantees that the composite
\begin{equation*}
\pre(\R)_{/C_k}\to\pre(\R)_{/Z}\to K_Z^{-1}(\pre(\R)_{/Z})\simeq(K^{-1}\pre(\R))_{/Z}
\end{equation*}
descends to a colimit-preserving functor
\begin{equation*}
(K^{-1}\pre(\R))_{/C_k}\simeq K_{C_k}^{-1}(\pre(\R)_{/C_k})\to K_Z^{-1}(\pre(\R)_{/Z})\simeq(K^{-1}\pre(\R))_{/Z}
\end{equation*}
(which then must also be given by the pullback along $\eta$) if and only if, for any diagram
\begin{center}
 \begin{tikzpicture}
 	\node (LT) at (0, 1) {$U$};
 	\node (B) at (1, 0) {$C_k$};
 	\node (RT) at (2, 1) {$V$};
 	\draw [->] (LT) -- node [left] {$ $} (B);
 	\draw [->] (LT) -- node [above] {$\phi$} (RT);
 	\draw [->] (RT) -- node [right] {$ $} (B);
 \end{tikzpicture}
\end{center}
in which $0\leq k\leq n$ and $\phi\in K_0$, the induced morphism $U\times_{C_k}Z\to V\times_{C_k}Z$ lies in $K$.

It is clear that it suffices to check this only for nondegenerate morphisms $V\to C_k$.
It now remains only to show that it suffices to check this for objects $Z$ among the essential image of $\R$.
This follows from the fact that the class $K$ is strongly saturated and the fact that $\R$ generates $\pre(\R)$ under colimits.
\end{proof}

The collection $\R = \gaunt_n^{\omega}$, $i=\id$, $K=T$, and $K_0=T_0$, satisfies \hyperlink{axiom:C.3-bis}{Axiom (C.3-bis)} by construction, whence  we may deduce that the pair consisting of the $\infty$-category $\precat_{(\infty,n)}$ and the Yoneda embedding $\gaunt_n^{\omega} \hookrightarrow \precat_{(\infty,n)}$ therefore satisfies \hyperlink{axiom:C.3}{Axiom (C.3)}.

Moreover the collection $\R = \Upsilon_n$, $i=j$, $K=S$, and $K_0=S_0$, also satisfies \hyperlink{axiom:C.3-bis}{Axiom (C.3-bis)} by construction, whence $( S^{-1}\pre(\Upsilon_n), g)$ also satisfies  \hyperlink{axiom:C.3}{Axiom (C.3)}, where $g$ is the composite of the Yoneda embedding and $j^*$.

\begin{lemma} \label{lma:YonnedaFactors}
	The Yoneda embedding $\Upsilon_n \to \pre(\Upsilon_n)$ factors through a fully-faithful inclusion
	\begin{equation*}
\Upsilon_n \hookrightarrow \tau_{\leq 0}\pre(\Upsilon_n).
\end{equation*}
	This induces a fully-faithful nerve functor
	\begin{equation*}
g: \gaunt_n \hookrightarrow \tau_{\leq 0}S^{-1} \pre(\Upsilon_n).
\end{equation*}
\end{lemma}

\begin{proof}
	The 0-truncated objects of $\pre(\Upsilon_n)$ are precisely those presheaves of spaces taking values in the 0-truncated spaces, i.e., functors $\Upsilon_n^{\op}\to\set$. The 0-truncated objects of $\cat_{(\infty,n)} = S^{-1} \pre(\Upsilon_n)$ consist of precisely those 0-truncated objects of $\pre(\Upsilon_n)$ which are $S$-local. By Lemma \ref{lma:gauntIsLocal}, the nerve of every gaunt $n$-category is $S$-local, and so the result follows.
\end{proof}

The restriction of $g$ to $\gaunt_n^{\omega}$ is the composition of $j^{\ast}$ with the Yoneda embedding $\gaunt_n^{\omega} \hookrightarrow \pre(\gaunt_n^{\omega}) $. This fully faithful functor will provide the ``Basic Data'' for our axiomatization.

\begin{proposition} \label{prop-restriction-equiv}
	The functor $ j^{\ast} $ restricts to an equivalence $ \cat_{(\infty,n)} \simeq S^{-1}\pre(\Upsilon_n)$.
\end{proposition}

\begin{proof}
	Lemma~\ref{lem:fully-faithfulj} shows that the restriction of $j^{\ast}$ to $\cat_{(\infty,n)}$ is fully faithful.
		Now to prove that $j^{\ast}$ is essentially surjective when restricted to $\cat_{(\infty,n)}$, it suffices to prove that $S^{-1}\pre(\Upsilon_n)$ is generated under colimits by the cells.
	Since every object is a colimit of representables, it suffices to prove that $\Upsilon_n$ itself is generated under colimits in $S^{-1}\pre(\Upsilon_n)$ by the cells.
	To prove this, we filter $\Upsilon_n$ in the following manner.

	Let $\Upsilon_n^{(0)} = \G_n$  be the globular category of cells. For any $k\geq 1$,
	define $\Upsilon_n^{(k)}$ to be the full subcategory of $\Upsilon_n$ spanned by the set
	\begin{equation*}
		\left\{ X \in \Upsilon_n \; \middle| \; 
		\begin{aligned}
			& \textrm{there exists a colimit diagram }f\colon K^\rhd \to S^{-1}\pre(\Upsilon_n) \\
			& \textrm{such that }f(+\infty) \simeq X\textrm{ and }f(K)\subset\Upsilon_n^{(k-1)} 
		\end{aligned}
		 \right\}.
	\end{equation*}
	That is, $\Upsilon_n^{(k)}\subset\Upsilon_n$ consists of colimits, formed in $S^{-1}\pre(\Upsilon_n)$, of diagrams of objects of $\Upsilon_n^{(k-1)}$. 

	We claim that the collection $\{ \Upsilon_n^{(k)} \}$ forms an exhaustive filtration of $\Upsilon_n$, so that we have $\cup_k \Upsilon_n^{(k)} = \Upsilon_n$. 
	First we observe that the strongly saturated class $S$ contains the map 
	\begin{equation*}
		  \sigma( i(o_1)) \cup^{C_0} \sigma( i(o_2)) \cup^{C_0} \cdots  \cup^{C_0} \sigma( i(o_m)) \to i([m]; o_1, \dots, o_m)
	\end{equation*}
	and thus, by induction, the union $\cup_k \Upsilon_n^{(k)}$ contains $\Theta_n$. 
	
	It now suffices to show that this union is closed under fiber products over cells.
	Since colimits commute with fiber products over cells (both in $\gaunt_n$ and $S^{-1}\pre(\Upsilon_n)$) it is sufficient to show that $C_j \times_{C_i} C_k$ is contained in the union for all $i,j, k \leq n$.
	The fiber products of cells were analyzed in detail in Remark \ref{rmk:fiberprodofcells} and the proof of Lemma \ref{lma:fiberproductsareColimits}, where it was shown that they can all be obtained from the cells by a finite number of the colimits provided by $S_{00}$.
	These are colimits in $\cat_{(\infty,n)}$, whence the result follows.
\end{proof}

\begin{corollary}
	The right adjoint $ R \colon \precat_{(\infty,n)} \to \cat_{(\infty,n)} $ to the inclusion is identified with $j^{\ast}$ under the equivalence above.
	In particular, it admits both a left adjoint $L_Tj_!$ and a right adjoint $j_{\ast}$.
\end{corollary}

\begin{proof}
	Let $X \in \precat_{(\infty,n)}$ be an object, and consider the map $RX \to X$.
	The claim is that $j^*RX \to j^*X$ is an equivalence.
	Since the cells generate $S^{-1}\pre(\Upsilon_n)$ under colimits, it's enough to observe that $\map(C_i, RX) \simeq (RX)(C_i) \to X(C_i) \simeq Map(C_i, X)$ is an equivalence, for any cell $C_i$.
\end{proof}

\begin{corollary}
	The $\infty$-category $\cat_{(\infty,n)}$ is a localization of $\pre(\gaunt_n^{\omega})$, and so $(\cat_{(\infty,n)},g)$ satisfies \hyperlink{axiom:C.1}{Axiom (C.1)}.
\end{corollary}

By definition, $\cat_{(\infty,n)}$ is generated under colimits by the cells;
in other words, $(\cat_{(\infty,n)},g)$ satisfies \hyperlink{axiom:C.2}{Axiom (C.2)}.
 By Prop.~\ref{prop:Axiom4reduction} and Prop.~\ref{prop-restriction-equiv}, $(\cat_{(\infty,n)},g)$ satisfies \hyperlink{axiom:C.3}{Axiom (C.3)}. 
Also by construction, the pair $(\cat_{(\infty,n)},g)$ satisfies \hyperlink{axiom:C.4}{Axiom (C.4)}.
Finally, we now aim to prove that the pair $(\cat_{(\infty,n)},g)$ satisfies the final versality axiom.
Here is the key point.
\begin{proposition}
	Let $(\C,f)$ be a pair consisting of a presentable $\infty$-category $\C$ and a fully faithful functor $f \colon \gaunt_n^{\omega} \hookrightarrow \C $ for which Axioms \hyperlink{axiom:C.1}{(C.1)}, \hyperlink{axiom:C.3}{(C.3)}, and \hyperlink{axiom:C.4}{(C.4)} hold.
	Then there is a left adjoint $ K \colon \cat_{(\infty, n)} \to \C $ and a natural transformation $ \eta \colon K g \to f$ that the restriction $\eta|\G_n$ is an equivalence.
\end{proposition}

\begin{proof}
	By \hyperlink{axiom:C.1}{Axiom (C.1)}, the left Kan extension of $f$ along the Yoneda embedding is a localization $ F \colon \pre(\gaunt_n^{\omega}) \to \C $;
	the right adjoint is a fully faithful functor $ G \colon \C \hookrightarrow \pre(\gaunt_n^{\omega})$.
	Write $W$ for the class of morphisms of $\pre(\gaunt_n^{\omega})$ that are carried to equivalences of $\C$ by $F$, so that $\C \simeq W^{-1}\pre(\gaunt_n^{\omega})$.
	The class $W$ is strongly saturated, and by \cite[Proposition 5.5.4.16]{HTT}, it is of small generation.

	By \hyperlink{axiom:C.4}{Axiom (C.4)}, the class $W$ contains the morphisms of Notation \ref{Note:FundPushouts}.
	We claim further that $T_0 \subseteq W$.
	To prove this, it suffices to show that $W$ is stable under the operation $H\times_{C_i} (-) $ for any $H\in \gaunt_n^{\omega}$.
	
	So let $W' \subseteq W$ be the subset consisting of those morphisms $ \phi \colon X \to Y $ of $W$ such that for any morphism $ Y \to C_i $ and any morphism $H \to C_i$ of $\gaunt_n^{\omega}$, the pullback $H \times_{C_i}\phi \colon H\times_{C_i} X \to H\times_{C_i} Y$ also lies in $W$.
	By the universality of colimits in $\pre(\gaunt_n^{\omega})$, it follows that $W'$ is closed under colimits.
	From Proposition \ref{prop:Axiom4reduction} for $\R = \gaunt_n^{\omega}$, $i=\id$, and $U=W$, we deduce that since $(\C, f)$ satisfies \hyperlink{axiom:C.3}{Axiom (C.3)}, there is a subset $W_0 \subseteq W $ that generates $W$ under colimits and is stable under the operation $H\times_{C_i} (-) $ for any $H\in \gaunt_n^{\omega}$.
	Hence $W_0 \subseteq W'$, and so $W'=W$.

	Since $T_0 \subseteq W$ (and thus $T \subseteq W$), it follows that $F$ factors through a left adjoint $\precat_{(\infty,n)} \to \C$, which by a small abuse we will also call $F$.
	Composing this left adjoint with the fully faithful left adjoint $j_! \colon \cat_{(\infty,n)} \hookrightarrow \precat_{(\infty,n)} $, we obtain our desired left adjoint $K \coloneq F j_!$.

	To construct the desired natural transformation $\eta$, compose the counit $j_!j^* \to \id$ with $F$ to obtain $K j^* \to F$, and then restrict along Yoneda to obtain $ \eta \colon K g \to f $.
	By definition, $\eta$ is an equivalence when restricted to $\Upsilon_n$, and thus \emph{a fortiori} when restricted to $\G_n$.
\end{proof}

\begin{corollary}
	The pair $(\cat_{(\infty,n)},g)$ is a theory of $(\infty,n)$-categories, and so $\thy_{(\infty,n)}$ is nonempty.
\end{corollary}

\section{The connectedness of the space of theories} \label{sec:connected} \noindent In this section, we will prove:
\begin{theorem}[Versal is Universal]\label{theorem:thyisconnected} The moduli space of theories of $(\infty,n)$-categories $\thy_{(\infty,n)}$ is connected.
\end{theorem}

\noindent First we introduce a lemma.

\begin{lemma} \label{lma:stronggenretract} Suppose $\D$ a small quasicategory, and suppose $\C$ a locally small quasicategory that admits small colimits.
Suppose $g\colon \D\to \C$ a dense functor. For any quasicategory $\E$ admitting all small colimits, and let $\Fun^{\LL}(\C, \E)\subset\Fun(\C,\E)$ denote the full sub-quasicategory consisting of those functors that preserve small colimits.
Then the functor $\Fun(\C,\E)\to\Fun(\D,\E)$ induced by $g$ restricts to a fully faithful functor
\begin{equation*}
\Fun^{\LL}(\C,\E)\to\Fun(\D,\E).
\end{equation*}
\end{lemma}

\begin{proof} The $\infty$-category $\C$ is a localization of $\pre(\D)$, whence we obtain a fully faithful embedding $\Fun^{\LL}(\C, \E) \to \Fun^{\LL}(\pre(\D), \E)$.
Now by \cite[5.1.5.6]{HTT}, left Kan extension induces an equivalence $\Fun(\D,\E) \simeq \Fun^{\LL}(\pre(\D), \E)$.
See \cite[5.5.4.20]{HTT}. 
\end{proof}

\begin{proof}[Proof of Theorem~\ref{theorem:thyisconnected}]
	Suppose that $(\C,f)$ and $(\D,g)$ are each theories of $(\infty,n)$-categories; that is, they each satisfy axioms C.1-5.
		 By the versality axiom C.5 we have left adjoints
	\begin{equation*}
		L_1\colon \C \to \D \qquad \textrm{and} \qquad L_2\colon \D \to \C
	\end{equation*}
	and natural transformations $\eta_1\colon L_1 \circ f \to g$ and $\eta_2\colon L_2 \circ g \to f$ such that $\eta_i|_{\G_n}$ is an equivalence. Then the theorem follows provided  that we demonstrate that both $L_1 \circ L_2$ and $L_2 \circ L_1$ are autoequivalences. We will show this for $L_2 \circ L_1$. The argument for $L_1 \circ L_2$ is identical. 

Thus $E\coloneq L_2 \circ L_1: \C \to \C$ is a colimit preserving endofunctor along with a natural transformation $\eta_2\circ L_2(\eta_1)\colon E\circ f\to f$. Since $f$ is dense, Lemma~\ref{lma:stronggenretract} ensures that there is a natural transformation $\eta\colon E\to \id$ whose composition with $f$ is $\eta_2\circ L_2(\eta_1)$. To see that $\eta$ is an equivalence, let $\E\subseteq\C$ be the full subcategory spanned by those $X$ such that $\eta_X$ is an equivalence. Since $E$ preserves colimits, $\E$ is stable under colimits, and since $\eta_2\circ L_2(\eta_1)|_{\G_n}$ is an equivalence, it follows from (C.2) that $\E=\C$.	
\end{proof}


\section{The loopspace of the space of theories} \label{sec:loopspace} To complete the proof of Theorem \ref{theorem:unicity}, we now compute the loopspace of $\thy_{(\infty,n)}$ -- i.e., the space of autoequivalences of $\cat_{(\infty,n)}$. In this section we will prove:
\begin{theorem}\label{thm:SpaceOfTheoriesIsBZ/2} The full subcategory $\Aut(\cat_{(\infty,n)})\subset\Fun(\cat_{(\infty,n)},\cat_{(\infty,n)})$ spanned by the autoequivalences is the discrete set $(\Z/2)^n$.
\end{theorem}

We begin by analysing the subcategory $\tau_{\leq 0}\cat_{(\infty,n)}$ of $0$-truncated objects of $\cat_{(\infty,n)}$. We have already seen (Lemma \ref{lma:YonnedaFactors}) that $\gaunt_n$ embeds as a full subcategory of $\tau_{\leq0}\cat_{(\infty,n)}$. We now show that this embedding is an equivalence.

\begin{lemma} \label{cor:0-TruncatedAreGaunt}
	There is an identification $\tau_{\leq 0} \cat_{(\infty,n)} \simeq \gaunt_n$. In particular a presheaf of sets on $\Upsilon_n$ is isomorphic to the nerve of a gaunt $n$-category if and only if it is $S$-local. 
\end{lemma}

\begin{proof} The nerve of a gaunt $n$-category is $S$-local (cf. Lemma \ref{lma:gauntIsLocal}). Conversely, for any $X \in \tau_{\leq 0 }\cat_{(\infty,n)} \subseteq \Fun(\Upsilon_n^{\op}, \set)$, we may restrict to $\G_n$ to obtain a globular set $H_X$. For $0\leq j<i\leq n$, apply $X$ to the unique nondegenerate $i$-cell $\mu\colon C_i \to C_i \cup^{C_j} C_i$ connecting the initial and terminal vertices; this gives rise to the various compositions
\begin{equation*}
X(C_i)\times_{X(C_j)}X(C_i)\cong X(C_i\cup^{C_j}C_i)\to X(C_i).
\end{equation*}
By examining the maps
\begin{equation*}
X(C_i)\times_{X(C_j)}X(C_i)\times_{X(C_j)}X(C_i)\cong X(C_i\cup^{C_j}C_i\cup^{C_j}C_i)\to X(C_i)
\end{equation*}
corresponding to the unique nondegenerate $i$-cell
\begin{equation*}
C_i \to C_i \cup^{C_j} C_i\cup^{C_j} C_i
\end{equation*}
connecting the initial and terminal vertices, we find that these compositions are associative, and by examining the maps $X(C_j)\to X(C_i)$ induced by the nondegenerate cell $C_i\to C_j$, we find that these compositions are unital. From this we deduce that $H_X$ forms a strict $n$-category. Finally, since $X$ is local with respect to $K_k\to C_k$, it follows that $H_X$ is gaunt. Now map $A \to X$, with $A \in \Upsilon_n$ induces a map $A \to \nu H_X$, and hence we have a map $X \to \nu H_X$ in $\tau_{\leq 0 }\cat_{(\infty,n)}$. By construction this is a cellular equivalence, whence $X \simeq \nu H_X$.
\end{proof}

\begin{proposition} \label{prop:AutoGroup}
	Let $\C$ be a presentable $\infty$-category for which there exists an equivalence $\tau_{\leq0}\C\simeq\gaunt_n$. Assume that $(\tau_{\leq 0} \C)^{\omega}$ is dense in $\C$. Then the $\infty$-category $\Aut(\C)$ is equivalent to the (discrete) group $(\Z/2)^{n}$.
\end{proposition}

\begin{proof}
    We observe that the existence of an equivalence $\tau_{\leq0}\C\simeq\gaunt_n$, Lemma \ref{lma:uniqueauto}, Lemma \ref{lem:Endoofgaunt}, and 
	Corollary \ref{cor:Aut_kappa_Gaunt_n}
	guarantee that $\Aut((\tau_{\leq 0}\C)^\omega)$ in $\Fun((\tau_{\leq 0}\C)^\omega,(\tau_{\leq 0}\C)^\omega)$ is equivalent to the discrete group $(\mathbb{Z}/2)^{\times n}$. It therefore suffices to exhibit an equivalence of $\infty$-categories $\Aut(\C)\simeq\Aut((\tau_{\leq0}\C)^\omega)$.

    Clearly $\Aut(\C)$ is contained in the full subcategory $\Fun^{\LL}(\C,\C)\subset\Fun(\C,\C)$ spanned by those functors that preserve small colimits. Since $(\tau_{\leq 0}C)^\omega$ is dense in $\C$, it follows from Lemma~\ref{lma:stronggenretract} that the inclusion $(\tau_{\leq0}C)^\omega\hookrightarrow C$ induces a fully faithful functor
    \begin{equation*}
\Fun^{\LL}(\C,\C)\hookrightarrow \Fun((\tau_{\leq 0}\C)^\omega,\C).
\end{equation*}
    Moreover, any autoequivalence of $\C$ restricts to an autoequivalence of $\tau_{\leq0}\C$ and hence an autoequivalence of $(\tau_{\leq 0} \C)^\omega$. Thus restriction furnishes us with a fully faithful functor from $\Aut(\C)$ to $\Aut((\tau_{\leq0}\C)^\omega)\simeq (\Z/2)^{n}$.
    
    It remains to show that the restriction functor is essentially surjective. For this, suppose $(\tau_{\leq0}\C)^\omega\to(\tau_{\leq0}\C)^\omega$ an autoequivalence. One may form the left Kan extension $\Phi\colon \C\to \C$ of the composite
\begin{equation*}
\phi\colon(\tau_{\leq0}\C)^\omega\to(\tau_{\leq0}\C)^\omega \hookrightarrow \C
\end{equation*}
    along the inclusion $(\tau_{\leq 0}\C)^\omega\hookrightarrow \C$. One sees immediately that $\Phi$ is an equivalence, and moreover its restriction to $(\tau_{\leq0}\C)^\omega$ coincides with $\phi$.
\end{proof}

Corollary \ref{cor:0-TruncatedAreGaunt} provides an identification $\tau_{\leq 0} \cat_{(\infty,n)} \simeq \gaunt_n$, and so Theorem \ref{thm:SpaceOfTheoriesIsBZ/2} now follows from Proposition \ref{prop:AutoGroup}.


\part{A recognition principle for categorical presentations}\label{part:examples}


\section{Presentations of $(\infty,n)$-categories} \label{sect:presheaves} The best-known examples of theories of $(\infty,n)$-categories are given by presentations in terms of generators and relations. In order to show that these examples also satisfy our axioms, we can compare them directly to our colossal model $\cat_{(\infty,n)}$ of $(\infty,n)$-categories. The main point is that $\Upsilon_n$ is so large that any `reasonable' set of generators can be compared to it. 

\begin{notation} Suppose $\R$ an ordinary category, and suppose $i\colon\R\to\Upsilon_n$ a functor. Suppose $T_0$ a set of morphisms of $\pre(\R)$, and write $T$ for the strongly saturated class of morphisms of $\pre(\R)$ it generates.
\end{notation}

\begin{theorem} \label{Thm:RecognitionForPresheaves}
	Suppose that the following conditions are satisfied.
	\begin{enumerate}[{\hspace{\parindent}(R.}1{)}]
		\item $i^*(S_0) \subseteq T$.
		\item $i_!(T_0) \subseteq S$.
		\item Any counit $R \to i^* i_!(R) = i^*(i(R))$ is in $T$ for any $R \in \R$.
		\item For each $0 \leq k \leq n$, there exists an object $R_k\in\R$ such that $i(R_k) \cong C_k \in \Upsilon_n$.
	\end{enumerate}
	Then $i^*\colon \cat_{(\infty,n)} \to T^{-1} \pre(\R)$ is an equivalence, and $T^{-1} \pre(\R)$ is a theory of $(\infty,n)$-categories. 
\end{theorem}

\begin{proof} Condition (R.1) implies both that $i_*$ carries $T$-local objects to $S$-local objects and that we obtain an adjunction:
	\begin{equation*}
		L^{T} \circ i^*\colon S^{-1} \pre(\Upsilon_n)  \rightleftarrows T^{-1} \pre(\R)\colon i_*. 
	\end{equation*}
	Similarly, condition (R.2) implies that $i^*$ carries $S$-local objects to $T$-local objects and that we obtain a second adjunction:
	\begin{equation*}
		L^{S} \circ  i_!\colon T^{-1} \pre(\R)  \rightleftarrows S^{-1} \pre(\Upsilon_n)\colon i^* .
	\end{equation*}
	Since $i^*$ sends $S$-local objects to $T$-local objects, $i^* \simeq L^{T} \circ i^*$ when restricted to the $S$-local objects of $\pre(\Upsilon_n)$. Thus $i^*\colon \pre(\Upsilon_n) \to \pre(\R)$ restricts to a functor
	\begin{equation*}
		i^*\colon \cat_{(\infty,n)} = S^{-1} \pre(\Upsilon_n) \to T^{-1} \pre(\R),
	\end{equation*}
that admits a left adjoint $L^S \circ i_!$ and a right adjoint $i_*$.

		Notice that $i_!(R) \cong i(R)$ in $\pre(\Upsilon_n)$, where we have identified $\R$ and $\Upsilon_n$ with their images under the Yoneda embedding in, respectively, $\pre(\R)$ and $\pre(\Upsilon_n)$. Thus by (R.3) the counit map applied to $r \in \R$,
\begin{equation*}
	R \to i^* \circ L^S \circ i_!(R) \cong i^* \circ L^S i(R)  \cong i^* i(R)
\end{equation*}	
 becomes an equivalence in $T^{-1} \pre(\R)$ (the last equality follows from Lemma \ref{lma:YonnedaFactors}, as the image of $i$ consists of $S$-local objects). 
The endofunctor $i^* \circ L^S \circ i_!\colon T^{-1} \pre(\R) \to T^{-1} \pre(\R)$ is a composite of left adjoints, hence commutes with colimits. Therefore, as $\R $ is dense in $T^{-1} \pre(\R)$, the functor $i^* \circ L^S \circ i_!$ is determined by its restriction to $\R$. It is equivalent the left Kan extension of its restriction to $\R$. Consequently $i^* \circ L^S \circ i_!$ is equivalent to the identity functor. 

On the other hand, for each $X \in \cat_{(\infty,n)}$, consider the other counit map $X \to i_* i^*X$. 
For each $k$, we have natural equivalences,
\begin{align*}
	\map(C_k, i_* i^* X) &\simeq \map( i(R_k), i_* i^* X) \\
	&  \simeq \map( i^* i(R_k), i^*X)  \\
	& \simeq \map(R_k, i^* X) \\
	& \simeq \map(i(R_k), X) \simeq \map(C_k, X),
\end{align*}
which follow from (R.3), (R.4), the identity $i_* (R_k) \cong i(R_k)$, and the fact that $i^*X$ is $T$-local. By Remark \ref{rmk:CellsDetect} this implies that the counit $X \to i_* i^*X$ is an equivalence. Thus $i^*$ is a functor with both a left and right inverse, hence is itself an equivalence of $\infty$-categories. 
\end{proof}

\begin{remark} Note that if the functor $i\colon \R \to \Upsilon_n$ is fully-faithful, then condition (R.3) is automatic. Note also that (R.3) and (R.4) together imply that the presheaves $i^*(C_k)$ on $\R$ are each $T$-equivalent to representables $R_k$.
\end{remark}

Condition (R.1) appears to be the most difficult to verify in practice. Heuristically, it states that $T$ contains \emph{enough morphisms}. To verify it, it will be convenient to subdivide it into two conditions.
\begin{lemma} \label{lma:NewR1}
	Condition (R.1) is implied by the conjunction of the following.
\begin{enumerate}[{\hspace{\parindent}(R.1-bis(}a{))}]
\item $i^*(S_{00})\subset T$.
\item For any morphism $U'\to V'$ of $T_0$, and for any morphisms $V'\to i^*(C_i)$ and $H\to C_i$ with $H \in \Upsilon_n$, the pullback
\begin{equation*}
U'\times_{i^*(C_i)}i^*H\to V'\times_{i^*C_i}i^*H
\end{equation*}
lies in $T$.
\end{enumerate}
\end{lemma}

\begin{proof} First, consider the subclass $T'\subset T$ containing those morphisms $U'\to V'$ of $T$ such that for any nondegenerate morphisms $V'\to C_k$ and $H\to C_k$, the pullback
\begin{equation*}
U'\times_{C_k}H \to V'\times_{C_k}H
\end{equation*}
lies in $T$. Since colimits in $\pre(\R)$ are universal, one deduces immediately that the class $T'$ is strongly saturated. Hence (R.1-bis(b)) implies that $T'=T$. Thus $T$ is closed under pullbacks along morphisms $H\to C_k$ and contains $i^*S_{00}$ (by (R.1-bis(a))), hence contains all of $i^*S_0$. 
%
\end{proof}

There are two main examples to which we shall apply Th.~\ref{Thm:RecognitionForPresheaves}: Rezk's model of complete Segal $\Theta_n$-spaces [\S~\ref{sect:Rezksmodel}] and the model of $n$-fold complete Segal spaces [\S~\ref{sect:CSSnmodel}].


\section{Strict $n$-categories as presheaves of sets} A category internal to an ordinary category $\D$ may be described as a simplicial object in $\D$, that is a $\D$-valued presheaf $C\colon\Delta^\textrm{op} \to \D$, which satisfies the following {\em strong Segal conditions}. For any nonnegative integer $m$ and any integer $1\leq k\leq m-1$, the following square is a pullback square:
 \begin{center}
 \begin{tikzpicture}
 	\node (LT) at (0, 1) {$C([m])$};
 	\node (LB) at (0, 0) {$C(\{0,1,\dots,k\})$};
 	\node (RT) at (3, 1) {$C(\{k,k+1,\dots,m\})$};
 	\node (RB) at (3, 0) {$C(\{k\})$.};
 	\draw [->] (LT) -- node [left] {$$} (LB);
 	\draw [->] (LT) -- node [above] {$$} (RT);
 	\draw [->] (RT) -- node [right] {$$} (RB);
 	\draw [->] (LB) -- node [below] {$$} (RB);
 	\node at (0.5, 0.5) {$\ulcorner$};
 \end{tikzpicture}
 \end{center}
Thus a strict $n$-category consists of a presheaf of sets on the category $\Delta^{\!\times n}$ which satisfies the Segal condition in each factor and further satisfies a globularity condition. Equivalently a strict $n$-category is a presheaf of sets on $\Delta^{\!\times n}$ which is local with respect to the classes of maps $\mathrm{Segal}_{\Delta^{\!\times n}}$ and $\mathrm{Glob}_{\Delta^{\!\times n}}$ defined below. 
\begin{notation}
	Objects of $\Delta^{\!\times n}$ will be denoted $\mathbf{m} = ([m_k])_{k=1, \dots, n}$. Let
	\begin{equation*}
j: \Delta^{\!\times n} \to \Fun((\Delta^{\!\times n})^{\op}, \set)
\end{equation*}
	denote the Yoneda embedding. Let 
	\begin{equation*}
		\boxtimes\colon \Fun(\Delta^{\op}, \set) \times \Fun((\Delta^{\!\times n-1})^{\op}, \set)   \to \Fun((\Delta^{\!\times n})^{\op},\set)
	\end{equation*} 
	be the essentially unique functor that preserves colimits separately in each variable and sends $(j[k], j(\mathbf{m}))$ to $j([k], \mathbf{m})$. Let $\mathrm{Segal}_{\Delta}$ denote the collection of maps that corepresent the Segal squares:
\begin{equation*}
\mathrm{Segal}_{\Delta} = \{ j{\{0,1,\dots,k\}} \cup^{j{\{k\}}} j{\{k,k+1,\dots,m\}} \to j[m] \ |\ 1\leq k\leq m-1\}
\end{equation*}
and inductively define
\begin{equation*}
	\mathrm{Segal}_{\Delta^{\!\times n}} = \{ \mathrm{Segal}_{\Delta} \boxtimes j(\mathbf{m}) \ | \ \mathbf{m} \in \Delta^{\!\times n-1} \} \cup \{ j[k] \boxtimes \mathrm{Segal}_{\Delta^{\!\times n-1}} \ |\ [k] \in \Delta \}.
\end{equation*}
Moreover for each $\mathbf{m} \in \Delta^{\times n}$, let $\widehat{\mathbf{m}} = ([\widehat{m}_j])_{1\leq j \leq n}$ be defined by the formula
\begin{equation*}
   [\widehat{m}_j]  = \begin{cases} [0] & \textrm{if there exists } i \leq j \textrm{ with } [m_i] = [0], \textrm{ and} \\
[m_j] & \textrm{else} \end{cases},
\end{equation*}
and let
\begin{equation*}
\mathrm{Glob}_{\Delta^{\times n}} = \{ j(\mathbf{m}) \to j(\widehat{\mathbf{m}})  \; | \; \mathbf{m} \in \Delta^{\times n} \}.
\end{equation*}
The presheaf underlying a strict $n$-category $C$ will be called its {\em nerve} $\nu C$.  
\end{notation}

Strict $n$-categories may also be described as certain presheaves on the category $\Theta_n$, the opposite of Joyal's category of $n$-disks (Definition \ref{dfn:wreathproduct}).

As $i\colon\Theta_n\to\cat_n$ is a dense functor, the corresponding nerve functor $\nu: \cat_n \to \Fun(\Theta_n^\text{op}, \set)$ is fully-faithful. The essential image consists of precisely those presheaves which are local with respect to the class of maps $\mathrm{Segal}_{\Theta_n}$ defined inductively to be the union of $\sigma \mathrm{Segal}_{\Theta_{n-1}}$ and the following: 
\begin{align*}
	\left\{ 
	\begin{aligned}
		& j({\{0, \dots, k\}}; o_1, \dots, o_k) \cup^{j({\{ k \}})} j({\{ k, \dots, m\}}; o_{k+1}, \dots, o_{m}) \to j({[m]}; o_1, \dots, o_m) \\
		& 0 \leq k \leq m, \quad o_i \in \Theta_{n-1}
	\end{aligned}
		 \right\}.
\end{align*}
We will call this latter class $\mathrm{Se}_{\Theta_n}$ for later reference.

\begin{notation} \label{ntn:simplicialsetK}
	Let $K$ denote the simplicial set
\begin{equation*}
\Delta^3\cup^{(\Delta^{\{0,2\}}\sqcup\Delta^{\{1,3\}})}(\Delta^0\sqcup\Delta^0)
\end{equation*}
obtained by contracting two edges in the three simplex.

Rezk observed \cite[\S~10]{Rezk} that $K$ detects equivalences in nerves of categories, and consequently it may be used to formulate his completeness criterion. We shall use it to identify the gaunt $n$-categories. To this end 
 set
\begin{align*}
		\mathrm{Comp}_\Delta &= \{ K \to j[0] \} \\
		\mathrm{Comp}_{\Delta^{\!\times n}} & = \{ \mathrm{Comp}_{\Delta} \boxtimes j(\mathbf{0}) \} \cup \{j[k] \boxtimes \mathrm{Comp}_{\Delta^{\!\times n-1}} \} \\
		\mathrm{Comp}_{\Theta_n} &= \iota_{!}\mathrm{Comp}_\Delta \cup \sigma_! \mathrm{Comp}_{\Theta_{n-1}} 
		. 
\end{align*}
where
\begin{equation*}
\iota_!\colon \Fun(\Delta^\textrm{op}, \set) \to \Fun(\Theta_n^\textrm{op}, \set)\textrm{\quad and\quad}\sigma_!\colon\Fun(\Theta_{n-1}^\textrm{op}, \set) \to \Fun(\Theta_n^\textrm{op}, \set)
\end{equation*}
are given by left Kan extension along $\iota$ and $\sigma$, respectively. 
\end{notation}

\begin{corollary} \label{cor:GauntncatsAsPresheavesOnDeltaN}
	  A presheaf of sets on $\Delta^{\!\times n}$ is isomorphic to the nerve of a gaunt $n$-category if and only if it is local with respect to the classes $\mathrm{Segal}_{\Delta^{\!\times n}}$, $\mathrm{Glob}_{\Delta^{\!\times n}}$, and $\mathrm{Comp}_{\Delta^{\!\times n}}$. 	  A presheaf of sets on $\Theta_n$ is isomorphic to the nerve of a gaunt $n$-category if and only if it is local with respect to the classes $\mathrm{Segal}_{\Theta_n}$ and 	$\mathrm{Comp}_{\Theta_n}$.
%
\end{corollary}


\begin{proof}
	Being local with respect $\mathrm{Segal}_{\Delta^{\!\times n}}$ and $\mathrm{Glob}_{\Delta^{\!\times n}}$ (or to $\mathrm{Segal}_{\Theta_n}$ for $\Theta_n$-presheaves) implies that the presheaf is the nerve of a strict $n$-category. Such an $n$-category is gaunt if and only if it is local with respect to the morphisms $\sigma^k(E) \to \sigma^k(C_0)$. This last follows from locality with respect to $\mathrm{Comp}_{\Delta^{\!\times n}}$ (or, respectively,	with respect to $\mathrm{Comp}_{\Theta_n}$) because the square
	\begin{center}
	\begin{tikzpicture}
		\node (LT) at (0, 1.5) {$[1] \sqcup [1]$};
		\node (LB) at (0, 0) {$[0] \sqcup [0]$};
		\node (RT) at (4, 1.5) {$[3]$};
		\node (RB) at (4, 0) {$E$};
		\draw [->] (LT) -- node [left] {} (LB);
		\draw [right hook->] (LT) -- node [above] {$i_{0,2} \sqcup i_{1,3}$} (RT);
		\draw [->] (RT) -- node [right] {} (RB);
		\draw [->] (LB) -- node [below] {} (RB);
		\node at (3.75, 0.25) {$\lrcorner$};
	\end{tikzpicture}
	\end{center}
is a pushout square of strict $n$-categories.
\end{proof}

The description of $\Theta_n$ as an iterated wreath product gives rise to a canonical functor
$\delta_n\colon\Delta^{\!\times n}\to\Theta_n$, described in \cite[Definition 3.8]{MR2331244}, which sends 
 $[k_1] \times [k_2] \times \cdots \times [k_n]$ to the object
\begin{equation*}
	([k_n]; \underbrace{i_{\Delta^{\times (n-1)}}([k_1] \times \cdots \times [k_{n-1}]), \dots, i_{\Delta^{\times (n-1)}}([k_1] \times \cdots \times [k_{n-1}])}_{k_n \textrm{ times}}). 	
\end{equation*}
This object may be thought of as generated by a $k_1 \times k_2 \times \cdots \times k_n$ grid of cells. If $X$ is a strict $n$-category then its nerve $\nu_{\Delta^{\times n}} X$ in $\Fun((\Delta^{\times n})^\op, \set)$ is obtained by the formula $\delta_n^* \nu_{\Theta_n} X$, where $\nu_{\Theta_n}$ is the nerve induced from the inclusion $\Theta_n \to \cat_n$ \cite[Proposition 3.9, Rk.~3.12]{MR2331244}.  

\begin{proposition} \label{prop:ThetaNSegalAreRetractsOfDeltaNSegal}
	Joyal's category $\Theta_n$ is the smallest full subcategory of $\gaunt_n$ containing the grids (the full subcategory of $\Theta_n$ spanned by the image of $\delta_n$) and closed under retracts. Furthermore, the  morphisms in the set $\mathrm{Segal}_{\Theta_n}$ may be obtained as retracts of the set $(\delta_n)_!(\mathrm{Segal}_{\Delta^{\times n}})$.
\end{proposition}

\begin{proof}
	Both statements follow by induction. First note that $\Theta_n$ itself is closed under retracts \cite[Proposition 3.14]{MR2331244}. In the base case, the category $\Theta_1 = \Delta$ consists of precisely the grids, and the sets of morphisms agree $\mathrm{Segal}_{\Theta_1} = \mathrm{Segal}_{\Delta}$. Now assume, by induction, that every object of $o \in \Theta_{n-1}$ is the retract of a grid $\delta_n( \mathbf{m}^o)$, for some object $\mathbf{m}^o = [m_1^o] \times \cdots \times [m_{n-1}^o] \in \Delta^{\times n}$. In fact, given any finite collection of objects  $\{ o_i \in \Theta_{n-1} \}$ they may be obtained as the retract of a single grid. This grid may be obtained as the image of $\mathbf{k} = [k_1] \times \cdots \times [k_{n-1}]$, where $k_j$ is the maximum of the collection $\{m_j^{o_i}\}$.
It now follows easily that the object $([n]; o_1, \dots, o_i) \in \Theta_n$ is a retract of the grid coming from the object $[n] \times \mathbf{k}$. 

To prove the second statement we note that there are two types of maps in $\mathrm{Segal}_{\Theta_n}$, those in $\sigma_!(\mathrm{Segal}_{\Theta_{n-1}})$ and the maps
\begin{equation*}
	j({\{0, \dots, k\}}; o_1, \dots, o_k) \cup^{j({\{ k \}})} j({\{ k, \dots, m\}}; o_{k+1}, \dots, o_{m}) \to j({[m]}; o_1, \dots, o_m)
\end{equation*}
for $ 0 \leq k \leq m$ and $o_i \in \Theta_{n-1}$. This later map is a retract of the image under $(\delta_n)_!$ of the map 
\begin{equation*}
	\left(j{\{0,1,\dots,k\}} \cup^{j{\{k\}}} j{\{k,\dots,m\}} \to j[m]\right) \boxtimes j(\mathbf{m})
\end{equation*}
which is a map in $\mathrm{Segal}_{\Delta^{\times n}}$. Here $\mathbf{m}$ is such that $({\{0, \dots, k\}}; o_1, \dots, o_k)$ is the retract of the grid corresponding to $[k] \times \mathbf{m}$.

The former class of morphisms in $\mathrm{Segal}_{\Theta_n}$, those in $\sigma_!(\mathrm{Segal}_{\Theta_{n-1}})$, are also retracts on elements in $\mathrm{Segal}_{\Delta^{\times n}}$. Specifically, if $\sigma_!(f) \in \sigma_!(\mathrm{Segal}_{\Theta_{n-1}})$, then by induction $f$ is the retract of $(\delta_{n-1})_!(g)$ for some $g \in \mathrm{Segal}_{\Delta^{\times n-1}}$. One may then readily check that $\sigma(f)$ is  the retract of $j[1] \boxtimes g \in \mathrm{Segal}_{\Delta^{\times n}}$. 
\end{proof}





\section{Rezk's complete Segal $\Theta_n$-spaces form a theory of $(\infty,n)$-categories}\label{sect:Rezksmodel}

Here we consider Joyal's full subcategory $\Theta_n$ of $\cat_n$ \cite{joyaldisks,MR1916373,MR2331244}; write $i\colon\Theta_n\to \Upsilon_n$ for the inclusion functor. Rezk \cite[11.4]{Rezk} identifies the set of morphisms $\mathcal{T}_{n,\infty}$ of $\pre(\Theta_n)$ consisting of the union of  $\mathrm{Segal}_{\Theta_n}$ and $\mathrm{Comp}_{\Theta_n}$\footnote{Rezk use a slightly different generating set based on the full decomposition of $[n]$ as the union
\begin{equation*}
	[1] \cup^{[0]} [1] \cup^{[0]} \cdots \cup^{[0]} [1].
\end{equation*}
Both Rezk's set of generators and the union  $\mathrm{Segal}_{\Theta_n} \cup \mathrm{Comp}_{\Theta_n}$ are readily seen to produce the same saturated class $T_{\Theta_n}$. We find it slightly more convenient to use the later class of generators.
}. Let us write $T_{\Theta_n}$ for the saturated class generated by $\mathcal{T}_{n,\infty}$, and let us write $\CSS(\Theta_n)$ for the localization $T_{\Theta_n}^{-1}\pre(\Theta_n)$. We now show that $\CSS(\Theta_n)$ is a theory of $(\infty,n)$-categories.

\begin{remark} \label{rk:CSSThetaInfty=Modelcat} It follows from \cite[A.3.7.3]{HTT} that $\CSS(\Theta_n)$ is canonically equivalent to the simplicial nerve of the cofibrant-fibrant objects in the simplicial model category $\Theta_n\mathrm{Sp}_{\infty}$ considered by Rezk --- i.e., the left Bousfield localization of the injective model category of simplicial presheaves on $\Theta_n$ with respect to the set $\mathcal{T}_{n,\infty}$.
\end{remark}

\begin{lemma} \label{lma:TThetanContainsS00}
	The saturated class $T_{\Theta_n}$ contains the set $i^*(S_{00})$.
\end{lemma}

\begin{proof}
	The set $S_{00}$ consists of the union of four subsets of maps, corresponding to the four families of fundamental pushouts of types (a), (b), (c), and (d) in Axiom (C.3). The second and last subsets corresponding to the (b) and (d) families pullback to morphisms which are contained in the generating set of $T_{\Theta_n}$. Thus it remains to prove the that the same holds for the remaining families (a) and (c). In particular, we wish to show that for each $0 \leq i \leq n$, each $0 \leq j, k \leq n-i$, and every nondegenerate morphism $C_{i+j}\to C_i$ and $C_{i+k}\to C_i$, the natural morphism
\begin{align} \label{eqn:TypeCPushoutForThetaN}
f(C_{i + j} \cup^{C_i} C_{i+k}))  \cup^{f(\sigma^{i+1}(C_{j-1} \times C_{k-1}))} (f(C_{i + k} \cup^{C_i} & C_{i +j}) \\
 & \to f(C_{i+j} \times_{C_i} C_{i + k}), \notag
\end{align}
is contained in $T_{\Theta_n}$ where the pushout is formed as in Notation \ref{Note:FundPushouts}.
	
In fact a stronger statement holds (cf. \cite[Proposition 4.9]{Rezk}). For each object $o \in \Theta_n$ we have a natural bijection of sets
\begin{equation*}
	\hom(o, C_{i+j} \times_{C_i} C_{i + k}) \cong \hom(o, C_{i + j} \cup^{C_i} C_{i+k})) \cup^{\hom(o, C_{i + m})} \hom(o, C_{i + k} \cup^{C_i} C_{i +j}).
\end{equation*}	
Thus Eq.~\ref{eqn:TypeCPushoutForThetaN} is in fact an equivalence in the presheaf category $\pre(\Theta_n)$. In particular the family (c) pulls back to a family of equivalences, which are hence contained in $T_{\Theta_n}$. An virtually identical argument applies the family (a), which also consists of morphisms pulling back to equivalences of presheaves.  	
\end{proof}

The functor $\sigma: \Theta_{n-1} \to \Theta_n$ gives rise to a functor $\sigma_{!}: \pre(\Theta_{n-1}) \to \pre(\Theta_n)$, left adjoint to $\sigma^*$. The classes $\mathrm{Segal}_{\Theta_n}$ and $\mathrm{Comp}_{\Theta_n}$ are defined inductively using the 1-categorical analog of $\sigma_!$, but may also be defined using $\sigma_!$. We therefore collect some relevant properties of this functor in the next two lemmas.

\begin{lemma} \label{lma:propertyOfSigma}
	The functor $\sigma_!$ preserves both pushouts and pullbacks, sends $T_{\Theta_{n-1}}$-local objects to $T_{\Theta_n}$-local objects, and satisfies $\sigma_!(T_{\Theta_{n-1}}) \subseteq T_{\Theta_{n}}$.
\end{lemma}  

\begin{proof}
The functor $\sigma_!$ is a left adjoint, hence it preserves all colimits in $\pre(\Theta_n)$, in particular pushouts. Moreover, $\sigma_!$ sends the generators of $T_{\Theta_{n-1}}$ to generators of $T_{\Theta_{n}}$. Together these imply the containment $\sigma_!(T_{\Theta_{n-1}}) \subseteq T_{\Theta_{n}}$. Direct computations, which we leave to the reader, show that $\sigma_!$ sends $T_{\Theta_{n-1}}$-local objects to $T_{\Theta_n}$-local objects and that the following formula holds,
\begin{equation*}
	\hom( ([n]; o_1, \dots, o_n), \sigma(X \times_Y Z)) = \coprod_{i_k: [n] \to [1] \atop 0 \leq k \leq n+1} \map(o_k, X \times_Y Z).
\end{equation*}
where $i_k: [n] \to [1]$ maps $i \in [n]$ to $0 \in [1]$ if $i < k$ and to $1 \in[1]$ otherwise. 
In the above formula when $k=0$ or $n+1$ the space $\map(o_k, X \times_Y Z)$ is interpreted as a singleton space. From this it follows that $\sigma$ preserves fiber products. 
\end{proof}

\begin{remark} \label{lma:NonDegenFromSuspension=Suspension}
If $V \in \Theta_n$ is of the form $V = \sigma(W) = ([1],W)$ for some $W \in \Theta_{n-1}$, then any nondegenerate morphism $f \colon V \to C_i = ([1],C_{i-1})$ is of the form $f = \sigma(g)$ for a unique nondegenerate $g \colon W \to C_{i-1}$. 

More generally, the $\Theta$-construction gives rise, for each $[m] \in \Delta$, to functors
\begin{align*}
	\sigma^{[m]} \colon   \Theta_{n-1}^{\times m} & \to \Theta_n \\
	 (o_1, \dots, o_m) &\mapsto ({[m]}; o_1, \dots, o_m).
\end{align*}
In the case of $\sigma^{[1]} = \sigma$, we obtain functors
\begin{equation*}
	\sigma^{[m]}_! \colon   \pre(\Theta_{n-1})^{\times m} \to \pre(\Theta_n)
\end{equation*}
by left Kan extension in each variable. These functors were also considered by Rezk \cite[\S~4.4]{Rezk}, and we adopt a similar notation: $\sigma^{[m]}_!(X_1, \dots, X_m) = ({[m]}; X_1, \dots, X_m)$.
\end{remark}

\bigskip

The following lemma is a result of \cite[Proposition 6.4]{Rezk},
but the proof given there (even in the corrected version) relies on the false proposition \cite[Proposition 2.19]{Rezk}.
However it is straightforward to supply an alternate proof (along the lines of \cite[Proposition 5.3]{Rezk}).

\begin{lemma} \label{lma:MoremapsinTTheta}
	Let $b_1, \dots, b_p$ be elements of $\pre(\Theta_{n-1})$, and let $0 \leq r \leq s \leq p$.
	Let $A$ and $B$ be defined as follows:
	\begin{align*}
		A &= \sigma_!^{\{0, \dots, s\}}(b_1, \dots, b_s) \cup^{\sigma^{\{r, \dots, s\}}_!(b_{r+1}, \dots, b_s)} \sigma_!^{\{r, \dots, p\}}(b_{r+1}, \dots, b_p), \text{ and} \\
		B &= \sigma_!^{[p]}(b_1, \dots, b_p).
	\end{align*}
	 Then the natural map $A \to B$ is in $T_{\Theta_n}$.
\end{lemma}
\begin{proof}
First note that if each of the $b_i$ were a representable presheaf, then the map $A \to B$ may be written as a pushout of the generating morphism $\mathcal{T}_{n,\infty}$, hence is manifestly an element of $T_{\Theta_n}$ (we leave this as an exercise). The general case, however, reduces to this case as every presheaf is (canonically) a colimit of representables, the functors $\sigma^{[\ell]}_!$ commute with these colimits separately in each variable, and $T_{\Theta_n}$, being a saturated class, is closed under colimits. 
\end{proof}

\begin{notation}
	We now define three additional classes of morphisms of $\pre(\Theta_n)$. Let $J_a$ be the set of all morphisms $H\to C_i$ ($0\leq i\leq n$) of $\Upsilon_n$; let $J_b$ be the set of all nondegenerate morphisms $H\to C_i$ ($0\leq i\leq n$) of $\Theta_n$; and let $J_c$ be set of all inclusions $C_j\hookrightarrow C_i$ ($0\leq j\leq i\leq n$) of $\mathbb{G}_n$. Now, for $x\in\{a,b,c\}$, set:
\begin{equation*}
	T_{\Theta_n}^{(x)}\coloneq \left\{ [f\colon U\to V] \in T_{\Theta_n} \; \middle| \; 
	\begin{aligned}
		&\textrm{for any }[H\to C_i]\in J_x\textrm{ and }[V\to C_i]\in\pre(\Theta_n), \\
		&\textrm{one has }f\times_{C_i}\nu H \in T_{\Theta_n}
	\end{aligned}
 \right\}
\end{equation*}
\end{notation}

\begin{lemma} \label{lma:TThetaX_strongsat}
	Each of the three classes $T_{\Theta_n}^{(x)}$ ($x\in\{a,b,c\}$) is a strongly saturated class. 
\end{lemma}

\begin{proof}
	As colimits are universal in $\pre(\Theta_n)$, the functors $(-) \times_{C_i} \nu H$ preserves all small colimits. Thus the class $[(-)\times_{C_i} \nu H]^{-1}(T_{\Theta_n})$ is a saturated class in $\pre(\Theta_n)$. Taking appropriate intersections of these classes and $T_{\Theta_n}$ yields the three classes in question. 
\end{proof}

We aim to show that the $\infty$-category $\CSS(\Theta_n)$ is a theory of $(\infty,n)$-categories. For this we need to prove Axioms (R.1-4) of Th.~\ref{Thm:RecognitionForPresheaves}. The most difficult property, (R.1), would follow from Lemma \ref{lma:NewR1} and Lemma \ref{lma:TThetanContainsS00} if we also knew the identity $T_{\Theta_n} = T^{(a)}_{\Theta_n}$. As these are saturated classes and $T^{(a)}_{\Theta_n} \subseteq T_{\Theta_n}$, it is enough to show that the generators $\mathrm{Segal}_{\Theta_n}$ and $\mathrm{Comp}_{\Theta_n}$ of $T_{\Theta_n}$ are contained in $T_{\Theta_n}^{(a)}$. We will ultimately prove this by an inductive argument, but first we need some preliminaries.

First, we note the following.
\begin{lemma} \label{lma:TC=T}
	One has $T_{\Theta_n} = T_{\Theta_n}^{(c)}$.
\end{lemma}

\begin{proof}
By Lemma \ref{lma:TThetaX_strongsat}, it is enough to check that the generators of $\mathrm{Segal}_{\Theta_n}$ and $\mathrm{Comp}_{\Theta_n}$ of $T_{\Theta_n}$ are contained in $T_{\Theta_n}^{(c)}$. For that assume that $[U \to V]$ is a generator of $T_{\Theta_n}$ and let $C_j \hookrightarrow C_i$ be an inclusion. We wish to demonstrate that for all $V \to C_i$ we have that
\begin{equation} \label{eqn:fiberproductofgen}
	U \times_{C_i} C_j \to V \times_{C_i} C_j
\end{equation}
is in $T_{\Theta_n}$. 
Recall that
\begin{equation*}
\mathrm{Comp}_{\Theta_n} = \iota_{!}\mathrm{Comp}_\Delta \cup \sigma_! \mathrm{Comp}_{\Theta_{n-1}}\textrm{\quad and\quad }\mathrm{Segal}_{\Theta_n} = \mathrm{Segal}_{\Theta_{n}} \cup \sigma_{!} \mathrm{Segal}_{\Theta_{n-1}}.
\end{equation*}
There are several cases:
\begin{enumerate}
	\item $i=j$. In this trivial case (\ref{eqn:fiberproductofgen}) reduces to $[U \to V] \in T_{\Theta_n}$.
	\item The morphism $V \to C_i$ factors as $V \to C_0 \to C_i$. In this case the fiber product $C_j \times_{C_i} C_0$ is either $C_0$ or empty. In the latter case (\ref{eqn:fiberproductofgen}) is an isomorphism, and in the former case it is $[U \to V] \in T_{\Theta_n}$. Notice that this case covers $\iota_{!}\mathrm{Comp}_\Delta$. 
	\item $j=0< i$ and the map $V \to C_i$ does not factor through $C_0$. In this case it follows that $[U \to V]$ is not in  $\iota_{!}\mathrm{Comp}_\Delta$, and hence 
	\begin{equation*}
		V = ([m]; o_1, ..., o_m)
	\end{equation*}
	is representable with $m \neq 0$. Moreover the map
	\begin{equation*}
V \to C_i = ([1]; C_{i-1})
\end{equation*}
	consists of a surjective map $[m] \to [1]$ (which specifies a unique $0< k \leq m$; the inverse image of $0 \in [1]$ consists of all elements strictly less than $k$) together with map $o_k \to C_{i-1}$. A direct calculation shows that in this situation (\ref{eqn:fiberproductofgen}) is either an isomorphism or a generator of $T_{\Theta_n}$. 
	\item $0<j \leq i$, the map
	\begin{equation*}
[U \to V] \in \sigma_! \mathrm{Comp}_{\Theta_{n-1}} \cup \sigma_{!} \mathrm{Segal}_{\Theta_{n-1}}
\end{equation*}
	is a suspension, and the map $V \to C_i$ does not factor through $C_0$. It follows that $V \to C_i$ is the suspension of a map. This case then follows by induction and Lemma \ref{lma:propertyOfSigma}.
	\item The final case is when $0< j < i$, the map $[U \to V] \in Se_{\Theta_n}$ is not a suspension, nor in $\mathrm{Comp}_{\Theta_n}$, and the map $V \to C_i$ does not factor through $C_0$. In this case (\ref{eqn:fiberproductofgen}) is a map of the form described in Lemma \ref{lma:MoremapsinTTheta}. \qedhere
\end{enumerate} 		
\end{proof}

\begin{remark*}
	We thank Charles Rezk for finding a critical gap in an earlier preprint version of the above proof. Correcting this led us to restructure many of the arguments in this section.  
\end{remark*}

\noindent Armed with this, we now reduce the problem to verifying that $T_{\Theta_n}=T^{(b)}_{\Theta_n}$.

\begin{lemma} \label{lma:ReductionOfTAtoTB+TC}
	If  $T^{(b)}_{\Theta_n}$ coincides with $T_{\Theta_n}$, then so does the class $T^{(a)}_{\Theta_n}$. 
\end{lemma}

\begin{proof}
	First note that as $\Theta_n$ is dense in $\pre(\Theta_n)$, and $T_{\Theta_n}$ is strongly saturated, it is enough to consider $H \in \Theta_n$ representable. 
	Let $f\colon U \to V$ be a morphism in $T_{\Theta_n}$, let $V \to C_i$ be given, and let $H \to C_i$ be arbitrary. There exists a unique factorization $H \to C_k \hookrightarrow C_i$, with $H \to C_k$ nondegenerate. Consider the following diagram of pullbacks in $\pre(\Theta_n)$:
	\begin{center}
	\begin{tikzpicture}
		\node (LT) at (0, 1) {$U''$};
		\node (LM) at (0, 0) {$V''$};
		\node (LB) at (0, -1) {$H$};
		\node (MT) at (1.5, 1) {$U'$};
		\node (MM) at (1.5, 0) {$V'$};
		\node (MB) at (1.5, -1) {$C_k$};
		\node (RT) at (3, 1) {$U$};
		\node (RM) at (3, 0) {$V$};
		\node (RB) at (3, -1) {$C_i$};
		
		\draw [->] (LT) -- node [left] {} (LM);
		\draw [->] (LM) -- node [left] {} (LB);
		\draw [->] (MT) -- node [left] {} (MM);
		\draw [->] (MM) -- node [left] {} (MB);
		\draw [->] (RT) -- node [right,font=\scriptsize] {$f$} (RM);
		\draw [->] (RM) -- node [right] {} (RB);
		
		\draw [->] (LT) -- node [above] {} (MT);
		\draw [->] (MT) -- node [above] {} (RT);
		\draw [->] (LM) -- node [above] {} (MM);
		\draw [->] (MM) -- node [above] {} (RM);
		\draw [->] (LB) -- node [below] {} (MB);
		\draw [right hook->] (MB) -- node [above] {} (RB);

		\node at (0.25, 0.75) {$\ulcorner$};
		\node at (0.25, -0.25) {$\ulcorner$};
		\node at (1.75, 0.75) {$\ulcorner$};
		\node at (1.75, -0.25) {$\ulcorner$};
		
	\end{tikzpicture}
	\end{center}
	Since $T^{(c)}_{\Theta_n} = T_{\Theta_n}$, we have $[U' \to V'] \in T_{\Theta_n}$, and if $T^{(b)}_{\Theta_n} = T_{\Theta_n}$, then we also have $[U'' \to V''] \in T_{\Theta_n}$, as desired.	
\end{proof}

\begin{lemma} \label{lma:OnlyNeedNondegenerates}
	For each $x\in\{a,b\}$, we have
\begin{equation*}
	T_{\Theta_n}^{(x)}=\left\{ [f\colon U\to V] \in T_{\Theta_n} \; \middle| \; 
	\begin{aligned}
		&\textrm{for any }[H\to C_i]\in J_x\textrm{ and any nondegenerate } \\ 
		&[V\to C_i]\in\pre(\Theta_n)\textrm{, one has }f\times_{C_i}\nu H \in T_{\Theta_n}
	\end{aligned}
	 \right\}
\end{equation*}
	In other words, to verify that $f\colon U \to V$ is in one of these classes, it suffices to consider only those fiber products $f \times_{C_i} \nu H$ with $V \to C_i$ nondegenerate.
\end{lemma}

\begin{proof}
	Let us focus on the case $x=a$. Let $f\colon U \to V$ be in class given on the right-hand side of the asserted identity. We wish to show that $f \in T_{\Theta_n}^{(a)} $, that is for any pair of morphism $H \to C_i$ and $V \to C_i$ we have
	\begin{equation*}
		 U' = U \times_{C_i} H  \to V \times_{C_i} H = V'
	\end{equation*}
is in $T_{\Theta_n}$. This follows as there exists a factorization $V \to C_k \to C_i$ and
	 a diagram of pullbacks:
	\begin{center}
	\begin{tikzpicture}
		\node (LT) at (0, 1) {$U'$};
		\node (LB) at (0, 0) {$U$};
		\node (RT) at (1.5, 1) {$V'$};
		\node (RB) at (1.5, 0) {$V$};
		\node (RRT) at (3, 1) {$H'$};
		\node (RRB) at (3, 0) {$C_k$};
		\node (RRRT) at (4.5, 1) {$H$};
		\node (RRRB) at (4.5, 0) {$C_i$};
		
		\draw [->] (LT) -- node [left] {$$} (LB);
		\draw [->] (RT) -- node [right] {$$} (RB);
		\draw [->] (RRT) -- node [right] {$$} (RRB);
		\draw [->] (RRRT) -- node [right] {$$} (RRRB);
		
		\draw [->] (LT) -- node [above] {$$} (RT);
		\draw [->] (RT) -- node [below] {$$} (RRT);
		\draw [->] (RRT) -- node [below] {$$} (RRRT);

		\draw [->] (LB) -- node [below] {$$} (RB);
		\draw [->] (RB) -- node [below] {$$} (RRB);
		\draw [right hook->] (RRB) -- node [below] {$$} (RRRB);
		
		\node at (0.25, 0.75) {$\ulcorner$};
		\node at (1.75, 0.75) {$\ulcorner$};
		\node at (3.25, 0.75) {$\ulcorner$};
		
	\end{tikzpicture}
	\end{center}
	such that $V \to C_k$ is nondegenerate. The analogous result for  $T_{\Theta_n}^{(b)}$ 
	follows by the same argument and the observation that $H' \to C_k$ is nondegenerate 
	if $H \to C_i$ is such.
\end{proof}

Now we settle an important first case of the equality $T_{\Theta_n}=T^{(b)}_{\Theta_n}$.

\begin{lemma} \label{lma:NonSuspensionSegalIsInTB+TC}
	Let $[U \to V] \in \mathrm{Segal}_{\Theta_n}$ be a morphism that is not contained in $\sigma_!(\mathrm{Segal}_{\Theta_{n-1}})$. Let $V \to C_i$ be nondegenerate, and let $[H \to C_i]\in J_b$ be a nondegenerate map in $\Theta_n$. Then the morphism $U \times_{C_i}  H \to V\times_{C_i}  H$ is contained in $T_{\Theta_n}$. 
\end{lemma}

\begin{proof}
 For the special case $i =0$, a more general version of this statement was proven by Rezk \cite[Proposition 6.6]{Rezk} and forms one of the cornerstone results of that work. Our current proof builds on Rezk's ideas.


The fundamental argument is to construct a category $\Q$ along with a functorial assignment of commuting squares
\begin{center}
\begin{tikzpicture}
	\node (LT) at (0, 1.5) {$A_\alpha$};
	\node (LB) at (0, 0) {$B_\alpha$};
	\node (RT) at (2, 1.5) {$U \times_{C_i}  H $};
	\node (RB) at (2, 0) {$V \times_{C_i}  H $};
	\draw [->] (LT) -- node [left] {$\wr$} (LB);
	\draw [->] (LT) -- node [above] {$$} (RT);
	\draw [->] (RT) -- node [right] {$$} (RB);
	\draw [->] (LB) -- node [below] {$$} (RB);
\end{tikzpicture}
\end{center}
for each $\alpha \in \Q$. This assignment is required to satisfy a host of conditions. 

First, each of the functors $A,B: \Q \to \pre(\Theta_n)$ is required to factor through $\tau_{\leq 0}\pre(\Theta_n)$, the category of 0-truncated objects. The 0-truncated objects of $\pre(\Theta_n)$ consist precisely of those presheaves of spaces taking values in the homotopically discrete spaces. There is no harm regarding such objects simply as ordinary set-valued presheaves, and we will do so freely. 

Second, we require that for each $\alpha \in \Q$ the natural morphism $A_\alpha \to B_\alpha$ is in the class $T_{\Theta_n}$. 
As $T_{\Theta_n}$ is saturated, this second condition implies that the natural map $\colim_\Q A \to \colim_\Q B$ is also in $T_{\Theta_n}$, where these colimits are taken in the $\infty$-category $\pre(\Theta_n)$ (hence are equivalently homotopy colimits for a levelwise model structure on simplical preseheaves, see Rk.~\ref{rk:CSSThetaInfty=Modelcat}).  


Third and last, we require that the natural maps $\colim_\Q A \to U \times_{C_i} H$ and $\colim_\Q B \to V \times_{C_i} H$ are equivalences in $\pre(\Theta_n)$ (i.e., levelwise weak equivalences of space-valued presheaves). If all of the above properties hold, then we obtain a natural commuting square
\begin{center}
\begin{tikzpicture}
	\node (LT) at (0, 1.5) {$\colim_\Q A$};
	\node (LB) at (0, 0) {$\colim_\Q B$};
	\node (RT) at (3, 1.5) {$U \times_{C_i}  H $};
	\node (RB) at (3, 0) {$V \times_{C_i}  H $};
	\draw [->] (LT) -- node [left] {$\wr$} (LB);
	\draw [->] (LT) -- node [above] {$\simeq$} (RT);
	\draw [->] (RT) -- node [right] {$$} (RB);
	\draw [->] (LB) -- node [above] {$\simeq$} (RB);
\end{tikzpicture}
\end{center}
in which the indicated morphisms are in the class $T_{\Theta_n}$. As this class is saturated it follows that $U \times_{C_i}  H \to V\times_{C_i}  H$ is also in this class. Thus if such a $\Q$ and associated functors can be produced, we will have completed the proof.  

At this point we deviate from Rezk's treatment. Specifically our category $\Q$ and associated functors will differ from his. 
We will focus on the more complicated case $i>0$, and leave the necessary simplifications in the case $i=0$ to the reader (or simply refer the reader to \cite[Proposition 6.6]{Rezk} ).

Under the assumptions of the statement of the lemma we have the following identifications of presheaves:
\begin{align*}
	U &= j({\{0, \dots, k\}}; o_1, \dots, o_k) \cup^{j({\{ k \}})} j({\{ k, k+1, \dots, m\}}; o_{k+1}, \dots, o_{m}) \\
	V &= j({[m]}; o_1, \dots, o_m)  \\
	H &= j([n]; u_1, \dots, u_n)
\end{align*}
where $ 0 \leq k \leq m$ and  $o_\alpha, u_\beta \in \Theta_{n-1}$ are given. 
If $i>0$, then the $i$-cell is the representable presheaf $j([1]; C_{i-1})$. A nondegenerate map $V \to C_i$ includes a nondegenerate map $f:[m] \to [1]$, and likewise a nondegenerate map $H \to C_i$ includes a nondegenerate map $g:[n] \to [1]$. Let $m'$ be the fiber over $0 \in [1]$, and let $m''$ be the fiber over $1$. Then $[m] = [m'] \cdot [m'']$ is the ordered concatenation of $[m']$ and $[m'']$. Similarly $[n] = [n'] \cdot [n'']$ is the ordered concatenation of the preimages of $0$ and $1$ under $g$. 

Let $\delta = (\delta', \delta''): [p] \to [m] \times_{[1]} [n]$ be a map which is an inclusion. There is a unique $-1 \leq r \leq p$ such that under the composite $[p] \to [m] \times_{[1]} [n] \to [1]$, an element $s$ maps to $0$ if and only if $s\leq r$ (hence maps to $1$ if and only if $s > r$). 
Associated to $\delta$ we have a subobject $C_\delta$ of $V \times_{C_i} H$, 
of the form $C_\delta = \sigma_!^{[p]}(c_1, \dots, c_p)$, where $c_{\ell}$ is given by the following formula:
\begin{equation*}
	\prod_{\delta'(\ell-1) < \alpha \leq \delta'(\ell)} o_{\alpha} \times \prod_{\delta''(\ell-1) < \beta \leq \delta''(\ell)} u_{\beta} 
\end{equation*}
if $\ell-1 \neq r$, and if $\ell-1 = r$ by
\begin{equation*}
	 \left(\prod_{\alpha} o_{\alpha} \right) \times \left(\prod_{\beta} u_{\beta} \right) \times \left( o_{m'} \times_{C_i} u_{n'} \right) \times \left( \prod_{\lambda} o_{\lambda} \right) \times \left(\prod_{\epsilon} u_{\epsilon} \right)
\end{equation*}
where the indices range over all $\delta'(\ell-1) < \alpha < m'$,   $\delta''(\ell-1) < \beta < n'$, $m' < \lambda \leq \delta'(\ell)$, and $m' < \beta \leq \epsilon''(\ell-1)$.
We have found the graphical image in Figure~\ref{fig:grids} to be especially useful in understanding the combinatorics of these subobjects. 
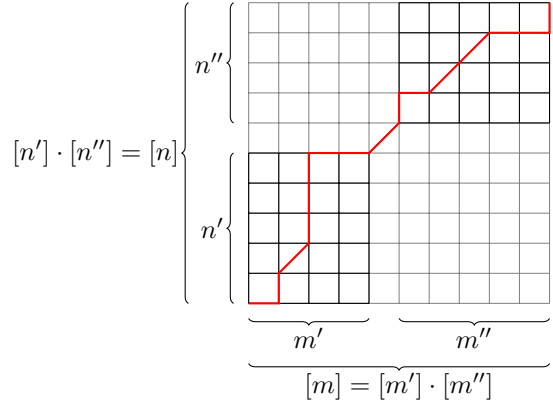
\begin{figure}[ht]
\begin{center}
\begin{tikzpicture}[decoration = brace, scale = 2]
	\draw[step = 2mm, help lines] (0,0) grid (2,2);
	\draw[step = 2mm] (0,0) grid (0.8,1);
	\draw[step = 2mm] (1, 1.2) grid (2,2);
	\draw (0.8,1) -- (1,1.2) -- (1, 2);
	\draw [decorate] (0.8, -0.1) -- node [below] {$m'$} (0,-0.1);
	\draw [decorate] (2,-0.1) -- node [below] {$m''$} (1, -0.1);
	\draw [decorate] (2,-0.4) -- node [below] {$[m] = [m'] \cdot [m'']$} (0, -0.4);
	
	\draw [decorate]  (-0.1,1.2) -- node [left] {$n''$} (-0.1, 2);
	\draw [decorate]  (-0.1, 0) -- node [left] {$n'$} (-0.1,1);
	\draw [decorate]  (-0.4,0) -- node [left] {$ [n'] \cdot [n''] = [n]$} (-0.4, 2);
	
	\draw[red, thick] (0,0) -- (0.2, 0) -- (0.2, 0.2) -- (0.4, 0.4) -- (0.4, 1) -- (0.8, 1)
		-- (1,1.2) -- (1, 1.4) -- (1.2, 1.4) -- (1.6, 1.8) -- (2, 1.8) -- (2,2); 
	
\end{tikzpicture}
\end{center}
	\caption{A graphical depiction of a typical map (shown in red) $\delta: [p] \to [m] \times_{[1]} [n]$.}
	\label{fig:grids}
\end{figure}

As subobjects of $V \times_{C_i} H$, the $C_\delta$ are naturally arranged into a poset. Let $W$ denote the disjoint union of all the maximal elements of this poset. Let $B_\bullet$ denote the simplicial \v{C}ech nerve associated to the morphism $W \to V \times_{C_i} H$. Each layer of $B_\bullet$ consists of a disjoint union of certain $C_\delta$. The map $W \to V \times_{C_i} H$ is a surjective map of set-valued presheaves.  It follows that it is also an effective epimorphism in the $\infty$-topos $\pre(\Theta_n)$, and hence \cite[Corollary 6.2.3.5]{HTT} the (homotopy) colimit of the simplcial diagram $B_\bullet$ is equivalent to $V \times_{C_i} H$. We set $\Q = \Delta$ and $B = B_\bullet$. 

We define $A_\bullet$ to be the fiber product of $B_\bullet$ with $U \times_{C_i} H$ over $V \times_{C_i} H$. Because colimits in $\infty$-topoi are universal, we have
\begin{align*}
	\colim_\Q A &\simeq \colim_\Q \left( B \times_{(V \times_{C_i} H)} (U \times_{C_i} H) \right) \\
	& \simeq \left( \colim_\Delta B_\bullet \right)  \times_{(V \times_{C_i} H)} (U \times_{C_i} H) \\ &\simeq U \times_{C_i} H.
\end{align*}
Thus all that remains is to show that the natural transformation $A_\bullet \to B_\bullet$ is levelwise in $T_{\Theta_n}$. 

As each layer of $B_\bullet$ is a disjoint union of certain $C_\delta$, it is sufficient to show that the map
\begin{equation*}
	C_\delta \times_{(V \times_{C_i} H)} (U \times_{C_i} H) \to C_{\delta}
\end{equation*} 
is in $T_\Theta$ for each $C_\delta$. As in the previous construction, there exist unique $0 \leq r \leq s \leq p$ such that $\delta'(t) < k$ if and only if $t < r$, and  $k < \delta'(t)$ if and only if $s < t$. The interval $\{r, \dots, s\} \subset [p]$ is precisely the preimage of $\{k\}$ under $\delta'$. We then have
\begin{align*}
	C_\delta &\times_{(V \times_{C_i} H)} (U \times_{C_i} H)  \\
	&\cong \sigma_!^{\{0, \dots, s\}}(b_1, \dots, b_s) \cup^{\sigma^{\{r, \dots, s\}}_!(b_{r+1}, \dots, b_s)} \sigma_!^{\{r, r+1, \dots, p\}}(b_{r+1}, \dots, b_p)
\end{align*}
and so the desired result follows from Lemma \ref{lma:MoremapsinTTheta}.
\end{proof}

\begin{remark}
	We note that the construction of a $\Q$, $A$, and $B$ demonstrably satisfying the above properties appears to be somewhat delicate. In the case $i=0$ the original published proof \cite[Proposition 6.6]{Rezk} was incorrect, and a corrected proof has been supplied by Rezk in \cite[Proposition 2.1]{MR2740648}.

	It is possible to give an alternative proof of Lemma \ref{lma:NonSuspensionSegalIsInTB+TC} which builds directly on Rezk's proof in the case $i=0$. Let $\Q$ be Rezk's category $\Q_{m,n}$ as defined in \cite{MR2740648}, and following Rezk define functors $A^{(i=0)}$ and $B^{(i=0)}$ as the objects
	\begin{align*}
		A^{(i=0)}_\delta &:= V_{(\delta_1, \delta_2)^{-1}(M \times N)}(D_1, \dots, D_p), \quad \text{and}\\
		B^{(i=0)}_\delta &:= V[p](D_1, \dots, D_p),
	\end{align*}
	where we are also using the notation of \cite{MR2740648}. Then these choices satisfy the requisite properties for the case $i=0$, the most difficult being \cite[Proposition 2.3]{MR2740648}. 
	
	For general $i$, notice that we have inclusions of subobjects
\begin{align*}
		V \times_{C_i} H &\subseteq V \times H, \\
		U \times_{C_i} H &\subseteq U \times H.
	\end{align*} 
We may define $A_\delta$ and $B_\delta$ as pullbacks
	\begin{align*}
		A_\delta &= A^{(i=0)}_\delta \times_{(U \times H)} (U \times_{C_i} H), \text{ and} \\
		B_\delta &= B^{(i=0)}_\delta \times_{(V \times H)} (V \times_{C_i} H). 
	\end{align*}
	Since colimits in $\infty$-topoi are universal we have
	\begin{align*}
	\colim_\Q	A &\simeq \colim_\Q A^{(i=0)} \times_{(U \times H)} (U \times_{C_i} H) \simeq U \times_{C_i} H, \\
	\colim_\Q	B &\simeq \colim_\Q B^{(i=0)} \times_{(V \times H)} (V \times_{C_i} H) \simeq V \times_{C_i} H,	
	\end{align*}
	and so all that remains is to verify that $A_\delta \to B_\delta$ is indeed in $T_{\Theta_n}$. This can be accomplished by explicitly computing $A_\delta$ and $B_\delta$ in terms of the functors $\sigma^{[\ell]}_!$ and invoking Lemma \ref{lma:MoremapsinTTheta}.
\end{remark}

We may now complete the proof of (R.1) for complete Segal $\Theta_n$-spaces.

\begin{theorem} \label{thm:ThetanR1}
	The triple $(\Theta_n, T_{\Theta_n}, i)$ satisfies axiom (R.1), namely $i^*(S) \subset T_{\Theta_n}$.
\end{theorem}

\begin{proof}
	By Lemma \ref{lma:ReductionOfTAtoTB+TC}, it is enough to show that the strongly saturated classes $T_{\Theta_n}^{(b)}$ contains the generating sets $\mathrm{Segal}_{\Theta_n}$ and $\mathrm{Comp}_{\Theta_n}$. By Lemma \ref{lma:OnlyNeedNondegenerates}, it suffices to show that for any $[U \to V] \in \mathrm{Segal}_{\Theta_n} \cup \mathrm{Comp}_{\Theta_n} $, any nondegenerate morphism $[H \to C_i]\in J_b$, and any nondegenerate morphism $V \to C_i$ of $\pre(\Theta_n)$, we must show that
	\begin{equation*}
		U' = U \times_{C_i} \nu H \to V \times_{C_i} \nu H = V'
	\end{equation*}
is contained in $T_{\Theta_n}$.  Observe the following:
\begin{itemize}
	\item If $[U \to V] \in \mathrm{Segal}_{\Theta_n}$ is not in the image of $\sigma_!$, then $U \to V$ is contained in $T_{\Theta_n}^{(b)}$ by Lemma \ref{lma:NonSuspensionSegalIsInTB+TC}.
	\item If $[U \to V] \in \mathrm{Comp}_{\Theta_n}$ is not in the image of $\sigma_!$, then $V=C_0$, and the only nondegenerate map $V \to C_i$ occurs when $i=0$. In this case $U \to V$ is in $T_{\Theta_n}^{(b)}$ by \cite[Proposition 6.1]{Rezk}. 
\end{itemize}	
Thus we may restrict our attention to those generators $U \to V$ that lie in the image of $\sigma_!$. We proceed by induction.  When $n=1$, the set of generators in the image of $\sigma_!$ is empty. 

Assume that
\begin{equation*}
T_{\Theta_{n-1}} = T_{\Theta_{n-1}}^{(a)} = T_{\Theta_{n-1}}^{(b)} =  T_{\Theta_{n-1}}^{(c)},
\end{equation*}
and let $U \to V$ be an element of $\mathrm{Segal}_{\Theta_n} \cup \mathrm{Comp}_{\Theta_n}$ that lies in the image of $\sigma_!$. Now note that if $C_i = C_0$, then $U' \to V'$ lies in $T_{\Theta_n}$, again by \cite[Proposition 6.1]{Rezk}. If $i \neq 0$, then by Lemma \ref{lma:NonDegenFromSuspension=Suspension}, the map $V \to C_i$ is also in the image of $\sigma_!$. In this case, if we have a factorization $H \to C_0 \to C_i$ (which, since $H \to C_i$ is nondegenerate, can only happen if $H = C_0$), then $U' \to  V'$ is an equivalence (as both are empty). Hence it $U\to V$ lies in $T_{\Theta_n}$.
	
	This leaves the final case, where both $[U \to V]$ and $[V \to C_i]$ lie in the image of $\sigma_!$, and $[H \to C_i]$ is nondegenerate with $H = j({[m]}; o_1, \dots, o_m) \neq C_0$, for some $m \geq 1$, $o_i \in \Theta_{n-1}$. The nondegenerate map $H \to C_i = ({[1]}; C_{i-1})$ is given explicitly by the following data (see also the proof of Lemma \ref{lma:propertyOfSigma}): a map $i_k: {[m]} \to {[1]}$ for some $1 \leq k \leq m$ such that $i_k(i) = 0$ if $i < k$ and $i_k(i) = 1$ otherwise, together with a single (nondegenerate) map $o_k \to C_{i-1}$. In this case we may explicitly compute the pullback 
	\begin{equation*}
		U' = U \times_{C_i}  H \to V \times_{C_i}  H = V'
	\end{equation*}	
and deduce that it is contained in the class $T_{\Theta_n}$.

As $[U \to V]$ is in the image of $\sigma_!^{[1]}$, it is of the form $({[1]}; U'') \to ({[1]}; V'')$ for some $[U'' \to V'']$ in $\mathrm{Segal}_{\Theta_{n-1}}$. The pullback is then given explicitly as: 
	\begin{align*}
		U' = ({[m]}; & o_1, \dots, o_k \times_{C_{i-1}} U'', o_{k+1}, \dots, o_m) \\
		& \to ({[m]}; o_1, \dots, o_k \times_{C_{i-1}} V'', o_{k+1}, \dots, o_m) = V'.
	\end{align*}
This map arises as the right-most vertical map in the following (oddly drawn) commuting square:
\begin{center}
\begin{tikzpicture}
	\node (LT) at (0, 2) {\footnotesize $j({[k-1]}; o_1, \dots, o_{k-1}) \cup^{j({\{k-1\}})} ({\{ k-1, k\}};o_k \times_{C_{i-1}} U'') \cup^{j({\{ k\}})} j({\{k, \dots, m\}}; o_{k+1}, \dots, o_m)$};
	\node (LB) at (0, -1) {\footnotesize $j({[k-1]}; o_1, \dots, o_{k-1}) \cup^{j({\{k-1\}})} ({\{ k-1, k\}};o_k \times_{C_{i-1}} V'') \cup^{j({\{ k\}})} j({\{k, \dots, m\}}; o_{k+1}, \dots, o_m)$};
	\node (RT) at (3.5, 1.25) {\footnotesize $({[m]}; o_1, \dots, o_k \times_{C_{i-1}} U'', o_{k+1}, \dots, o_m)$};
	\node (RB) at (3.5, -0.25) {\footnotesize $({[m]}; o_1, o_2, \dots, o_k \times_{C_{i-1}} V'', o_{k+1}, \dots, o_m)$};
	\draw [->] (LT.200) -- node [left] {$$} (LB.160);
	\draw [->] (LT) |-  node [above] {$$} (RT);	
	\draw [->] (RT) -- node [right] {$$} (RB);
	\draw [->] (LB) |-  node [below] {$$} (RB);
\end{tikzpicture}
\end{center}
The left-most vertical map is a pushout of identities and (by induction) a map in $\sigma_!(T_{\Theta_{n-1}})$. Thus by Lemma \ref{lma:propertyOfSigma} it is contained in $T_{\Theta_n}$. Both horizontal maps are contained in $T_{\Theta_n}$ by \cite[Proposition 6.4]{Rezk}, whence the right-most vertical map $[U' \to V']$ is also contained in $T_{\Theta_n}$, as desired. 
\end{proof}

\begin{lemma} \label{lma:ThetanR2}
	The triple $(\Theta_n, T_{\Theta_n}, i)$ satisfies axiom (R.2), namely $i_!(T_{\Theta_n}) \subseteq S$. 
\end{lemma}

\begin{proof}
	As $i_!$ commutes with colimits, to show that $i_!(T_{\Theta_n}) \subseteq S$ it is sufficient to show this property for a subset that generates $T_{\Theta_n}$ under colimits. The maps in $\mathrm{Comp}_{\Theta_n}$ are clearly mapped into $S$. This leaves the maps $\mathrm{Segal}_{\Theta_n}$. We now write $S=S_n$ and induct on $n$. When $n=0$, one has $\Upsilon_0 = \Theta_0 = \mathrm{pt}$.
		
Assume that $i_!(T_{\Theta_{n-1}}) \subseteq S_{n-1}$. The suspension functor $\sigma_!\colon\pre(\Upsilon_{n-1}) \to \pre(\Upsilon_n)$ preserves colimits and sends the generators $S_{n-1}$ into $S_n$. Hence the suspensions of maps in $i_!(T_{\Theta_{n-1}}) $ are in $S_n$. Moreover, by construction the image under $i_!$ of the following map
\begin{equation*}
	j({\{0,1\}}; C_i) \cup^{j({\{1\}})} j({\{1,2\}}; C_i) \to j({\{0,1,2\}};C_i, C_i)
\end{equation*}
is in $S_n$ for all cells $C_i$. 
By induction, it follows that all the Segal generators are mapped into $S_n$ 
except possibly the following
	\begin{align} \label{eqn:FinalGridSegalMapToCheck}
	j({\{0, \dots, k\}}; o_1, \dots, o_k) \cup^{j({\{k\}})} & j({\{ k, \dots, m\}}; o_{k+1}, \dots, o_m) \\ 
	& \to j({\{ 0, \dots, m\}}; o_1, \dots, o_m ) \notag
	\end{align}
where $o_i \in \Theta_{n-1}$. To show that $i_!$ maps the above morphism to a morphism in $S_n$, we observe that the above map may be rewritten as follows. The source may be written as
\begin{align*}
	j({\{0,\dots, k\}}; o_1, \dots, o_k) & \times_{j({\{k-1, k\}}; C_0)} \left[ j({\{k-1,k\}}; C_0) \cup^{j({\{k\}})} j({\{k,k+1\}}; C_0) \right] \\
	& \times_{j({\{k, k+1\}}; C_0)} j({\{k,k+1, m\}}; o_{k+1}, \dots, o_m)
\end{align*}
while the target is 
\begin{align*}
	j({\{0,\dots, k\}}; o_1, \dots, o_k) & \times_{j({\{k-1, k\}}; C_0)}  j({\{k-1,k, k+1\}}; C_0, C_0) \\  
	& \times_{j({\{k, k+1\}}; C_0)} j({\{k,k+1, m\}}; o_{k+1}, \dots, o_m).
\end{align*}
Schematically then, the map of \eqref{eqn:FinalGridSegalMapToCheck} is of the form
\begin{equation*}
	A \times_{C_{1}} U \times_{C_{1}} B \to A \times_{C_{1}} V \times_{C_{1}} B
\end{equation*}
for $U \to V$ in $S$ and $A, B \in \Upsilon_n$. By property (C.2) of $\cat_{(\infty,n)}$ (cf. also Proposition \ref{prop:Axiom4reduction}) it follows that \eqref{eqn:FinalGridSegalMapToCheck} lies in $S_n$ also.
\end{proof}

\begin{theorem} \label{thm:Rezksmodel}
	The triple $(\Theta_n, T_{\Theta_n}, i)$ satisfies the axioms (R.1-4); The $\infty$-category $\CSS(\Theta_n)$ of complete Segal $\Theta_n$-spaces is a theory of $(\infty,n)$-categories.
\end{theorem}

\begin{proof} 
	Condition (R.4) is clear, and the functor $i: \Theta_n \to \Upsilon_n$ is a fully-faithful inclusion, hence (R.3) is automatically satisfied. Conditions (R.1) and (R.2) follow from Th.~\ref{thm:ThetanR1} and Lemma \ref{lma:ThetanR2}. 
\end{proof}

\begin{remark}
	From the above theorem and Th.~\ref{thm:SpaceOfTheoriesIsBZ/2} it follows that the automorphism group of the $\infty$-category $\CSS(\Theta_n)$ is the discrete group $(\Z/2)^n$. A more direct proof of this fact, based on computing the automorphisms of the category $\Theta_n$, has appeared in \cite{1312.4994}.
\end{remark}


\section{$n$-Fold complete Segal spaces are a homotopy theory of $(\infty, n)$-categories}\label{sect:CSSnmodel}

We give the iterative construction of the $\infty$-category $\CSS(\Delta^{\!\times n})$ of $n$-fold complete Segal spaces, following \cite{barwick-thesis} and \cite[\S~1]{G}, and we show that $\CSS(\Delta^{\!\times n})$ is a theory of $(\infty,n)$-categories. 

\begin{definition}[\protect{\cite{barwick-thesis}}] Let $\CSS(\Delta^0)$ be the $\infty$-category $\mathcal{S}$ of Kan simplicial sets.
Suppose now that $n$ is a positive integer; assume that both a presentable $\infty$-category $\CSS(\Delta^{\!\times n-1})$ and a fully faithful functor
\begin{equation*}
c_{n-1}\colon\CSS(\Delta^0)\hookrightarrow\CSS(\Delta^{\!\times n-1})
\end{equation*}
that preserves all small colimits have been constructed. Let us call a simplicial object $X\colon\mathrm{N}\Delta^{\op}\to\CSS(\Delta^{\!\times n-1})$ an \emph{$n$-fold Segal space} if it satisfies the following pair of conditions.
\begin{enumerate}[{\hspace{\parindent}(B.}1{)}]
\item The object $X_0$ lies in the essential image of $c_{n-1}$.
\item For any integers $0<k<m$, the object $X_m$ is exhibited as the limit of the diagram
\begin{equation*}
X(\{0,1,\dots,k\})\rightarrow X(\{k\})\leftarrow X(\{k,k+1,\dots,m\}).
\end{equation*}
\end{enumerate}
Now for any $n$-fold Segal space $X$, one may apply the right adjoint to the functor $c_{n-1}$ objectwise to $X$ to obtain a simplicial space $\iota_1X$. Let us call $X$ an \emph{$n$-fold complete Segal space} if it satisfies the following additional condition.
\begin{enumerate}[{\hspace{\parindent}(B.}1{)}]\addtocounter{enumi}{2}
\item The Kan complex $(\iota_1X)_0$ is exhibited as the limit of the composite functor
\begin{equation*}
\Delta_{/\mathrm{N}E}^{\op}\to\Delta^{\op}\stackrel{\iota_1X}{\longrightarrow}\CSS_0,
\end{equation*} 
\end{enumerate}
where the category $E$ is as in Ex.~\ref{Ex:Strictncats}. Denote by $\CSS(\Delta^{\!\times n})$ the full subcategory of $\Fun(\mathrm{N}\Delta^{\op},\CSS(\Delta^{\!\times n-1})$ spanned by the $n$-fold complete Segal spaces.
\end{definition}

In order to make sense of the inductive definition above, it is necessary to show that $\CSS(\Delta^{\!\times n})$ is a presentable $\infty$-category, and to construct a fully faithful, colimit-preserving functor $c_n\colon\CSS(\Delta^0)\hookrightarrow\CSS(\Delta^{\!\times n})$. To prove presentability, we demonstrate that the $\infty$-category $\CSS(\Delta^{\times n})$ is in fact an accessible localization of $\Fun(\mathrm{N}\Delta^{\op},\CSS(\Delta^{\!\times n-1}))$; then the desired functor $c_n$ will be the composite
\begin{equation*}
\CSS(\Delta^0)\stackrel{c_{n-1}}{\hookrightarrow}\CSS(\Delta^{\!\times n-1})\stackrel{c}{\longrightarrow}\Fun(\Delta^{\op},\CSS(\Delta^{\!\times n-1}))\stackrel{L}{\longrightarrow}\CSS(\Delta^{\!\times n}),
\end{equation*}
where $c$ denotes the constant functor and $L$ denotes the purported localization.

\begin{lemma} For any positive integer $n$, the $\infty$-category $\CSS(\Delta^{\!\times n})$ is an accessible localization of $\Fun(\mathrm{N}\Delta^{\op},\CSS(\Delta^{\!\times n-1}))$.
\end{lemma}

\begin{proof} Denote by $j$ any Yoneda embedding (the context will always be made clear). Let $K$ denote the simplicial set as in~\ref{ntn:simplicialsetK}, which we regard as a simplicial space that is discrete in each degree. This is a pushout along an inclusion, hence this is also a (homotopy) pushout in the $\infty$-category of simplicial spaces. Now let $T$ be the strongly saturated class of morphisms of $\Fun(\mathrm{N}\Delta^{\op},\CSS(\Delta^{\!\times n-1}))$ generated by the three sets
\begin{eqnarray}
&\{j([0],\mathbf{m})\to j(\mathbf{0})\ |\ \mathbf{m}\in\Delta^{\!\times n-1}\},&\nonumber\\
& \{ \mathrm{Segal}_{\Delta} \boxtimes \mathbf{m} \ |\ \mathbf{m} \in \Delta^{\!\times n-1} \},&\nonumber\\
&\{c_{n-1}(K)\to j(\mathbf{0})\},&\nonumber\nonumber
\end{eqnarray}
One deduces immediately that a simplicial object of $\CSS(\Delta^{\!\times n-1})$ is a Segal space if and only if it is local with respect to each of the first two sets of morphisms. To show that $\CSS(\Delta^{\!\times n})$ coincides with the localization $T^{-1}\Fun(\mathrm{N}\Delta^{\op},\CSS(\Delta^{\!\times n-1}))$, it is enough to show that a $1$-fold Segal space $X$ is complete if and only if the natural map
\begin{equation*}
X_0\to\map(K,X)
\end{equation*}
is an equivalence. 
By the Yoneda lemma, our claim is just a restatement of \cite[Proposition 10.1]{Rezk}.
\end{proof}

\begin{corollary} For any nonnegative integer $n$, the $\infty$-category $\CSS(\Delta^{\!\times n})$ is an accessible localization of $\pre(\Delta^{\!\times n-1})$.
\end{corollary}

\begin{proof} If $n=0$, there is nothing to prove. If $n$ is positive, then let us suppose that we have written $\CSS_{n-1}$ as a localization $T_{\Delta^{\!\times n-1}}^{-1}\pre(\Delta^{\!\times n-1})$ for some strongly saturated class $T_{\Delta^{\!\times n-1}}$ of small generation. Denote by
\begin{equation*}
\boxtimes\colon\pre(\Delta)\times\pre(\Delta^{\!\times n-1}) \to\pre(\Delta^{\!\times n})
\end{equation*}
the essentially unique functor that carries pairs of the form $(j[k],j(\mathbf{m}))$ to $j([k],\mathbf{m})$ and preserves colimits separately in each variable. Now let $T_{\Delta^{\!\times n}}$ be the strongly saturated class generated by the class $T$ above along with the set
\begin{equation*}
\{j[k]\boxtimes U \to j[k] \boxtimes V\ |\ [U\to V]\in T_{\Delta^{\!\times n-1}}\}.
\end{equation*}
Now $\CSS(\Delta^{\!\times n})$ coincides with $T_{\Delta^{\!\times n}}^{-1}\pre(\Delta^{\!\times n})$.
\end{proof}

\begin{remark}
	The class $T_{\Delta^{\!\times n}}$ is precisely the strongly saturated class generated by the union of $\mathrm{Segal}_{\Delta^{\!\times n}}$, $\mathrm{Glob}_{\Delta^{\!\times n}}$, and $\mathrm{Comp}_{\Delta^{\!\times n}}$ as in Cor~\ref{cor:GauntncatsAsPresheavesOnDeltaN}.   
\end{remark}

Write $d\colon\Delta^{\!\times n}\to \Upsilon_n$ for the composite of the functor $\delta_n\colon\Delta^{\!\times n}\to\Theta_n$ described in \cite[Definition 3.8]{MR2331244} followed by the  fully faithful functor $i:\Theta_n\hookrightarrow \Upsilon_n$ consider in the previous section. We will now show that the triple $(\Delta^{\times n}, T_{\Delta^{\times n}}, d)$ satisfies conditions (R.1-4) of Th.~\ref{Thm:RecognitionForPresheaves}, hence $\CSS(\Delta^{\times n})$ is a theory of $(\infty,n)$-categories. In contrast to the previous section, the functor $d$ is not fully-faithful and hence condition (R.3) is not automatic. We thus begin with this condition. 


\begin{lemma} \label{lma:DeltaNR3}
	The triple $(\Delta^{\times n}, T_{\Delta^{\times n}}, d)$ satisfies condition (R.3), that is for all objects $\mathbf{m} \in \Delta^{\times n}$, the canonical map $\mathbf{m} \to \delta_n^* \delta_n(\mathbf{m})\simeq d^* d(\mathbf{m})$ is in $T_{\Delta^{\times n}}$.
\end{lemma}

\begin{proof}
We will proceed by induction on $n$, the base case $n=0$ being trivial.  As the functor $j[m] \boxtimes (-)$ preserves colimits and sends the generators of $T_{\Delta^{\times n-1}}$ into $T_{\Delta^{\times n}}$, we have a containment $j[m] \boxtimes T_{\Delta^{\times n-1}} \subseteq T_{\Delta^{\times n}}$. Thus by induction the canonical map, 
	\begin{equation*}
	 j[m] \boxtimes \mathbf{m} \to j[m] \boxtimes \delta^*_{n-1} \delta_{n-1} (\mathbf{m})	
	\end{equation*} 
is in $T_{\Delta^{\times n}}$. In particular when $m = 0$, the map $j[m] \boxtimes \mathbf{m} \to j[m] \boxtimes \mathbf{0}$ is a composite of maps in $\mathrm{Glob}_{\Delta^{\!\times n}}$, whence the map
\begin{equation*}
	j[0] \boxtimes \delta^*_{n-1} \delta_{n-1} (\mathbf{m})	 \to j[0] \boxtimes \mathbf{0}
\end{equation*}
is in $T_{\Delta^{\times n}}$. 

We will first prove the lemma for objects of the form  $j[1] \boxtimes \mathbf{m} \in \Delta \times \Delta^{\times n-1} \cong \Delta^{\times n}$.  One may readily check that the following is a pushout square of presheaves of sets:
\begin{center}
\begin{tikzpicture}
	\node (LT) at (0, 2) {$(j\{0\} \boxtimes \delta^*_{n-1} \delta_{n-1} (\mathbf{m})) \sqcup ({j\{1\} \boxtimes \delta^*_{n-1} \delta_{n-1} (\mathbf{m})}) $};
	\node (LB) at (0, 0) {$(j\{0\} \boxtimes \mathbf{0}) \sqcup (j\{1\} \boxtimes \mathbf{0})$};
	\node (RT) at (6, 2) {$j[1] \boxtimes \delta^*_{n-1} \delta_{n-1} (\mathbf{m})$};
	\node (RB) at (6, 0) {$\delta^*_{n} \delta_{n} (j[1] \boxtimes \mathbf{m})$};
	\draw [->] (LT) -- node [left] {$$} (LB);
	\draw [right hook->] (LT) -- node [above] {$$} (RT);
	\draw [->] (RT) -- node [right] {$$} (RB);
	\draw [->] (LB) -- node [below] {$$} (RB);
	\node at (5.5, 0.5) {$\lrcorner$};
\end{tikzpicture}
\end{center}
Moreover, as the topmost map is an inclusion of sets and pushouts in $\pre(\Delta^{\times n})$ are computed object-wise, this is also a (homotopy) pushout square in $\pre(\Delta^{\times n})$. As we just observed, the left-most map is in the strongly saturated $T_{\Delta^{\times n}}$, whence the right-most  map is also in $T_{\Delta^{\times n}}$. It follows that the composite,
\begin{equation*}
	j[1] \boxtimes \mathbf{m} \to j[1] \boxtimes \delta^*_{n-1} \delta_{n-1} (\mathbf{m}) \to \delta^*_{n} \delta_{n} (j[1] \boxtimes \mathbf{m})
\end{equation*}
is in $T_{\Delta^{\times n}}$. 

To prove the general case, i.e., that the map $j[k] \boxtimes \mathbf{m} \to \delta^*_{n} \delta_{n} (j[k] \boxtimes \mathbf{m})$ is in $T_{\Delta^{\times n}}$, we induct on $k$. Assume the result holds when $k \leq m$. We will prove it for $k=m+1$. First consider the following commutative square:
\begin{center}
\begin{tikzpicture}
	\node (LT) at (0, 2) {$j[m] \boxtimes \mathbf{m} \cup^{j[0] \boxtimes \mathbf{m}} j[1] \boxtimes \mathbf{m}$};
	\node (LB) at (0, 0) {$\delta_n^* \delta_n (j[m] \boxtimes \mathbf{m}) \cup^{j[0] \boxtimes \mathbf{0}} \delta_n^* \delta_n(j[1] \boxtimes \mathbf{m})$ };
	\node (RT) at (5, 2) {$j[m+1] \boxtimes \mathbf{m}$};
	\node (RB) at (5, 0) {$\delta_n^* \delta_n(j[m+1] \boxtimes \mathbf{m})$};
	\draw [->] (LT) -- node [left] {$\wr$} (LB);
	\draw [->] (LT) -- node [above] {$\sim$} (RT);
	\draw [->] (RT) -- node [right] {$$} (RB);
	\draw [->] (LB) -- node [below] {$$} (RB);
\end{tikzpicture}
\end{center}
The indicated maps are in  $T_{\Delta^{\times n}}$; the topmost map is a generator and the lefttmost vertical map 
by induction. As $T_{\Delta^{\times n}}$ is saturated, the rightmost vertical map is in $T_{\Delta^{\times n}}$ if and only if the bottommost map is as well. Thus it suffices to prove that the natural map 
\begin{equation*}
	\delta_n^* \delta_n (j[m] \boxtimes \mathbf{m}) \cup^{j[0] \boxtimes \mathbf{0}} \delta_n^* \delta_n(j[1] \boxtimes \mathbf{m}) \to \delta_n^* \delta_n(j[m+1] \boxtimes \mathbf{m})
\end{equation*} 
is in $T_{\Delta^{\times n}}$.

The Yoneda embedding is dense for any presheaf $\infty$-category and hence the object $\delta_n^* \delta_n(j[m+1] \boxtimes \mathbf{m})$ may canonically be written as a colimit of representable presheaves. Let $\D = (\Delta^{\times n} \downarrow \delta_n^* \delta_n(j[m+1] \boxtimes \mathbf{m}))$ denote the overcategory consisting of pairs $(j[p] \boxtimes \mathbf{p}, \phi)$ where $\phi$ is a map
\begin{equation*}
	\phi: j[p] \boxtimes \mathbf{p} \to \delta_n^* \delta_n(j[m+1] \boxtimes \mathbf{m}).
\end{equation*} 
Let $B:\D \to \pre(\Delta^{\times n})$ denote the functor which forgets the map $\phi$. We have a canonical equivalence in $\pre(\Delta^{\times n})$:
\begin{equation*}
	\colim_{\D} B \simeq \delta_n^* \delta_n(j[m+1] \boxtimes \mathbf{m}).
\end{equation*}

By adjunction, specifying a map $\phi: j[p] \boxtimes \mathbf{p} \to \delta_n^* \delta_n(j[m+1] \boxtimes \mathbf{m})$ is equivalent to a specifying a map $\phi': \delta_n(j[p] \boxtimes \mathbf{p}) \to \delta_n(j[m+1] \boxtimes \mathbf{m})$, i.e., a map in $\pre(\Theta_n)$:
\begin{equation*}
	\phi': ([p]; \underbrace{\delta_{n-1}(\mathbf{p}), \dots, \delta_{n-1}(\mathbf{p}) }_{p \text{ times}}) \to ([m+1]; \underbrace{\delta_{n-1}(\mathbf{m}), \dots, \delta_{n-1}(\mathbf{m}) }_{m+1 \text{ times}}).
\end{equation*}
In particular every such map includes the data of a map $\overline{\phi}: [p] \to [m+1]$. To simplify notation we will denote the object $(j[p] \boxtimes \mathbf{p}, \phi)$ as $B_\phi$ or simply $\phi$. 

Let $\C$ denote the full subcategory of $\D$ consisting of the union of the following three types of objects:
\begin{enumerate}
	\item those $B_\phi$ in which $\overline{\phi}$ factors as 
	\begin{equation*}
		\overline{\phi}: [p] \to \{ 0, \dots, m \} \subseteq [m+1],
	\end{equation*}
	\item those $B_\phi$ in which $\overline{\phi}$ factors as 
	\begin{equation*}
		\overline{\phi}: [p] \to \{ m, m +1 \} \subseteq [m+1], \quad \text{and}
	\end{equation*}
	\item those $B_\phi$ in which $(\overline{\phi})^{-1}(\{ m \}) = \{r\} \subseteq [p]$ consists of a singleton for some $0 \leq r \leq p$. 
\end{enumerate}

For any object $D\in\D$, the under category $\C_{D/}$ actually has an initial object and is thus weakly contractible. Consequently (see, e.g., \cite[Th.~4.1.3.1 and Proposition 4.1.1.8]{HTT}), the induced morphism of (homotopy) colimits over these categories is an equivalence; in particular, it follows that the following canonical maps are equivalences in $\pre(\Delta^{\times n})$:
\begin{equation*}
	\colim_{\C} B \simeq \colim_{\D} B \simeq \delta_n^* \delta_n(j[m+1] \boxtimes \mathbf{m}).
\end{equation*}

For each $\phi \in \C$, let $A_\phi$ denote the fiber product
\begin{equation*}
	A_\phi := B_\phi \times_{\delta_n^* \delta_n(j[m+1] \boxtimes \mathbf{m})} \left( \delta_n^* \delta_n (j[m] \boxtimes \mathbf{m}) \cup^{j[0] \boxtimes \mathbf{0}} \delta_n^* \delta_n(j[1] \boxtimes \mathbf{m}) \right).
\end{equation*}
This gives rise to a new functor $A: \C \to \pre(\Delta^{\times n})$, and as colimits in $\pre(\Delta^{\times n})$ are universal, we have natural equivalences:
\begin{align*}
	\colim_\C A &\simeq \left( \colim_\C B \right) \times_{\delta_n^* \delta_n(j[m+1] \boxtimes \mathbf{m})} \left( \delta_n^* \delta_n (j[m] \boxtimes \mathbf{m}) \cup^{j[0] \boxtimes \mathbf{0}} \delta_n^* \delta_n(j[1] \boxtimes \mathbf{m}) \right) \\
	& \simeq  \delta_n^* \delta_n (j[m] \boxtimes \mathbf{m}) \cup^{j[0] \boxtimes \mathbf{0}} \delta_n^* \delta_n(j[1] \boxtimes \mathbf{m}).
\end{align*} 
Thus the desired result follows if we can demonstrate that the natural map
\begin{equation*}
	\colim_\C A \to \colim_\C B
\end{equation*}
is in the class $T_{\Delta^{\times n}}$. This class, being saturated, is closed under colimits, and so it suffices to show that each of the maps
\begin{equation*}
	A_\phi \to B_\phi
\end{equation*}
is in $T_{\Delta^{\times n}}$. If $B_\phi \in \C$ is of type (a) or type (b), then $A_\phi \simeq B_\phi$ is an equivalence, hence in the desired class. If $B_\phi = (j[p] \boxtimes \mathbf{p}, \phi)$ is of type (c), so that $(\overline{\phi})^{-1}(\{m\}) = \{ r\}$ for $0 \leq r \leq p$, then a direct calculation reveals:
\begin{equation*}
	A_\phi \simeq \left( j\{0, \dots, r\} \boxtimes \mathbf{p} \right) \cup^{\left(j\{r\} \boxtimes \mathbf{p}\right)} \left( j\{r, r+1, \dots, p\} \boxtimes \mathbf{p}\right) \to j[p] \boxtimes \mathbf{p} \simeq B_\phi.
\end{equation*}
As this is one of the generators of $T_{\Delta^{\times n}}$, the result follows. 
\end{proof}

\begin{theorem}\label{thm:CSSnisathy}
	The $\infty$-category $\CSS(\Delta^{\times n})$ of $n$-fold complete Segal spaces is a theory of $(\infty,n)$-categories. 
\end{theorem}

\begin{proof}
	We will show that the triple $(\Delta^{\times n}, T_{\Delta^{\times n}}, d)$ satisfies conditions (R.1-4) of Th.~\ref{Thm:RecognitionForPresheaves}. Condition (R.4) clearly holds. Condition (R.3) is the statement of Lemma \ref{lma:DeltaNR3}. 
	
	For condition (R.2) we must show that $i_! (\delta_n)_!(T_{\Delta^{\times n}}) \subseteq S$. By Lemma \ref{lma:ThetanR2} it is sufficient to show that $(\delta_n)_!(T_{\Delta^{\times n}}) \subseteq T_{\Theta_n}$, and as $(\delta_n)_!$ preserves colimits it is sufficient to check this on the generating classes $\mathrm{Segal}_{\Delta^{\!\times n}}$, $\mathrm{Glob}_{\Delta^{\!\times n}}$, and $\mathrm{Comp}_{\Delta^{\!\times n}}$. In each case this is clear: the set $\mathrm{Glob}_{\Delta^{\!\times n}}$ maps under $(\delta_n)_!$ to equivalences in $\pre(\Theta_n)$, the set $\mathrm{Comp}_{\Delta^{\!\Theta_n}}$ is constructed as the image of $\mathrm{Comp}_{\Delta^{\!\times n}}$ under $(\delta_n)_!$, and the image of $\mathrm{Segal}_{\Delta^{\!\times n}}$ under $(\delta_n)_!$ is a subset of $\mathrm{Segal}_{\Delta^{\!\Theta_n}}$.
	
For the final condition (R.1) we must show that $\delta_n^* i^*(S) \subseteq T_{\Delta^{\times n}}$. By Th.~\ref{thm:ThetanR1}, it suffices to show that $\delta_n^*(T_{\Theta_n}) \subseteq T_{\Delta^{\times n}}$. As $\delta_n^*$ preserves colimits, it is sufficient to prove this for the generating class of $T_{\Theta_n}$. 
As we previously mentioned, the set $\mathrm{Comp}_{\Delta^{\!\Theta_n}}$ consists of elements in the image of $(\delta_n)_!$. By Proposition \ref{prop:ThetaNSegalAreRetractsOfDeltaNSegal} the remaining generators of $T_{\Theta_n}$ are retracts of maps in the image of $(\delta_n)_!$. Thus $T_{\Theta_n}$ is contained in the strongly saturated class generated from $(\delta_n)_!(T_{\Delta^{\times n}})$. 
Hence $\delta_n^*(T_{\Theta_n})$ is contained in the strongly saturated class generated by $\delta_n^* (\delta_n)_!(T_{\Delta^{\times n}})$. Using Lemma \ref{lma:DeltaNR3} one readily deduces that the generators of $T_{\Delta^{\times n}}$ are mapped, via  
$\delta_n^* (\delta_n)_!$ back into $T_{\Delta^{\times n}}$. 
As the composite functor $\delta_n^* (\delta_n)_!$ preserves colimits, this implies that the saturated class generated by $\delta_n^* (\delta_n)_!(T_{\Delta^{\times n}})$ is contained in $T_{\Delta^{\times n}}$, whence $\delta_n^*(T_{\Theta_n}) \subseteq T_{\Delta^{\times n}}$.
\end{proof}

\begin{corollary}\label{cor:RezkequalsCSSn}
	The functor $\CSS(\Theta_n)\to\CSS(\Delta^{\!\times n})$ induced by $\delta_n$ is an equivalence of $\infty$-categories.
\end{corollary}

\section{Epilogue: Model categories of $(\infty,n)$-categories}
\label{sect:models}

We conclude with a brief discussion of model categories of $(\infty,n)$-categories, in which we describe some interactions between our results here and those of Bergner, Lurie, Rezk, and Simpson. We first note that a spate of further corollaries to our main results can be obtained by employing the following.

\begin{construction} Suppose $\A$ a category equipped with a subcategory $w\A$ that contains all the objects of $\A$ (i.e., a \emph{relative category} in the terminology of \cite{BarKan1011}). We call the morphisms of $w\A$ \emph{weak equivalences}. In this situation, one may form the \emph{hammock localization} $\LH\A$ of Dwyer--Kan \cite{MR81h:55019}; this is a simplicial category. One may apply to each mapping space a fibrant replacement $R$ that preserves products (e.g., $\mathrm{Ex}^{\infty}$) to obtain a category enriched in \emph{Kan complexes}, which we shall denote $\LH_f\A$. We may now apply the simplicial nerve construction \cite[1.1.5.5]{HTT} to obtain a $\infty$-category $\mathrm{N}\LH_f\A$, which we shall denote simply by $\NH\A$. We shall call $\NH\A$ the \emph{$\infty$-category underlying} the relative category $\A$.

When $\A$ is a simplicial model category, the simplicial localization $\LH\A$ is equivalent \cite{MR81m:55018} to the full sub-simplicial category $\A^{\circ}$ spanned by the cofirant-fibrant objects. In this case, our $\NH\A$ is equivalent to $\mathrm{N}\A^{\circ}$, as used by Lurie \cite[A.2]{HTT}.
\end{construction}

\begin{remark}\label{rmk:QEsyieldLEs} When $\A$ and $\B$ are model categories, and $F\colon\A\rightleftarrows\B\colon G$ is a Quillen equivalence between them, there is \cite{MR81m:55018} an induced equivalence of hammock localizations $\LH\A\simeq\LH\B$, and thus of underlying $\infty$-categories $\NH\A\simeq\NH\B$.
\end{remark}

\begin{example}
	The $\infty$-category underlying the relative category of $n$-relative categories \cite{BarKan1102} is a theory of $(\infty,n)$-categories.
\end{example}

\begin{definition} Let us call a 
	 model category $\A$ a \emph{model category of $(\infty,n)$-categories} if its underlying $\infty$-category $\NH\A$ is a theory of $(\infty,n)$-categories.
\end{definition}

\begin{example} By \cite{JT}, the Joyal model category of simplicial sets is a model category of $(\infty,1)$-categories. More generally, all of the following model categories are model categories of $(\infty,1)$-categories:
\begin{enumerate}
\item the Joyal model category of quasicategories $\mathrm{QCat}$ \cite{MR0420609,Joyal,HTT},
\item the Rezk model category of complete Segal spaces $\CSS$ \cite{MR1804411},
\item the Bergner model category of simplicial categories \cite{MR81h:55018,MR2276611},
\item the Tamsamani--Hirschowitz--Simpson--Pellissier model categories of Segal Categories \cite{MR984042,math.AG/9807049,math.AT/0308246,MR2341955,simpson:book,MR2321038},
\item the Barwick--Kan model category of relative categories \cite{BarKan1011}.
\end{enumerate}
\end{example}

\begin{example}
	By Th.~\ref{thm:Rezksmodel} and Th.~\ref{thm:CSSnisathy}, both Rezk's model category $\Theta_nSp$
	of complete $\Theta_n$-spaces \cite{Rezk} and the model category of $n$-fold complete Segal spaces \cite{barwick-thesis,G} are model categories of $(\infty,n)$-categories.
\end{example}

\noindent We may now use the construction above to find more examples of theories of $(\infty,n)$-categories.


\begin{example}[{\cite[Proposition 1.5.4]{G}}] Let $\M$ be a left proper simplicial combinatorial model category which is an {absolute distributor} (\cite[Definition 1.5.1]{G}). Then the category $\Fun(\Delta^\op, \M)$, of simplicial objects in $\M$, admits the $\M$-enriched complete Segal model structure $\CSS_{\M}$, which is again left proper, simplicial, combinatorial, and an absolute distributor. If  $\M$ is a model category of of $(\infty,n-1)$-categories, then $\CSS_{\M}$ is a model category of $(\infty,n)$-categories.
\end{example}

\noindent The condition of being an {\em absolute distributor} is needed in order to formulate the correct notion of {\em complete} $\M$-enriched Segal object. We refer the reader to \cite{G} for details, but note that being an absolute distributor is a property of the underlying $\infty$-category of the given model category. In particular it is preserved under any Quillen equivalence. 

\begin{example} 
	Suppose that $\M$ is a model category satisfying the following list of conditions.
\begin{enumerate}[{\hspace{\parindent}(M.}1{)}]
\item The class of weak equivalences of $\M$ are closed under filtered colimits.
\item Every monomorphism of $\M$ is a cofibration.
\item For any object $Y$ of $\M$, the functor $X\mapsto X\times Y$ preserves colimits.
\item For any cofibrations $f\colon X\to Y$ and $f'\colon X'\to Y'$, the \emph{pushout product}
\begin{equation*}
f\Box f'\colon(X\times Y')\cup^{(X\times X')}(Y\times X')\to Y\times Y'
\end{equation*}
is a cofibration that is trivial if either $f$ or $f'$ is.
\item The $\infty$-category $\NH\M$ is a homotopy theory of $(\infty,n-1)$-categories.
\end{enumerate}

\noindent Work of Bergner \cite{MR2321038} and Lurie \cite{G}, combined with \ref{thm:CSSnisathy} above, shows that each of the following is an example of a model category of $(\infty,n)$-categories:
\begin{itemize}
\item the projective or (equivalently) the injective model category \cite[2.2.16, 2.3.1, 2.3.9]{G} $\Seg_{\M}$ of $\M$-enriched preSegal categories, and
\item the model category \cite[A.3.2]{HTT} $\cat_{\M}$ of categories enriched in $\M$.
\end{itemize}
Moreover, following Simpson \cite{simpson:book} the injective (\textsc{aka} Reedy) model category of Segal $(n-1)$-categories \cite{math.AG/9807049,math.AT/0308246,simpson:book} satisfies conditions (M.1-4); indeed, the most difficult of these to verify is (M.4), which Simpson does in \cite[Th.~19.3.2 (using Corollary 17.2.6)]{simpson:book}. 
\end{example}


Thus, for example, the injective and projective model categories of $\Theta_nSp$-enriched Segal categories $\Seg_{\Theta_nSp}$ as well as the model category $\cat_{\Theta_nSp}$ of categories enriched in $\Theta_nSp$ are seen to be model categories of $(\infty,n)$-categories. Indeed very recent work of Bergner and Rezk \cite{1204.2013} discusses these model categories in detail and links them by an explicit chain of Quillen equivalences.

Additionally, we see that the injective model category of Segal $n$-categories is also a model category of $(\infty,n)$-categories, as is the model category of categories enriched in Segal $(n-1)$-categories. 


A partial converse to \ref{rmk:QEsyieldLEs} holds, which allows one to deduce Quillen equivalences between these various model categories.

\begin{lemma}[{\cite[A.3.7.7]{HTT}}] Two combinatorial model categories $\A$ and $\B$ are connected by a chain of Quillen equivalences if and only if $\NH\A$ and $\NH\B$ are equivalent $\infty$-categories.
\end{lemma}


From this it follows that if $\A$ and $\B$ are combinatorial model categories with the property that both $\NH\A$ and $\NH\B$ are theories of $(\infty,n)$-categories, then $\A$ and $\B$ are connected by a chain of Quillen equivalences. This applies to all of the model categories of $(\infty,n)$-categories mentioned above. 

A zig-zag of Quillen equivalences can be a troublesome gadget to work with. It is usually far more informative to have a single direct and explicit Quillen equivalence between competing model categories of $(\infty,n)$-categories. While our techniques do not generally provide such a direct Quillen equivalence, we do offer the following recognition principle. 

\begin{proposition}\label{prop:cellsdetectquillenequiv}
	Let $\A$ and $\B$ be two model categories of $(\infty,n)$-categories and let $L: \A \leftrightarrows \B: R$ be a Quillen adjunction between them. Then $(L,R)$ is a Quillen equivalence if and only if the left derived functor $\NH L: \NH \A \to \NH \B$ preserves the cells up to weak equivalence. 
\end{proposition}

\begin{proof}
	A Quillen equivalence induces an equivalence $\NH L: \NH \A \to \NH \B$ of $\infty$-categories. By
	Lemma \ref{cor:0-TruncatedAreGaunt} and Lemma \ref{lma:CellsPreservedByEquiv} any such equivalence necessarily preserves the cells up to equivalence. Conversely, as the left-derived functor $\NH L: \NH \A \to \NH \B$ preserves (homotopy) colimits and  $\NH \A$ and $\NH \B$ are generated under (homotopy) colimits by the cells (Axiom C.2), it follows that $\NH L$ induces an equivalence of $\infty$-categories. In particular it induces an equivalence of homotopy categories, and hence $(L,R)$ is a Quillen equivalence. 
\end{proof}

In particular the above applies when the cells are fibrant-cofibrant objects of $\A$ and $\B$ which are preserved by $L$ itself.

\begin{example}
	The standard Quillen adjunction (cf. \cite[Lemma 2.3.13]{G}) from Segal $n$-categories to $n$-fold complete Segal spaces  is a Quillen equivalence. 
\end{example}

\begin{example} The functor $\delta_n$ induces a Quillen equivalence between the model category of complete Segal $\Theta_n$-spaces \cite{Rezk} and the model category of $n$-fold complete Segal spaces \cite[1.5.4]{G}. (See also Bergner--Rezk \protect{\cite{bergrezk}}).
\end{example}

A category with a specified subcategory of weak equivalences is a {\em relative category}, and hence gives rise to a homotopy theory. Thus any theory of $(\infty,n)$-categories arising this way may, in principle, be compared using our axioms. We therefore end with the following.

\begin{conjecture}
 The $\infty$-category underlying Verity's $n$-trivial weak complicial sets \cite{MR2450607, MR2342841} is a homotopy theory of $(\infty, n)$-categories. The relative category consisting of Batanin's $\omega$-categories \cite{MR1623672} such that every $k$-cell is an equivalence for $k>n$, together with the class of morphisms which are essentially $k$-surjective for all $k$ is a theory of $(\infty,n)$-categories. 
\end{conjecture}


\bibliographystyle{amsplain}
\bibliography{SpaceOfTheories}

\end{document}